\documentclass[10pt]{iopart}
\usepackage[utf8]{inputenc}
\usepackage[T1]{fontenc}
\usepackage[english]{babel}
\expandafter\let\csname equation*\endcsname\relax
\expandafter\let\csname endequation*\endcsname\relax
\usepackage{amsmath}

\usepackage{booktabs}
\usepackage{array}

\usepackage{amsfonts,amsthm,amssymb}
\usepackage{graphicx,enumerate,bbold}
\usepackage[notrig]{physics}
\usepackage{enumitem} 
\usepackage[font=small]{caption}
\usepackage[font=footnotesize]{subcaption}
\usepackage{txfonts}
\usepackage{hyperref}
\hypersetup{
	colorlinks,
	linkcolor={blue},
	citecolor={blue},
	urlcolor={blue}
}
\usepackage{cleveref}
\usepackage{float}
\usepackage{xparse}
\usepackage{iopams}  

\usepackage{tikz}
\usepackage{placeins}
\usepackage[boxed, ruled, lined, linesnumbered,algosection]{algorithm2e}
\usepackage{algpseudocode}

\theoremstyle{plain}
\newtheorem{theorem}{Theorem}[section]
\newtheorem{lemma}[theorem]{Lemma}
\newtheorem{proposition}[theorem]{Proposition}

\crefname{definition}{Definition}{Definitions}
\crefname{figure}{Figure}{Figures}
\crefname{example}{Example}{Examples}
\crefname{section}{Section}{Sections}
\crefname{chapter}{Chapter}{Chapters}
\crefname{table}{Table}{Tables}
\crefname{algorithm}{Algorithm}{Algorithms}
\crefname{myalgo}{Algorithm}{Algorithms}
\crefname{proc}{Procedure}{Procedures}

\theoremstyle{definition}
\newtheorem{definition}{Definition}[section]

\theoremstyle{remark}
\newtheorem{remark}{Remark}[section]

\newcommand{\bitem}{\begin{itemize}}
\newcommand{\eitem}{\end{itemize}}

\newcommand{\bpm}{\begin{pmatrix}}
\newcommand{\epm}{\end{pmatrix}}
\newcommand{\bsm}{\left(\begin{smallmatrix}}
\newcommand{\esm}{\end{smallmatrix}\right)}

\DeclareMathOperator{\prox}{prox}

\usepackage{customCommands}

\begin{document}
\title{Multi-Channel Potts-Based Reconstruction for Multi-Spectral Computed Tomography
}
\author{Lukas Kiefer$^{1,2}$, Stefania Petra$^1$, Martin Storath$^3$ and Andreas Weinmann$^2$}
\address{$^1$ Mathematical Imaging Group, Heidelberg University, Germany}
\address{$^2$ Department of Mathematics and Natural Sciences, University of Applied
Sciences Darmstadt, Germany}
\address{$^3$ Department of Applied Natural Sciences and Humanities, University
of Applied Sciences W\"urzburg-Schweinfurt, Germany}

\begin{abstract}
We consider reconstructing multi-channel images from measurements performed by 
photon-counting and energy-discriminating detectors
in the setting of multi-spectral X-ray computed tomography (CT).
Our aim is to exploit the strong structural correlation that is known to exist between the channels
of multi-spectral CT images. To that end, we adopt the multi-channel Potts prior
to jointly reconstruct all channels. This nonconvex prior produces piecewise constant solutions 
with strongly correlated channels. In particular, edges are 
strictly enforced to have the same spatial position across channels
which is a benefit over TV-based methods whose channel-couplings are
typically less strict. 
We consider the Potts prior in two frameworks:
(a) in the context of a variational Potts model, and
(b) in a Potts-superiorization approach that perturbs the iterates of
a basic iterative least squares solver.
We identify an alternating direction method
of multipliers (ADMM) approach as well as a
Potts-superiorized conjugate gradient method as 
particularly suitable.
In numerical experiments, we compare the Potts prior based approaches
to existing TV-type approaches on realistically simulated multi-spectral CT data
and obtain improved reconstruction for compound solid bodies.
\\[1em]
Keywords: multi-spectral computed tomography, image reconstruction, structural regularization, multi-channel Potts prior, superiorization, Potts model, piecewise constant Mumford-Shah model,
non-convex optimization, ADMM
\end{abstract}

\section{Introduction}

We consider the \emph{multi-channel} reconstruction problem which arises in multi-spectral X-ray computed tomography (CT).
X-ray imaging entails \emph{polychromatic} X-ray sources, i.e.~emitted photons have a spectrum of energies. 
Conventional energy-integrating CT detectors do not capture
different energies.
However, there are materials 
which may not be distinguished from one another in conventional CT
as their linear attenuation coefficients
(LAC) are nearly equal.
Yet, their LAC's might differ when the whole energy spectrum is
considered. Thus, such materials may be distinguished, when measurements at multiple energy-levels are available.
In this paper we exploit this phenomenon by using
\emph{energy-discriminating} photon counting detectors \cite{shikhaliev2008energy}.

As in conventional CT, due to acquisition noise, sampling effects and modeling effects,
the task of reconstructing volume functions from multi-spectral CT data corresponds to solving
an ill-posed inverse problem. This calls for regularization which is often 
performed by imposing prior structural knowledge on the unknown
solution in the form of a penalty called prior function. 
Typically, the prior function is made part of an energy minimization model of  Tikhonov-type;
see, e.g., \cite{rigie2015joint,kazantsev2018joint}.
In principle, one could apply conventional priors and data models in CT
to reconstruct separately each channel of the \emph{multi-channel} images 
in multi-spectral CT.
However, it is well-known that the channel images in multi-spectral CT are 
strongly correlated:
transitions between materials manifest as edges in the channel images 
and these edges share the \emph{same spatial positions} across the channels.
Therefore, an appropriate prior should
enforce this spatial correlation of the channels and
aid the reconstruction of particularly noise-prone channels on hand of
the less problematic channels.
To this end,  multi-channel extensions of the total variation prior have been
proposed \cite{gao2011multi,rigie2015joint,kazantsev2018joint,toivanen2020joint,KieferP17}.
Typically, these TV-type priors enforce 
inter-channel correlations by imposing higher costs if the edges 
are not aligned over the channels. 
However, these priors are essentially based on penalizing the $\ell_1$-norm of the channel-wise gradients, and do not necessarily ensure that
the resulting boundaries between materials
completely align 
over the channels. 
Another side-effect of penalizing the $\ell_1$-norm of the gradient, is that it leads to contrast reduction, i.e.~in the reconstructed image
the jump heights between regions with constant intensities 
are reduced.

In this paper we propose the use of the \emph{Potts prior} for the reconstruction
of compound solid bodies from multi-spectral CT measurements. 
The Potts prior penalizes the length of the jump set of the image,
i.e. the support of the gradient
as a subset of the domain.
As a result, it produces piecewise constant images with sharply localized edges
such that the total boundary length
of all segments is small.
Further, the Potts prior has no contrast diminishing effect
\cite{storath2013jump}.
The Potts prior is an established nonconvex regularizer for 
joint image reconstruction and partitioning as well as for the reconstruction of approximately 
piecewise constant images given via indirect measurements;
see \Cref{sec:RelatedWork} for a discussion on related work.

Its multi-channel extension --the \emph{multi-channel Potts prior}--
has the appealing property that jumps 
in multiple channels only cost once if they are located 
in the same spatial position.
Therefore, in comparison to TV-type priors, the multi-channel Potts prior produces
sharply localized edges which are \emph{strictly} enforced
to share spatial positions across  channels.
The strong channel-coupling provided by the multi-channel Potts prior 
is especially attractive for multi-spectral CT reconstruction
of compound solid bodies.
We illustrate this by
comparing the multi-channel Potts prior to existing TV-based priors. 
We note that for piecewise smooth images, as they often appear in medical applications, the use of 
piecewise smooth Mumford-Shah models is advisable
\cite{bar2004variational,fornasier2010iterative,hohm2015algorithmic}.
In the context of smooth signals, the result of the Potts prior may be interpreted
as a partitioning of the signal. 

Being nonsmooth and nonconvex, the (multi-channel) Potts prior calls for tailored algorithmic approaches. 
To this end, we consider two approaches based on energy minimization and superiorization, respectively. 
In particular, we apply the multi-channel Potts prior within the variational \emph{Potts model} 
(also known as the piecewise constant Mumford-Shah model)
and within a new 
\emph{Potts-based superiorization approach}.

\subsection{Related Work}
\label{sec:RelatedWork}
We first review work related to multi-spectral CT, followed by work on
the Potts model. We conclude with
work on superiorization.

\emph{The reconstruction problem and priors in multi-spectral CT.}
We start out with related work concerning the reconstruction problem in multi-spectral CT.
Kazantsev et al.~\cite{kazantsev2018joint} consider multi-spectral measurements 
obtained from 
energy-discriminating photon-counting detectors and propose a
prior based on directional total variation (dTV) \cite{bayram2012directional} 
to exploit structural similarities between  channels.
The reference channels for dTV are chosen probabilistically in each iteration on the basis of
signal-to-noise ratios. 
The authors compare their method to other 
existing TV-based approaches by reconstructing the geocore phantom, which we also consider
in our experiments.
Toivanen et al. \cite{toivanen2020joint} study variational methods for the reconstruction 
of three channels obtained by sequential measurements following low dose acquisition protocols.
In \cite{gao2011multi}, the authors propose an approach of compressive sensing type, where 
the multi-channel reconstruction is modeled as the sum of a low-rank and a sparse matrix.
In \cite{rigie2015joint}, the authors use total nuclear variation to regularize the
reconstruction and enforce structural correlation between channels.
Furthermore, channel-coupling regularizers based on tensor nuclear norms  
are proposed in \cite{SemerciHKM13}.

Another interesting problem in multi-spectral CT --apart from the reconstruction problem-- is
the material decomposition problem.
It describes the task of determining the pixel-wise material composition 
after the multi-channel image has been reconstructed (typically with channel-wise
filtered backprojection).
Variational approaches to the material decomposition problem were proposed in  \cite{Long_2014,ducros2017regularization,ding2018image}.

\emph{The Potts model.}
The Potts prior is often used in the context of energy minimization methods. 
This class of methods model the result as the minimizer of an energy function that
comprises a data fidelity term and a regularizing term.
Thereby, the result is close to the available data and
it is regularized by enforcing prior knowledge.
The Potts prior together with a data fidelity term 
yields the Potts model.
We start with work on the Potts model for directly sensed image data. 
Originally named after R. Potts due to his work on statistical mechanics \cite{potts1952some}, 
Geman and Geman \cite{geman1984stochastic} adopted the Potts model for 
the edge-preserving reconstruction of piecewise constant images and proposed an algorithmic approach
based on simulated annealing.
The problem was first considered from the viewpoint of variational calculus by Mumford and Shah \cite{mumford1989optimal}.
Shortly afterwards, Ambrosio and Tortorelli proposed approximations by smooth functionals 
which are frequently adopted for algorithmic approaches \cite{ambrosio1990approximation}.
Other popular algorithmic approaches use active contours 
\cite{chan2001active}, graph cuts
\cite{boykov2001fast}, convex relaxations 
\cite{pock2009convex,chambolle2012convex} and the ADMM \cite{storath2014fast}. 
We proceed with related work on the Potts model for inverse problems.
Existence of minimizers in the inverse setting has been investigated 
in  \cite{fornasier2013existence,fornasier2010iterative,ramlau2010regularization,jiang2014regularizing,storath2013jump}.
Ramlau and Ring proved that, under mild assumptions on the imaging operator 
and possible function values,
the Potts model is a regularizer in the sense of inverse problems 
\cite{ramlau2010regularization}. 
For deconvolution problems, level-set based active contour methods 
\cite{kim2002curve} and Ambrosio-Tortorelli type approaches
have been proposed \cite{bar2004variational}.
Concerning Radon measurements, Ramlau and Ring propose a method based on active contours \cite{ramlau2007mumford} and 
a similar method for SPECT/CT data is proposed in \cite{klann2013regularization,klann2011mumford}.
We note that methods based on active contours have the disadvantage that they require a relatively
good guess for the initialization and for the expected number of gray values in the image.
In \cite{storath2013jump}, the problem for general imaging operators  is approached
by an iterative graph cut strategy. 
We remark that graph cut approaches need to work on a discretized 
codomain. Therefore, they either need a good guess on the values of the unknown image or a very fine 
discretization of the codomain which can become very expensive.
Another class of approaches
is based on solving sequences of surrogate problems 
\cite{fornasier2010iterative,weinmann2015iterative}.
Finally, in \cite{storath2015joint} the authors propose an ADMM approach. 
Inspired by this, we adapt ADMM to the multi-spectral CT reconstruction problem, as detailed in
\Cref{sec:PottsModel}.
A thorough treatment of optimization strategies based on the ADMM can
be found in \cite{boyd2011distributed}.

\emph{Superiorization.}
The \emph{superiorization methodology} \cite{davidi2009perturbation,herman2012superiorization} 
is an alternative to energy minimization approaches. 
Algorithms emerging from energy minimization approaches
typically alternate between a data fidelity step and a regularization step.
Superiorization methods  also alternate between these two steps as follows: 
first, one takes an iterative \emph{basic algorithm} that is typically
'feasibility-seeking' (i.e.~finding some point that is compatible with a family of constraints or an 
image that complies with the measured projection data) 
and then investigates its \emph{perturbation resilience} (i.e.~whether
its termination is still guaranteed when the iterates are appropriately perturbed).
Secondly, one defines a target function which expresses exogenous or prior knowledge of the domain in question. In each iteration, the iterate is perturbed towards a non-ascending direction of this target function. It suffices that this perturbation merely improves on the target function without necessarily minimizing it.
The overall process leads to a solution of the basic problem that is superior  to the one that would have been reached without the interlaced perturbations
(with respect to the target function)
without paying a high computational price.
The terms ``superiorization'' and ``perturbation resilience'' first appeared in \cite{davidi2009perturbation} and early developments  were presented in  \cite{butnariu,brz06,brz08}.
Important notions concerning superiorization were  introduced in \cite{herman2012superiorization}.
A condensed introduction to the superiorization methodology and examples can be found in
\cite{censor2019derivative}. 
Superiorization is an active research field \cite{Sup-Special-Issue-2017} 
and current research can be found in the continually updated bibliography in \cite{censor2015superiorization}. Theoretical work is typically concerned with
perturbation resilience which lies at the heart of the superiorization methodology \cite{brz06,Bargetz2018}. Only few 
works \cite{Censor-Levi:2019,cz3-2015} have addressed  the 'guarantee problem', that is, a mathematical guarantee that the overall process of the superiorized version
of the basic algorithm will not only retain its feasibility-seeking nature but also preserve globally the target
function reductions.
Similarities between superiorization methods and optimization methods have been studied in
\cite{byrne2019simulations,censor2020superiorization}.
In \cite{helou2018superiorization,zibetti2018super}, 
the (preconditioned) conjugate gradient algorithm is perturbed by non-ascending directions w.r.t.\,the total variation prior for (single-channel) tomographic image reconstruction problems.
Therein, the authors also show theoretical perturbation resilience for the conjugate gradient (CG).
In this work, we make extensive use of this property and superiorize CG  by a non-convex and non-continuous target function.

\subsection{Contribution}
We investigate reconstruction schemes for multi-spectral
CT problems based on the multi-channel Potts prior 
which enforces strong spatial correlation between the channels.
We investigate two approaches.
In the first approach, called \emph{Potts ADMM}, we apply the Potts prior within the variational Potts model
which leads to a non-convex optimization problem. 
We adapt the ADMM approach of \cite{storath2015joint}
to the multi-channel reconstruction of multi-spectral CT data.	
In the second approach, called \emph{Potts S-CG}, we propose a new superiorization approach
that perturbs the iterates of the conjugate gradient (CG) method
with a block-wise Potts prior as target function towards Potts regularized solutions.
As the Potts prior is neither continuous nor convex, 
existing derivative-based methods for perturbations with TV \cite{zibetti2018super,censor2020superiorization}
cannot be employed directly.
Consequently, we employ a derivative-free perturbation approach
to superiorize  CG w.r.t.~the 
Potts prior.
We identify Potts ADMM and Potts S-CG
as suitable choices for the Potts model and Potts superiorization by
comparing
the ADMM to a penalty method for the Potts model, and the superiorized CG method to a Landweber iteration.
In numerical experiments, we illustrate the benefits of 
the multi-channel Potts prior
in the reconstruction of multi-channel images 
from multi-spectral CT data. 
Specifically we compare the proposed Potts ADMM and  Potts S-CG  to  existing TV-type approaches by applying them to realistically simulated multi-spectral CT data.
To the authors' knowledge the Potts prior has
neither been considered in the context of multi-spectral CT nor in the context of the superiorization methodology yet.

\subsection{Organization of the Paper}
In \Cref{sec:MultispectralCTimaging}, we briefly explain the measurement process and
the forward model in multi-spectral CT.
\Cref{sec:ThePottsPrior} explains the (multi-channel) Potts prior and its benefits in 
multi-channel image reconstruction.
In \Cref{sec:PottsModel}, we briefly explain the Potts model and an algorithmic approach
based on  ADMM.
In \Cref{sec:PottsSuperiorization}, we derive new Potts-based superiorization approaches, which employ
the Potts prior to perturb the iterates of the CG method.  
In \Cref{sec:ComparisonADMMandPottsCG}, we compare the Potts ADMM approach and 
the proposed Potts-based superiorization approach to 
a penalty method and a Landweber superiorization approach.
In \Cref{sec:Experiments}, we apply both methods to image reconstruction from simulated multi-spectral CT data and compare 
the results to existing methods.
In \Cref{sec:Conclusion}, we draw conclusions.

\section{Measurements and Reconstructions in Multi-Spectral CT}
\label{sec:MultispectralCTimaging}
\subsection{The Forward Problem in Multi-Spectral CT}
We consider the  spectral version of Beer's law 
\begin{equation}\label{eq:BeersLawMulti}
I_1(\epsilon)=I_0(\epsilon)\exp\left(\int_{\mathcal{L}}- \hat{u}(z, \epsilon)\, dz \right),
\end{equation}
where $\epsilon$ denotes the energy, $I_1$ is the spectrum of the X-ray beam
incident on the detector, $z\in\Omega$ denotes the spatial position.
Further, $\hat{u}$ holds the \emph{energy dependent}
linear attenuation coefficients 
that need to be reconstructed and $I_0$ is the \emph{energy dependent} intensity
flux of the X-ray source corresponding to ray $\mathcal{L}$ from the source to a given
detector element.

In the following, we consider 
$M$ detectors and assume that the measurements 
\eqref{eq:BeersLawMulti} are taken from $p$ equidistant angles.
Thus, a total number of $m=pM$ discrete measurements are available.
Towards a discrete model, we discretize the (unknown) function $\hat{u}$ 
on the continuous domain $\Omega$
which results in a function $u$ on an $n\times n$ pixel grid $\Omega'$.
It is given by
\begin{equation}
u(j,\epsilon)=\sum_{z\in\Omega} \chi_j(z) \hat{u}(z,\epsilon),
\end{equation}
where $\chi_j$ denotes the characteristic function
corresponding to the pixel $j\in\Omega'$.
From \eqref{eq:BeersLawMulti} and for pairs of source-detector positions indexed by $i$, we now 
obtain
\begin{equation}\label{eq:BeerMultiDiscrete}
I_{1,i}(\epsilon)=I_{0,i}(\epsilon)
\exp \Big(- \sum_{j\in{\Omega'}} A_{ij}\,
u(j,\epsilon) \Big),\quad i=1,\ldots,m,
\end{equation}
where we denote by
\begin{equation}
A_{ij}:=\int_{\mathcal{L}_i} \chi_j(z)\, dz, \quad i=1,\ldots,m, \quad j\in\Omega',
\end{equation}
the intersection length of ray $i$ with pixel $j$. (In particular, 
$A_{ij}=0$ if the ray $i$ does not intersect pixel $j$.)
We assume that each photon-counting detector 
acquires the measurements on an 
interval of energies $[\epsilon_c,\epsilon_{c+1})$.
Hence, the detectors provide the measurements
\begin{equation}\label{eq:DiscreteMeasurements1}
Y_{i,c}=\int_{\epsilon_c}^{\epsilon_{c+1}} I_{0,i}(\epsilon)
\exp \Big(- \sum_{j\in{\Omega'}} A_{ij}
u(j,\epsilon) \Big)\, d\epsilon,\quad i=1,\ldots,m,\, c=1,\ldots,C.
\end{equation}
After discretizing the energy spectrum with step-size $\delta>0$, we obtain
\begin{equation}\label{eq:DiscreteMeasurements}
Y_{i,c}\approx\sum_{\epsilon=\epsilon_c,\epsilon_c+\delta,\ldots,\epsilon_{c+1}-\delta} I_{0,i}(\epsilon)
\exp \Big(- \sum_{j\in{\Omega'}} A_{ij}\,
u(j,\epsilon) \Big),\quad i=1,\ldots,m,\, c=1,\ldots,C.
\end{equation}
The following simplified measurement model becomes increasingly accurate
for fine energy resolutions, i.e., if $\delta$ is small,  \cite{kazantsev2018joint}:
\begin{equation}\label{eq:DiscreteMeasurementsSimplified}
Y_{i,c}\approx I_{0,i}(\epsilon_c)
\exp \Big(- \sum_{j\in{\Omega'}} A_{ij}\,
u(j,\epsilon_c) \Big),\quad i=1,\ldots,m,\, c=1,\ldots,C.
\end{equation}
After taking the logarithm, \eqref{eq:DiscreteMeasurementsSimplified}
corresponds to 
\begin{equation}
f_{i,c} \approx
\sum_{j\in{\Omega'}} A_{ij}\,
u(j,\epsilon_c) ,\quad i=1,\ldots,m,\, c=1,\ldots,C,
\end{equation}
where we denote $f_{i,c}=-\log \Big(\frac{Y_{i,c}}{I_{0,i}(\epsilon_c)}\Big)$.
Towards a compact notation, we let 
$f_c\in \R^{m}$ hold the $f_{i,c}$ for all rays $i=1,\ldots,m=pM$ and define the linear operator
$A = (A_{i,j})_{i=1,\ldots,m,j\in\Omega'} : \R^{n\times n}\to \R^m$,
which holds the \emph{ray incidence geometry} of the measurement setup.
Hence, we obtain the linear measurement model 
\begin{equation}\label{eq:linearMeasurementModel}
f_c \approx A u_c \quad 
\end{equation}
for the channels $c = 1,\ldots,C$.

It is common to assume 
acquisition noise for photon counting detectors, that is,
the measurements are Poisson distributed. 
Thus, in a setup incorporating noise,
the corresponding variant of \eqref{eq:DiscreteMeasurementsSimplified}
is given by 
\begin{equation}\label{eq:Poisson}
Y_{i,c}\sim \mathrm{Poiss}\Big\{I_{0,i}(\epsilon_c)\exp \Big(- \sum_{j\in{\Omega'}} A_{ij}\,
u(j,\epsilon_c) \Big)\Big\}.
\end{equation}
As a result, note that $f_c$ in \eqref{eq:linearMeasurementModel} is not deterministic due to \eqref{eq:Poisson}.
The multi-spectral CT reconstruction problem corresponds to determine
$u_1,\ldots,u_C$ from the measurements \eqref{eq:linearMeasurementModel}.
\begin{remark}
	We note that the made assumptions guarantee a measurement model
	which is linear.
	This model is 
    frequently used in the literature
     \cite{gonzales2010full,gao2011multi,SemerciHKM13,rigie2015joint,kazantsev2018joint,toivanen2020joint}.
   	We briefly recall its known limitations.
	\begin{enumerate}[label=(\roman*)]
		\item For coarse energy resolutions, the step between \eqref{eq:DiscreteMeasurements} and \eqref{eq:DiscreteMeasurementsSimplified} becomes increasingly inaccurate.
		\item We made the implicit assumption that the
		energy-dependent detection answers in \eqref{eq:DiscreteMeasurements1} are represented
		by characteristic functions $\chi_{[\epsilon_c,\epsilon_{c+1})}$ of the 
		pairwise disjoint intervals	$[\epsilon_c,\epsilon_{c+1})$.
		More realistic detector answer functions would not be constant on the corresponding energy 
		spectrum nor would they have disjoint supports.
	\end{enumerate}
\end{remark}

\subsection{Reconstruction}
In multi-spectral CT, it is common to use model-based reconstruction approaches
as a direct inversion of \eqref{eq:linearMeasurementModel} is not feasible due to 
the present acquisition noise \eqref{eq:Poisson}, sampling effects, and 
necessary simplifications in the modeling process.
Instead, the reconstruction of $u$ is usually modeled as the minimizer of 
an energy function.
Typically, such an energy function corresponds to the
(weighted) sum of a data term $\mathcal{D}$ and a regularizer $\mathcal{R}$, i.e.,
the reconstruction is modeled as the solution of a minimization problem
of the form
\begin{equation}\label{eq:ReconstructionGeneric}
\argmin_{u\in\R^{n\times n \times C}}
\sum_{c=1}^{C}\mathcal{D}(Au_c,f_c) + \gamma\,\mathcal{R}(u).
\end{equation}
The data term $\mathcal{D}$ promotes closeness to the data $f$ and $\mathcal{R}$
imposes regularity on $u$ by penalizing deviations from a priori 
fixed assumptions on $u$.
The parameter $\gamma>0$ balances the two terms.

\paragraph{Data term.}
Concerning the data term $\mathcal{D}$, a natural choice would be 
based on the log-likelihood associated with the Poisson distribution \cite{shepp1982maximum}.
However, for non-trivial imaging operators $A\neq \id$ as in \eqref{eq:Poisson}, 
employing such a data term becomes computationally expensive.
Thus, it is common for CT reconstruction problems to employ 
the penalized weighted least squares model (PWLS) instead \cite{ding2018image,kazantsev2018joint,SemerciHKM13}.
The PWLS is a quadratic approximation to the log-likelihood function \cite{sauer1993local},
so that it is computational more tractable for 
non-trivial imaging operators.
The PWLS is given by
\begin{equation}
\label{eq:DataTerm}
\mathcal{D}(Au_c,f_c) = \| Au_c -f_c\|_{W_c}^2 = 
\| W_c^{\frac{1}{2}}Au_c - W_c^{\frac{1}{2}}f_c\|^2,
\end{equation}
where each $W_c\in\R^{m\times m}$ is a diagonal matrix which weighs the measurements
within the respective channel.
More precisely, the entries of $W_c$ are chosen as
the number of detected photons in the $c$-th energy bin \eqref{eq:DiscreteMeasurementsSimplified}.
(Please note that $\|\cdot\|$ corresponds to the Frobenius norm, i.e.,
$\| u \|^2 = \sum_{i,j,c} u_{ijc}^2$.)

\paragraph{Regularizing term.}
In general, an appropriate regularizing term $\mathcal{R}$ in \eqref{eq:ReconstructionGeneric} should
enforce prior knowledge on the unknown result $u$.
For example, compound solid bodies are (approximately) piecewise homogeneous so that 
an adequate regularizer may enforce piecewise constancy on $u$. 
We note that it is generally assumed that the channels $u_1,\ldots,u_C$ of multi-spectral CT
images are strongly correlated
\cite{kazantsev2018joint,toivanen2020joint,ding2018image}.
In particular, it was observed that
the edges in the channels are spatially correlated, i.e., they are located
at the same positions across the channels.
As a result, a regularizer $\mathcal{R}$ in \eqref{eq:ReconstructionGeneric} should
enforce prior structural knowledge on each channel
\emph{as well as} strong correlation between the channels.

In this paper, we propose to use the (multi-channel) 
Potts prior for $\mathcal{R}$, which combines structural knowledge and strong channel correlation 
as explained in the next section.	
\section{The Potts Prior for Multi-Spectral CT Reconstruction}
\label{sec:ThePottsPrior}
The Potts prior is an established 
regularizer for the reconstruction from indirectly given data
\cite{ramlau2007mumford,ramlau2010regularization,
	klann2011mumford,klann2011mumfordElectron,storath2015joint,weinmann2015iterative,kiefer2018iterative}.
In particular,
the reconstruction is modeled as the minimizer of an energy functional 
which consists of the Potts prior and some data term.
The thereby obtained model is called the \emph{Potts model} 
(sometimes also referred to as piecewise constant Mumford-Shah model).
We point out that, under mild assumptions,
the Potts model is a regularizer in the sense of inverse problems 
\cite{ramlau2010regularization}.

We start out by describing the Potts prior for single-channel 
image recovery. Subsequently, we discuss its extension to multi-channel images
and explain how this multi-channel Potts prior enforces strong correlation
between channels. This property is especially beneficial when dealing
with multi-spectral CT data since the edges in different energy channels
are generally 
assumed to be strongly spatially correlated.

\subsection{The Single-Channel Potts Prior}
\label{sec:PottsPrior}
In the following, we let $\hat{u}:\Omega\to\R$ be an image defined on the
rectangular domain $\Omega\subset\R^2$.
We denote the \emph{Potts prior} by the symbol $\| \nabla \hat{u} \|_0$.
Formally, $\| \nabla \hat{u}\|_0$ measures the length of the support of the (distributional)
gradient of $\hat{u}$, i.e.,
\begin{equation}
\label{eq:PottsPriorSingleChannel}
\|\nabla \hat{u} \|_0 = \mathrm{length}\,\Big(\big\{ x\in \Omega: \nabla \hat{u}(x) \neq 0\big\}\Big),
\end{equation}
where $\rm length$ is understood w.r.t.\,the one-dimensional Hausdorff measure.
To fix ideas, for a piecewise constant image $\hat{u}$ 
with sufficiently regular discontinuity set, the Potts prior \eqref{eq:PottsPriorSingleChannel}
measures the total arc length of this discontinuity set.
An important observation is that
$\|\nabla \hat{u}\|_0$ is finite if and only if the image $\hat{u}$ is \emph{piecewise constant},
so that the results exhibit crisp boundaries which are sharply localized.

Frequently, the Potts prior is replaced by 
the total variation prior which is given by $\|\nabla \hat{u} \|_1$ \cite{rudin1992nonlinear}.
Using the TV prior typically corresponds to solving convex problems, while
using the Potts prior typically leads to nonconvex problems.
Hence, the TV prior is theoretically and algorithmically easier to access 
than the Potts prior.
However, the results of 
TV often lack sharply localized boundaries in situations with limited-data 
\cite{chartrand2007exact,chartrand2009fast} and it can lead to contrast reduction
\cite{storath2013jump}.
Furthermore, as a piecewise constant image directly corresponds to a partitioning,
the Potts prior may also be used for joint reconstruction and partitioning.
Joint approaches typically yield better results than executing these steps successively \cite{klann2011mumford,ramlau2007mumford,ramlau2010regularization,storath2015joint}.

In practice, data are discrete and we
consider images $u:\Omega'\to\R$ on a lattice
$\Omega' = \{1,\ldots,n\}\times \{1,\ldots,n\}$.
A common discretization of the 
gradient in \eqref{eq:PottsPriorSingleChannel}
is given in terms of finite differences 
\cite{chambolle1995image,chambolle1999finite,storath2014fast} via
\begin{equation}\label{eq:discreteL0}
\| \nabla \hat{u}\|_0 \approx \sum_{s=1}^{S}\omega_s \| \nabla_{d_s} u \|_0,
\end{equation}
where $\|\nabla_{d_s}u \|_0 = |\{x\in\Omega: x+d_s\in\Omega,\,u(x)\neq u(x+d_s)\}|$
counts the number of intensity changes of $u$ in the direction $d_s\in\Z^2$ weighted
by $\omega_s > 0$. Consequently, the $\ell_0$-terms
on the righthand side of \eqref{eq:discreteL0} promote
sparsity in the number of intensity changes of $u$.
The simplest choice of directions are the unit vectors $d_1=(1,0)^T,\,d_2=(0,1)^T$
together with unit weights $\omega_{1,2} = 1$. 
This neighborhood system, however, can lead
to results which suffer from geometrical staircasing 
\cite{storath2014fast}: artificial block artifacts may be introduced
to account for edges in the image to be reconstructed which are not parallel to the coordinate axes.
This can be improved upon by including
the diagonal directions $d_3 = (1,1)^T,\,d_4=(1,-1)^T$ and using the weights 
$\omega_{1,2} = \sqrt{2}-1, \omega_{3,4}=1-\frac{\sqrt{2}}{2}$;
for details see \cite{chambolle1999finite,Boykov2003,storath2014fast}.
Because of the improved quality, 
we will use this system of directions and weights throughout the rest of the paper
if not stated otherwise.

\subsection{The Multi-Channel Potts Prior}
\label{sec:MultichannelPottsPrior}
In the following, we give an extended description of 
the multi-channel version of the Potts prior \eqref{eq:PottsPriorSingleChannel} 
which has been previously used for regularizing 
given color and multispectral images; see, e.e., \cite{storath2014fast}.
To this end, we start with a multi-channel
image $\hat{u}:\Omega\to \R^C$ with codimension $C\in\N$. The $c$-th component
function of $\hat{u}$ is denoted by $\hat{u}_c$.
Further, we denote
the (distributional) Jacobian of $\hat{u}$ analogously to the single-channel case by
$\nabla \hat{u}$. Similarly to the Potts prior for $C=1$ \eqref{eq:PottsPriorSingleChannel},
the Potts prior for multi-channel images is given by the length
of the support of the (distributional) Jacobian, i.e.,
\begin{equation}
\label{eq:PottsPriorMultiChannel}
\|\nabla \hat{u} \|_0 = \mathrm{length}\,\Big(\big\{ x\in \Omega: \nabla \hat{u}(x) \neq 0\big\}\Big),
\end{equation}
where the length is understood as in \eqref{eq:PottsPriorSingleChannel}.
The discrete counterpart of \eqref{eq:PottsPriorMultiChannel} 
for the function $u:\Omega\to\R^C$ on the lattice $\Omega = \{1,\ldots,n\}\times \{1,\ldots,n\}$ is 
now given by 
\begin{equation}\label{eq:PottsPriorDiscreteMultichannel}
\| \nabla \hat{u} \|_0 \approx \sum_{s=1}^{S} \omega_s \| \nabla_{d_s} u \|_0,
\end{equation}
where the righthand side counts the number of finite difference vectors
which are not fully zero, i.e.,
\begin{equation}\label{eq:discreteL0Multichannel}
\|\nabla_{d_s}u \|_0 = \big|\Big\{x\in\Omega: x+d_s\in\Omega,\,u_c(x)\neq u_c(x+d_s) 
\text{ for at least one } c\in \{1,\ldots,C\} \Big\}\big|.
\end{equation}
The integer vectors $d_s\in\Z^2$ and weights $\omega_s>0$ 
have the same role as in the single-channel case \eqref{eq:discreteL0}.
In respect of \eqref{eq:discreteL0Multichannel}, the costs
of the multi-channel Potts prior for
a jump between $x$ and $x+d_s$ in all channels are the same
as for opening a jump in a single channel only.
On the one hand, this enforces potential jumps at close positions to be spatially aligned across channels
in view of the lower costs compared with multiple jumps in different channels
(\emph{jump alignment}).
On the other hand, jumps are always introduced for all channels simultaneously
as a higher closeness to the data can be ensured
(\emph{jump consistency}).
Thus, the edges of multi-channel images
regularized by the multi-channel Potts prior 
are spatially aligned across all channels, or,
in other words, 
the spatial locations of the channel-wise edges of the produced image are 
enforced to be completely correlated.
In Figure \ref{fig:MultichannelPotts},
we illustrate the jump alignment and the jump consistency
provided by the multi-channel Potts prior 
by comparing it to the channel-wise Potts prior 
for dual-channel one-dimensional data.
\begin{figure}[t]
	\centering
	\captionsetup[subfigure]{justification=centering}
	\def\figwidth{0.38\textwidth}
	\def\hs{\hspace{3em}}	
	\def\vs{\\[0.25em]}
	\begin{subfigure}{\figwidth}
		\caption*{\textbf{Jump alignment}}
		\includegraphics[width=\textwidth]{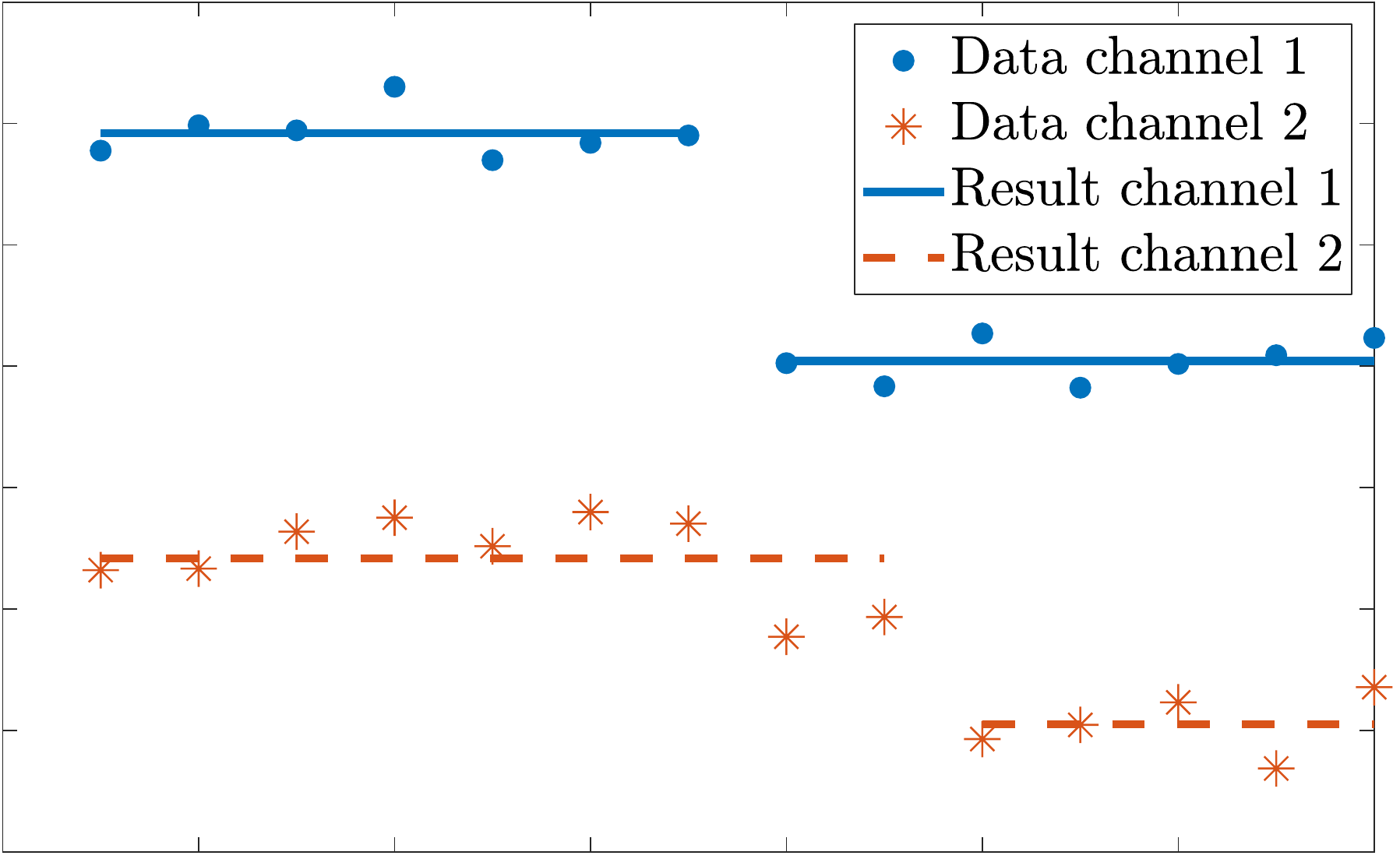}
		\subcaption{Channel-wise Potts prior\label{subfig:alignJumps_channelwise}}
	\end{subfigure}\hs
	\begin{subfigure}{\figwidth}
		\caption*{\textbf{Jump consistency}}
		\includegraphics[width=\textwidth]{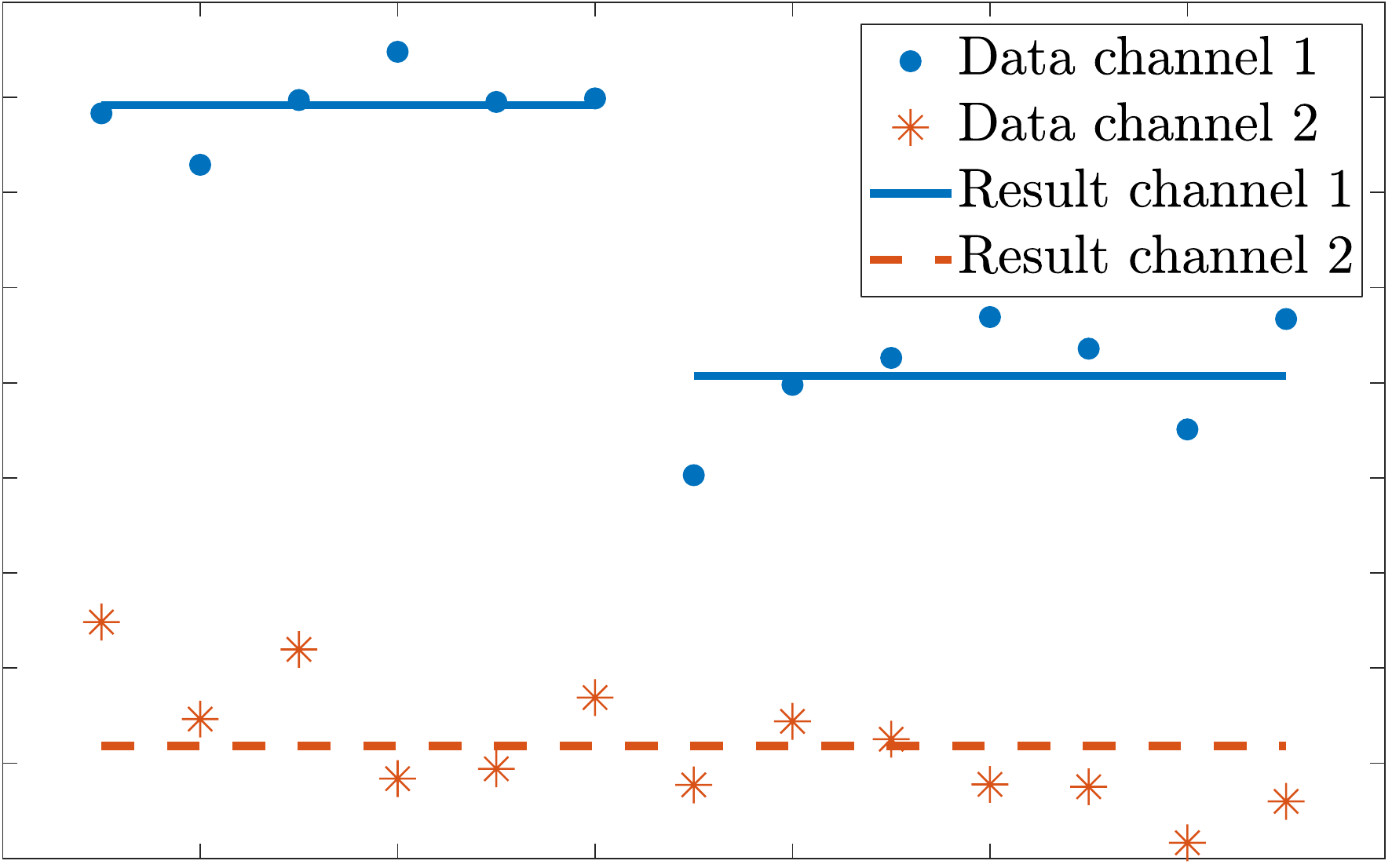}
		\subcaption{Channel-wise Potts prior\label{subfig:forceJumps_channelwise}}
	\end{subfigure}\vs	
	\begin{subfigure}{\figwidth}
		\includegraphics[width=\textwidth]{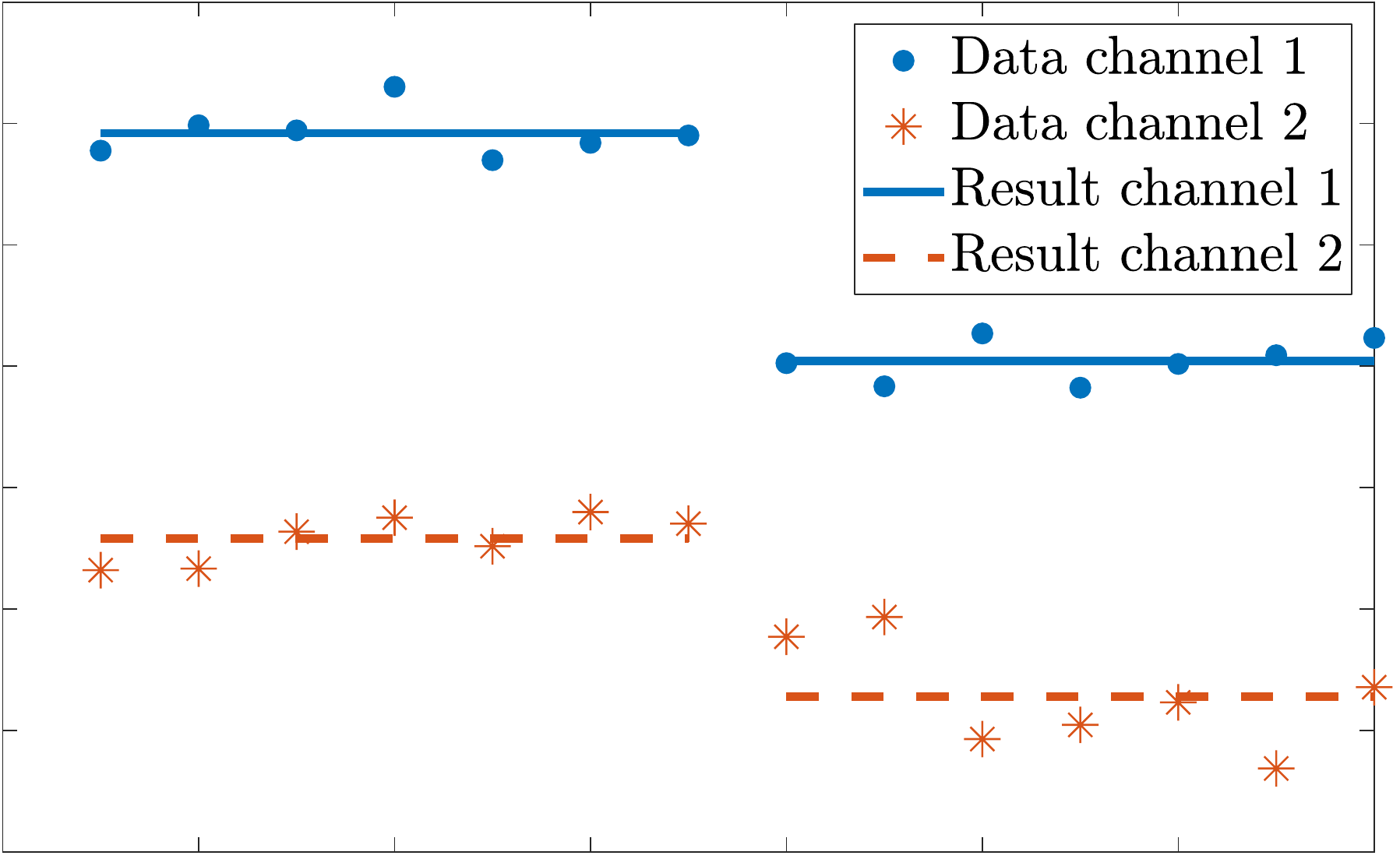}
		\caption{Multi-channel Potts prior\label{subfig:alignJumps_multichannel}}
	\end{subfigure}\hs
	\begin{subfigure}{\figwidth}
		\includegraphics[width=\textwidth]{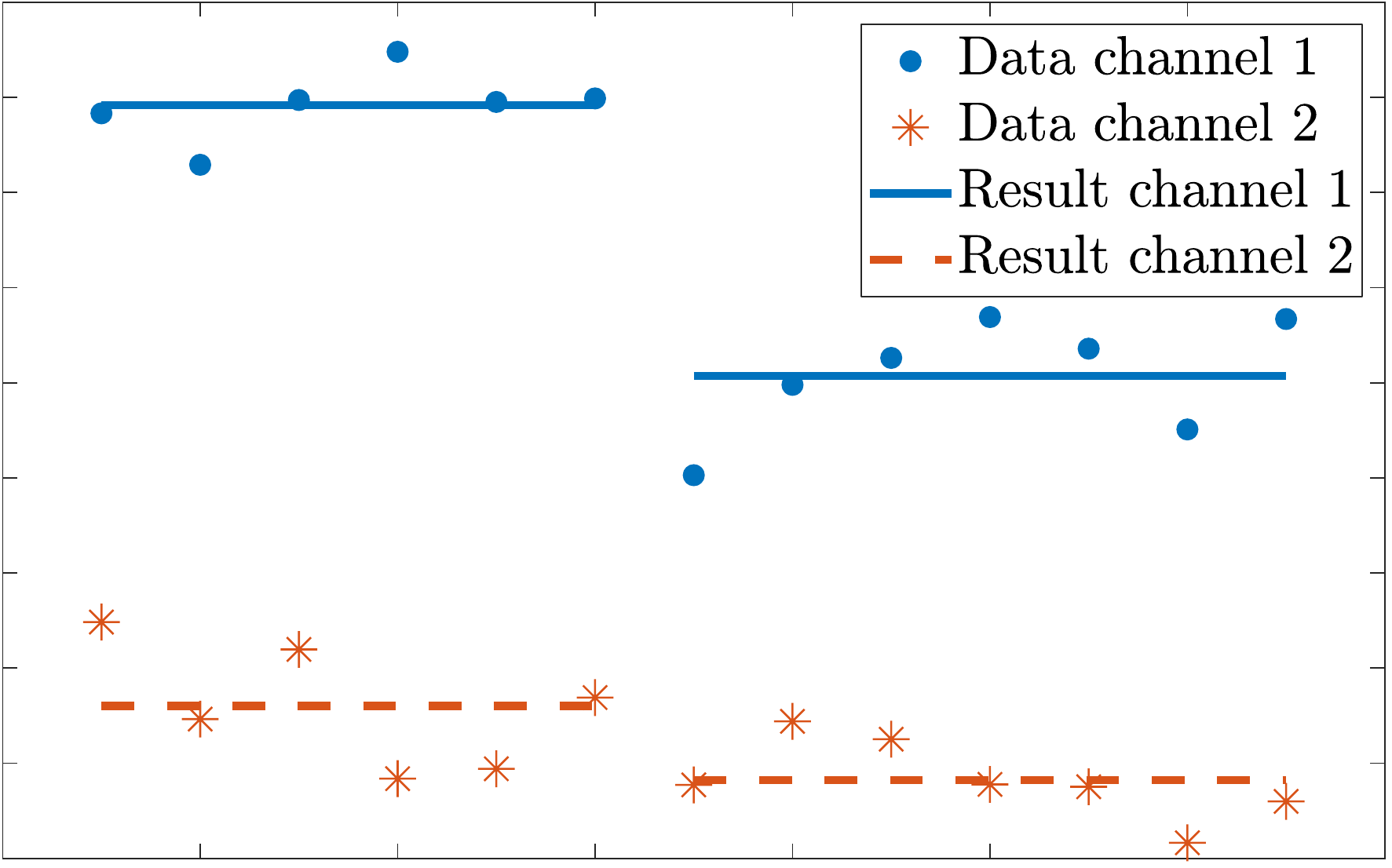}
		\subcaption{Multi-channel Potts prior\label{subfig:forceJumps_multichannel}}
	\end{subfigure}	
	\caption{\label{fig:MultichannelPotts}Channel-coupling with the multi-channel Potts prior.
		\emph{First column}:
		(\subref{subfig:alignJumps_channelwise})
		the channel-wise Potts prior tends to produce jumps at shifted positions 
		in the channels.
		(\subref{subfig:alignJumps_multichannel}) 
		The multi-channel Potts prior enforces the jumps to be spatially aligned across channels
		(called the \emph{jump alignment} property 
		of the multi-channel Potts prior).
		\emph{Second column}:
		(\subref{subfig:forceJumps_channelwise})
		the channel-wise Potts prior may open jumps in a single channel only. 
		(\subref{subfig:forceJumps_multichannel})
		The jumps produced by the multi-channel Potts prior are typically 
		present in all channels. Hence, the multi-channel Potts prior provides \emph{jump consistency}.} 
\end{figure}

In the following two sections, we work out two approaches to the numerical reconstruction of 
piecewise constant multi-channel volume functions by means of the (multi-channel) 
Potts prior.
Firstly, we describe the variational Potts model, which corresponds to 
a minimization problem of the form
\eqref{eq:ReconstructionGeneric} with the Potts prior as the regularizer $\mathcal{R}$.
We approach the Potts model numerically with the ADMM scheme of \cite{storath2015joint}.
Secondly, we derive a new approach based on the Potts prior, which
is \emph{not} of the form \eqref{eq:ReconstructionGeneric}, i.e., it is not
based on an energy minimization model.
Rather, we use the Potts prior to \emph{superiorize} 
an existing iterative solver for (weighted) least squares problems --
the conjugate gradient algorithm (CG).
In particular, the Potts prior is viewed as a target function
which \emph{perturbs} the iterations of CG towards Potts superiorized results.

\section{Potts-Based ADMM}
\label{sec:PottsModel}
The \emph{Potts model}
for the (linear) forward operator $A:\R^{n\times n} \to \R^{m}$ and multi-spectral data $f\in\R^{m\times C}$ 
is given by the following minimization problem
\begin{equation}\label{eq:PottsModelMultichannel}
u^\ast = \argmin_{u:\Omega\to\R^C}
\sum_{c=1}^{C}\mathcal{D}(Au_c\,,f_c)
+\gamma \sum_{s=1}^{S}\omega_s\| \nabla_{d_s} u \|_0.
\end{equation}
Here, $\mathcal{D}$ stands for the PWLS data fidelity term \eqref{eq:DataTerm}.
The nonnegative parameter $\gamma$ balances the data fidelity and the 
Potts prior. It is worth considering the limit cases of $\gamma$: 
for $\gamma \to 0$ the minimizer of \eqref{eq:PottsModelMultichannel} 
has minimal data fidelity and no regularization takes places.
For $\gamma \to \infty$ the minimizer of \eqref{eq:PottsModelMultichannel} will be a constant function
(with minimal data fidelity).
It is well-known that the Potts problem \eqref{eq:PottsModelMultichannel} is non-smooth, non-convex 
and NP-hard \cite{veksler1999efficient,boykov2001fast,weinmann2015iterative}.

\paragraph{Algorithmic approach with the ADMM.}
To obtain good (approximate) solutions of the Potts model
\eqref{eq:PottsModelMultichannel}, 
we follow the approach first presented in \cite{storath2014fast} and extended to inverse problems in
\cite{storath2015joint}, which
is based on the ADMM. 
The ADMM has been successfully applied to non-convex problems 
\cite{xu2016empirical,chartrand2013nonconvex,hohm2015algorithmic,yang2017alternating,wen2012alternating}. 
Contributions from the analytical side are
\cite{wang2019global,hong2016convergence,li2015global,wang2018convergence,wang2014convergence}.
We note that the ADMM performs particularly well in the context of imaging problems 
\cite{goldstein2009split,ng2010solving,steidl2010removing,boyd2011distributed,yang2011alternating}.
We briefly recall the scheme for the reader's convenience and 
provide the necessary adaptations to 
multi-spectral CT.

As a first step, we recast \eqref{eq:PottsModelMultichannel} as a constrained problem.
To this end, we introduce splitting variables 
$u_1,\ldots,u_S$ (each
corresponding to a direction $d_s$ of the discrete gradient)
and the variable $v$ (corresponding to the data term) under the restriction that 
these new variables are equal to not alter the original problem \eqref{eq:PottsModelMultichannel}. 
Hence, we consider the constrained problem
\begin{equation}\label{eq:PottsSplitted}
\begin{split}
&\min_{v,u_1,\ldots,u_S}
\sum_{c=1}^{C}\mathcal{D}(Av_c,f_c) +\sum_{s=1}^{S}\gamma\omega_s\| \nabla_{d_s} u_s \|_0\\
&\text{subject to}\quad  u_s - u_{t} = 0\text{ for all }1\leq s < t \leq S,\\
&\hspace{5.2em} v-u_s=0\text{ for all } s=1,\ldots,S,
\end{split}
\end{equation}
where $v_c$ refers to the $c$-th channel of $v$.
It is a basic but important observation that we did not alter the original problem, i.e.,
\eqref{eq:PottsModelMultichannel} and \eqref{eq:PottsSplitted}
are equivalent.
The ADMM approach requires to form the augmented Lagrangian of \eqref{eq:PottsSplitted}, that is,
\begin{align}
\label{eq:AugmentedLagrangian}
\begin{split}
\mathcal{L}&_{\mu,\rho}\big(v,\{u_s\},\{\lambda_{s,t}\},\{\tau_s\}\big)
= \sum_{c=1}^{C}\mathcal{D}(Av_c,f_c)\\&
+\sum_{s=1}^{S}\bigg\{
\gamma\omega_s \| \nabla_{d_s} u_s \|_0+\tfrac{\rho}{2}\| v-u_s+\tfrac{\tau_s}{\rho}\|^2
-\tfrac{1}{2\rho}\|\tau_{s} \|^2
+\sum_{t=s+1}^{S}\| u_s-u_t+\tfrac{\lambda_{s,t}}{\mu} \|^2
-\tfrac{1}{2\mu}\|\lambda_{s,t}\|^2
\bigg\}.
\end{split}
\end{align}
The constraints in \eqref{eq:PottsSplitted} are now part of the functional in the form
of the corresponding squared deviations weighted by the nonnegative parameters 
$\mu,\rho$. The variables  $\lambda_{s,t},\tau_s \in \R^{n\times n \times C}$ are the corresponding Lagrange 
multipliers.

In each ADMM iteration, the  Lagrangian $\mathcal{L}$ is sequentially 
minimized w.r.t. $v$ and $u_1,\ldots,u_S$
and afterwards a gradient ascent step is performed on the multipliers.
After some algebraic manipulation (see \ref{sec:derivationADMM}),
the minimization of the Lagrangian w.r.t.\,the data variable $v$ reads
\begin{equation}\label{eq:Vstep}
\begin{split}
\argmin_v \mathcal{L}_{\mu,\rho}&=
\argmin_v \sum_{c=1}^{C}\mathcal{D}(Av_c,f_c) + \tfrac{\mu S}{2}
\Big \| v-\tfrac{1}{S}\sum_{s=1}^{S}\Big(u_s - \tfrac{\tau_s}{\rho}\Big)\Big\|^2
\end{split}
\end{equation}
and the minimization of the Lagrangian w.r.t.\, $u_s$ is given by
\begin{equation}\label{eqUstep}
\begin{split}
\argmin_{u_s} \mathcal{L}_{\mu,\rho} &= 
\argmin_{u_s} 
\tfrac{2\gamma\omega_s}{\rho+\mu(S-1)}\|\nabla_{d_s} u_s\|_0 + \Big\| u_s -
\tfrac{\rho v + \tau_s+\sum_{t=s+1}^S (\mu u_t-\lambda_{s,t}) 
	+\sum_{r=1}^{s-1}(\mu u_r + \lambda_{r,s})}
{\rho+\mu(S-1)}\Big\|^2.\\
\end{split}
\end{equation}
We elaborate on efficiently solving  
\eqref{eq:Vstep}-\eqref{eqUstep} in the subsequent paragraph.
As it is common when dealing with non-convex problems, 
we employ monotonically increasing sequences $(\mu_k)_{k\in\N}$, $(\rho_k)_{k\in\N}$ as 
coupling parameters. This allows the splitting variables to develop
rather independently in the first iterations and forces them to become equal in the later
iterations.
We summarize the ADMM scheme for the multi-channel Potts model \eqref{eq:PottsModelMultichannel}
in Algorithm \ref{alg:ADMM}.

\begin{algorithm}[t]
	\def\vs{\\[0.5em]}
	\caption{\label{alg:ADMM}
		Potts ADMM
	}
	\SetCommentSty{footnotesize}
	\PrintSemicolon
	\KwIn{Forward operator $A\in\R^{m\times n^2}$, multi-spectral sinogram $f\in\R^{m\times c}$, PWLS weights $W\in\R^{m\times m\times C}$,
		stopping parameter $\mathtt{tol}>0$, jump penalty $\gamma>0$
	}
	\BlankLine
	\KwOut{$u^\ast\in\R^{n\times n\times C}$
	}
	\BlankLine
	Initialize $v^0,u_1^0,\ldots,u_S^0,\lambda_{s,t}^0,\tau_s^0$ by zero.
	
	$k\leftarrow 0$
	\BlankLine
	\Repeat{$\| u_s^k-u^k_{s+1} \|_\infty < \mathtt{tol}$ ~for all $s=1,\ldots,S-1$ and~
		$\|u_s^k - v^k\|_\infty < \mathtt{tol}$~ for all $s$}
	{
		\tcc{Solve the channel-wise PWLS problems}
		$v^{k+1}=
		\argmin_v \| Av - f\|_{W_c}^2+\tfrac{\mu_k S}{2}
		\Big \| v-\tfrac{1}{S}\sum_{s=1}^{S}\Big(u_s^k - 	\tfrac{\tau_s^k}{\rho_k}\Big)\Big\|^2$
		\vs
		\tcc{Solve linewise Potts problems along the directions $d_s$}
		\For{$s=1,\ldots,S$}{
			$w_s^k =\frac{\rho_k v_k + \tau_s^k+\sum\limits_{t=s+1}^S (\mu_k u_t^k
				-\lambda_{s,t}^k )
				+\sum\limits_{r=1}^{s-1}(\mu_k u_r^{k+1} + \lambda_{r,s}^k)}
			{\rho_k+\mu_k(S-1)}$
			\vs
			$u_s^{k+1} = \argmin_{u_s} 
			\tfrac{2\gamma\omega_s}{\rho_k+\mu(S-1)}\|\nabla_{d_s} u_s\|_0 + \| u_s - w_s^k\|^2$
		}
		
		\tcc{Update the multipliers}
		$\lambda_{s,t}^{k+1} = \lambda_{s,t}^k +\mu_k \big(u_s^{k+1} - u_t^{k+1}\big)$~~for all $s\neq t$\vs
		
		$\tau_s^{k+1} = \tau_s^k +\rho_k \big(v^{k+1} - u_{s}^{k+1}\big)$
		~~for all $s$
		
		\tcc{Increase the coupling parameters}
		$\mu_k\leftarrow \mu_{k+1}$,~$\nu_k\leftarrow\nu_{k+1}$\vs
		
		$k\leftarrow k+1$\vs
	}
	\BlankLine
	\Return{$u^\ast = \frac{1}{S}\sum_{s=1}^{S}u_s^k$}
\end{algorithm}

\paragraph{Solving the subproblems in the ADMM scheme.}
The crucial observation is that the subproblems \eqref{eq:Vstep} and \eqref{eqUstep}
can be solved efficiently.
We start with the first subproblem \eqref{eq:Vstep}. 
For arbitrary data $f,z\in\R^{n\times n \times C}$ and $\mu>0$ it is given by
\begin{equation}\label{eq:genericTikhonov}
v^\ast= \argmin_{v\in\R^{n\times n \times C}} \sum_{c=1}^{C}
\| W^{\frac{1}{2}} A v_c - W^{\frac{1}{2}}f_c\|^2
+\tfrac{\mu}{2}\| v - z\|^2.
\end{equation}
The $\ell_2$-regularized problem \eqref{eq:genericTikhonov} can be solved for each channel $c$
separately and 
the unique minimizer 
$v^\ast$ is channel-wise determined by the weighted normal equations
\begin{equation}
\label{eq:weightedNormalEquations}
\Big(A^T W_c A + \frac{\mu}{2}I \Big)\,v^\ast_c = A^T W_c f_c + \mu z_c.
\end{equation}
(Note that $I$ denotes the identity.)
Thus, solving \eqref{eq:genericTikhonov}
corresponds to solving the linear system \eqref{eq:weightedNormalEquations} 
for each $c=1,\ldots,C$.
We use the conjugate-gradient method to solve 
\eqref{eq:weightedNormalEquations}.

The remaining subproblems \eqref{eqUstep} in $u_1,\ldots,u_S$ have the generic form
\begin{equation}
\label{eq:genericPottsSubproblem}
\argmin_{u_s\in \R^{n\times n\times C}}  \| u_s - f\|^2+\gamma\| \nabla_{d_s} u_s\|_0
\end{equation}
for some data $f\in\R^{n\times n\times C}$ and $\gamma > 0$.
We note that solutions of \eqref{eq:genericPottsSubproblem}
are not necessarily unique. However, the set of data 
for which the solution is non-unique is a zero-set in $\R^n$ \cite{wittich2008complexity}.
Since only one direction $d_s$ of the discretized gradient is involved,
\eqref{eq:genericPottsSubproblem} decomposes into one-dimensional 
(vector-valued) Potts problems
along the paths induced by the direction $d_s$; see Figure \ref{fig:decompositionPottsproblems} 
for an illustration of these paths. 
We obtain a minimizer of \eqref{eq:genericPottsSubproblem} by determining 
the minimizers of each of these
one-dimensional Potts problems separately. This corresponds to minimizing
functionals of the following type
\begin{equation}
\label{eq:Potts1D}
P(u) =\| u -g\|^2+\gamma \|\nabla u \|_0.
\end{equation}
Here, $u,g\in\R^{C\times n}$ and the $\ell_0$-term denotes the number of jumps of $u$, i.e.,
\begin{equation}
\| \nabla u \|_0 = \big|\{ i\in\{1,\ldots,n-1\} : u_c(i) \neq u_c(i+1)\text{ for at least one } c\}\big|.
\end{equation}
Despite being a non-convex problem,
we use the well-known fact that the one-dimensional Potts problem \eqref{eq:Potts1D}
can be solved 
non-iteratively and exactly by dynamic programming 
\cite{bellman1969curve, blake1989comparison,auger1989algorithms,winkler2002smoothers, jackson2005algorithm, friedrich2008complexity} which we briefly describe in the following.
Further details can be found in \cite{friedrich2008complexity,storath2019smoothing}.
Assume that we have computed the minimizers $u^l$ of \eqref{eq:Potts1D}
for reduced data $(g_1,\ldots,g_l)\in\R^{C\times l}$ for each $l=1,\ldots,r$,  $r<n$.
Then the minimum value 
of \eqref{eq:Potts1D} for data $(g_1,\ldots,g_{r+1})$ is given by
\begin{equation}\label{eq:PottsBellman}
P(u^{r+1})=\min_{l=1,\ldots,r+1} P^\ast(u^l)+\gamma+\mathcal{E}^{l:r+1},
\end{equation}
where we define $u^0$ to be the empty vector, let $P(u^0) = -\gamma$ and denote by $\mathcal{E}^{l:r+1}$ 
the sum of the quadratic deviations of $(g_l,\ldots,g_{r+1})$ from its
channel-wise means.
If we denote the minimizing argument of \eqref{eq:PottsBellman} by $l^\ast$, the minimizer
$u^{r+1}$ is given by 
\begin{equation}\label{eq:dynamicUpdate}
u^{r+1} = (u^{l^\ast-1}, \mu_{[l^\ast,r+1]},\ldots,\mu_{[l^\ast,r+1]})
\end{equation}
w.r.t.\,the vector $\mu_{[l^\ast,r+1]}\in\R^C$ of channel-wise means of data $(g_{l^\ast},\ldots g_{r+1})$. Hence, we obtain a minimizer
for full data $g$ by successively computing $u^1,u^2,\ldots$ until we reach $u^n$.
An efficient way to evaluate \eqref{eq:PottsBellman} uses precomputed first and second
moments of data $g$ and stores only the jump locations \cite{friedrich2008complexity}.
Thereby, the described method  has quadratic worst case time complexity in the data length $n$
and  linear worst case time as well as linear space complexity in the number of channels $C$ 
\cite{storath2014fast}.
Additionally, we prune the search space in \eqref{eq:PottsBellman} for speedup as in \cite{storath2014fast}.
Thus, the proposed algorithm can be efficiently applied
to vector-valued data with a large number of channels 
which appear in multi-spectral CT.

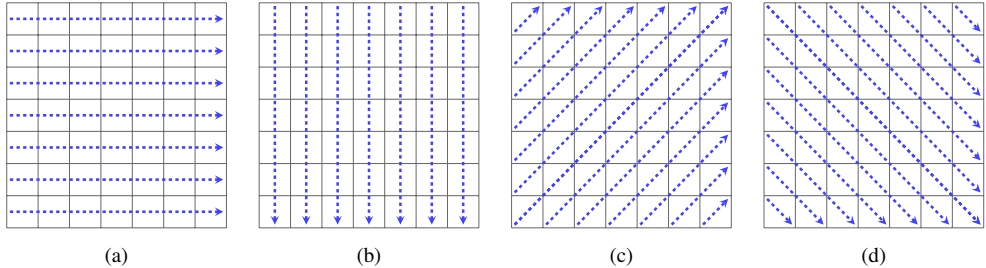
\begin{figure}[!t]
	\newcommand*{\xMin}{0}%
\newcommand*{\yMin}{0}%
\newcommand*{\xMax}{6}%
\newcommand*{\yMax}{6}
\newcommand*{\Step}{1}
\newcommand*{\yOff}{0.1}
\newcommand*{\xOff}{0.1}
\newcommand*{\arrowSize}{0.8mm}
\newcommand*{\lineSize}{0.3mm}
\newcommand*{\emphLineSize}{1.2mm}
\newcommand*{\arrCol}{blue!75}
\newcommand*{\arrColSpecial}{red!75}
\newcommand*{\gridCol}{black!60}
\usetikzlibrary{backgrounds}

\def\decomposingWidth{0.23}
\begin{subfigure}{\decomposingWidth\textwidth}
	\resizebox{\textwidth}{\textwidth}{	
		\begin{tikzpicture}
		\draw[step=\Step,\gridCol,line width=\lineSize] (\xMin,\yMin) grid (\xMax+1,\yMax+1);
		\foreach \i in {\yMin,...,\yMax} {
			\draw[->,\arrCol,>=stealth,dashed,line width=\arrowSize] (\xMin+\xOff,\i+0.5) -- (\xMax-\xOff+1,\i+0.5)  ;
		}
		\end{tikzpicture}
	}
	\subcaption{\label{subfig:1Dhoriz}}
\end{subfigure}
\hfill
\begin{subfigure}{\decomposingWidth\textwidth}
	\resizebox{\textwidth}{\textwidth}{	
		\begin{tikzpicture}
		\draw[step=\Step,\gridCol,line width=\lineSize] (\xMin,\yMin) grid (\xMax+1,\yMax+1);
		\foreach \i in {\xMin,...,\xMax} {
			\draw[->,\arrCol,>=stealth,dashed,line width=\arrowSize] (\i+0.5,\yMax-\yOff+1) -- (\i+0.5,\yMin+\yOff) ;
		}
		\end{tikzpicture}
	}
	\subcaption{\label{subfig:1Dvert}}
\end{subfigure}	
\hfill
\begin{subfigure}{\decomposingWidth\textwidth}
	\resizebox{\textwidth}{\textwidth}{	
		\begin{tikzpicture}
		\draw[step=\Step,\gridCol,line width=\lineSize] (\xMin,\yMin) grid (\xMax+1,\yMax+1);
		\foreach \i in {\yMin,...,\yMax} {
			\draw[->,\arrCol,>=stealth,dashed,line width=\arrowSize]
			(\xMin+\xOff,\i+\yOff) -- (\xMax-\i+1-\xOff,\yMax+1-\yOff);
		}
		\foreach \i in {\xMin,...,\xMax} {
			\draw[->,\arrCol,>=stealth,dashed,line width=\arrowSize]
			(\i+\xOff,\yMin+\yOff) -- (\xMax+1-\xOff,\yMax-\i+1-\yOff);
		}		
		\end{tikzpicture}
	}
	\subcaption{\label{subfig:1Dnortheast}}	
\end{subfigure}
\hfill
\begin{subfigure}{\decomposingWidth\textwidth}
	\resizebox{\textwidth}{\textwidth}{	
		\begin{tikzpicture}
		\draw[step=\Step,\gridCol,line width=\lineSize] (\xMin,\yMin) grid (\xMax+1,\yMax+1);
		\foreach \i in {\yMax,...,\yMin} {
			\draw[->,\arrCol,>=stealth,dashed,line width=\arrowSize]
			(\xMin+\xOff,\i+1-\yOff) -- (\i+1-\xOff,\yMin+\yOff);
		}
		\foreach \i in {\xMin,...,\yMax} {
			\draw[->,\arrCol,>=stealth,dashed,line width=\arrowSize]
			(\i+\xOff,\yMax+1-\yOff) -- (\xMax+1-\xOff,\yMin+\i+\yOff);
		}		
		\end{tikzpicture}
	}
	\subcaption{\label{subfig:1Dsoutheast}}
\end{subfigure}
\caption{
	\label{fig:decompositionPottsproblems}
	Decomposition of 
	problem \eqref{eq:genericPottsSubproblem} into one-dimensional Potts problems along the lines in directions $d_1,\ldots,d_4$ 
	on the discrete image domain 
	$\Omega$ .
}

	
\end{figure}

\section{Potts-Based Superiorization}
\label{sec:PottsSuperiorization}
Superiorization denotes a class of techniques that obtain
more regularized results (w.r.t.\,an appropriately chosen prior) from a basic iterative 
reconstruction method.
SM takes an iterative \emph{basic algorithm} and \emph{perturbs} its iterates employing an exogenous \emph{target function} to obtain a \emph{superior} solution w.r.t.\,this
target function/prior \cite{censor2020superiorization}.
An example of a basic algorithm that is perturbation resilient is the Landweber iteration for the least squares problem. The Landweber iteration can be superiorized 
w.r.t.\,any target function that yields non-ascending directions 
that are bounded and summable.
To fix ideas, if one chooses total variation as the target function
\cite{zibetti2018super,helou2018superiorization}, such non-ascending directions may be efficiently computed based on directional derivatives.
Both energy minimization methods and superiorization
methods typically result in two steps within an algorithmic scheme: 
a forward step followed by a regularizing/perturbation step. However, the interpretation of
these steps are different. For such an interpretation and an overview of superiorization and 
a comparison to (convex) optimization methods, we refer to \cite{censor2020superiorization}.

We propose methods which take a particular
variant of the conjugate gradient method (CG) 
as the basic algorithm and 
the (multi-channel) Potts prior \eqref{eq:PottsModelMultichannel} as the target function 
to obtain Potts superiorized, i.e., piecewise constant, solutions.
As a result, the iterates are not CG iterates but perturbed
versions of them. Thus, the classical result that 
the CG method converges to a minimizer of the (weighted) least squares problem 
after a fixed number of iterations cannot be expected to hold anymore.
However, we will see that the
termination of its perturbed counterpart  at an approximate least squares solution may still be ensured.

\subsection{Perturbing the CG Method with the Block-Wise Potts Prior.}
\label{sec:BasicPottsPerturbing}
To keep things focused, we first consider single-channel measurements and
the simplest anisotropic neighborhood for the Potts prior.
In a subsequent section, we provide the extension to multi-channel measurements and
the more isotropic neighborhood with diagonal directions.

As a starting point, we consider the 
penalized weighted least squares model (PWLS)
corresponding to the forward operator $A$, weights $W$ and (single-channel) data $f$, 
which is, analogously to \eqref{eq:DataTerm}, given by 
\begin{equation}
\label{eq:weightedLS}
u^\ast \in \argmin_u \frac{1}{2}
\big\| Au - f\big\|^2_W =\argmin_u \frac{1}{2}\big\| W^{\frac{1}{2}}Au - W^{\frac{1}{2}}f\big\|^2.
\end{equation}
As the next step, we duplicate the objective variable
and incorporate the difference between the two variables weighted by
a parameter $\mu> 0$. 
This will allow us to apply the CG method to an overdetermined problem and 
will serve as the basis for using a block-wise (more accessible) version of the Potts prior.
The modified problem is given by 
\begin{equation}
\label{eq:weightedLSsplit}
(u_1^\ast,u_2^\ast)\in\argmin_{u_1,u_2} \frac{1}{2}
\Bigg\|
\begin{pmatrix}
W^{\frac{1}{2}}A & 0 \\
0 & W^{\frac{1}{2}}A\\
\mu I & -\mu I
\end{pmatrix}
\begin{pmatrix}
u_1 \\ u_2
\end{pmatrix}
- 
\begin{pmatrix}
W^{\frac{1}{2}}f \\ W^{\frac{1}{2}}f \\0
\end{pmatrix}
\Bigg\|^2,
\end{equation}
where $I$ denotes the identity.
This formulation is inspired by the splitting 
used in our work \cite{kiefer2018iterative}. 
Below we will see that this splitting is necessary since otherwise the Potts perturbation steps would
correspond to solving NP-hard problems.
Please also note that due to the splitting the least squares system \eqref{eq:weightedLSsplit} is
overdetermined (even when matrix $A$ is underdetermined) and thus 
can be solved using the CG method adapted to the least squares problem.
The problems \eqref{eq:weightedLS} and 
\eqref{eq:weightedLSsplit} are equivalent in the following sense.
A minimizer $u^\ast$ of \eqref{eq:weightedLS} yields via 
$(u^\ast,u^\ast)$ a minimizer of \eqref{eq:weightedLSsplit}. Conversely,
a minimizer $(u^\ast_1,u^\ast_2)$ of \eqref{eq:weightedLSsplit} satisfies
$u_1^\ast = u_2^\ast$ and $u_1^\ast$ is a minimizer of \eqref{eq:weightedLS}.
We prove this 
for a more general situation in Proposition \ref{prop:WeightedLSBlockwiseEquality} below.
Consequently, we did not alter the original PWLS problem \eqref{eq:weightedLS}.

In contrast to the Landweber method --applying the gradient descent method
on the least squares objective-- the step sizes
of the CG method do not depend on the operator norm of
$A$ which is particularly advantageous for CT problems,
where the operator norm of $A$ 
is typically large and the Landweber method becomes unfavorable. 

In \cite{zibetti2018super}, the authors studied
several variants of the CG method
and identified a particular one which is \emph{strongly perturbation resilient}. 
That is, its termination is still ensured when its iterates are perturbed by a sequence of (summable) perturbations.
In the following, we briefly recall these relations.
The pseudocode of a single step of this variant of the CG method 
is given in Algorithm \ref{alg:CGstep} in the appendix.

\paragraph{Strong perturbation resilience.}
Strong perturbation resilience is a concept used in
the theory of superiorization. 
It ensures that an iterative algorithm still terminates
when its iterates are (appropriately) perturbed.
To give a precise definition of 
strong perturbation resilience, we need some preparations.

As a first step, we denote by $\mathcal{S}_\mathcal{T}$ 
the solution set of a given problem $\mathcal{T}$. 
Furthermore, we let $\mathcal{A}$ be an algorithmic operator which generates
sequences $\mathcal{X}_\mathcal{A}=(u^k)_{k\in\N_0}$ of the form
\begin{equation}\label{eq:basicAlgorithmSequence}
u^{k+1} = \mathcal{A}(u^k) \quad\text{ for all } k \geq 0
\end{equation}
that converge to points in 
$\mathcal{S}_\mathcal{T}$ for any initial $u^0$.
For our purposes, $\mathcal{T}$ will correspond to 
the PWLS \eqref{eq:weightedLSsplit} and 
$\mathcal{S}_\mathcal{T}$ to the set of its minimizers.
\begin{definition}[Proximity function, $\varepsilon$-compatibility, proximity sets]
	\label{def:proximity}
	Given a problem $\mathcal{T}$, a \emph{proximity function} is a function
	$\mathcal{P}r_\mathcal{T}(x):\R^n\to [0,\infty)$ that measures how incompatible
	$x$ is with $\mathcal{T}$. For any $\varepsilon > 0$ and any proximity function,
	we define that
	$x$ is \emph{$\varepsilon$-compatible} with $\mathcal{T}$ if 
	$\mathcal{P}r_\mathcal{T}(x)\leq \varepsilon$. We call the sets
	\begin{equation}\label{eq:proximitySet}
	\Gamma_\varepsilon = \big\{x\in\R^n: \mathcal{P}r_\mathcal{T}(x)\leq \varepsilon\big\}
	\end{equation}
	\emph{proximity sets}.
\end{definition}
For weighted least squares problems, such as \eqref{eq:weightedLSsplit},
the proximity function is typically given by the residual sum of squares
and
$\varepsilon$ is chosen according to the assumed level of noise.
Next, we define the notion of an
$\varepsilon$-output of a sequence w.r.t.\,a proximity set.
\begin{definition}[$\varepsilon$-output of a sequence with respect to $\Gamma_\varepsilon$]
	For some $\varepsilon>0$, a nonempty proximity set $\Gamma_\varepsilon$
	and a sequence $\mathcal{X}=(u^k)$ of points, the \emph{$\varepsilon$-output}
	$O(\Gamma_\varepsilon,\mathcal{X})$ of $\mathcal{X}$ w.r.t.\,$\Gamma_\varepsilon$
	is defined to be the element $u^k$ with smallest $k\in\N$ such that 
	$u^k\in\Gamma_\varepsilon$.
\end{definition}
We can now specify 
strong perturbation resilience of an algorithmic operator.
\begin{definition}[Strong perturbation resilience]
	\label{def:perturbationResilience}
	Let $\Gamma_\varepsilon$ denote the proximity set and 
	$\mathcal{A}$ be an algorithmic operator.
	The algorithmic scheme \eqref{eq:basicAlgorithmSequence}  is 
	\emph{strongly perturbation resilient} if 
	\begin{enumerate}[label=(\roman*)]
		\item there is an $\varepsilon>0$ such that the $\varepsilon$-output
		$O(\Gamma_\varepsilon,\mathcal{X}_\mathcal{A})$ of the sequence 
		$\mathcal{X}_\mathcal{A}$ exists for every initial $u^0$;
		\item for all $\varepsilon\geq 0$ such that
		$O(\Gamma_\varepsilon,\mathcal{X}_\mathcal{A})$ is defined for every 
		$u^0$, we also have that
		$O(\Gamma_{\varepsilon'},\mathcal{Y}_\mathcal{A})$ is defined for
		every $\varepsilon'\geq \varepsilon$ and for every sequence 
		$\mathcal{Y}_\mathcal{A}=(\bar{u}^k)$ generated by 
		\begin{equation}\label{eq:basicSuperiorization}
		\bar{u}^{k+1} = \mathcal{A}(\bar{u}^k+\beta_kv^k),\quad\text{ for all } k\geq 0,
		\end{equation}
		where the terms $\beta_kv^k$ are bounded perturbations, i.e.,
		$(v^k)_{k\in \N}$ is bounded, $\beta_k\geq 0$ for all $k$ and 
		$\sum_{k=0}^{\infty}\beta_k < \infty$.
	\end{enumerate}
\end{definition}
We record  the strong perturbation resilience of 
the presented variant of the CG method, which has been proven
in \cite[Theorem A.1]{zibetti2018super}.
\begin{theorem}\label{thm:perturbationResilienceCG}
The CG method as recorded in Algorithm \ref{alg:CGstep} is strongly perturbation resilient.
\end{theorem}

\paragraph{Perturbations by steps towards the proximal mapping of the
	block-wise Potts prior.}
Towards Potts superiorized solutions, we perturb
the iterates of Algorithm \ref{alg:CGstep} by means of
the Potts prior. 
The basic scheme is summarized in Algorithm \ref{proc:Superiorization}.
\begin{algorithm}[t]
	\caption{\label{proc:Superiorization}
		Basic Potts superiorized CG
	}
	\SetCommentSty{footnotesize}
	\PrintSemicolon
		Choose an initial perturbation parameter $\beta_0>0$ and an annealing
		parameter $0<a<1$, set $k\leftarrow 0$. Iterate until stopping criterion:
		
		Perform a CG step on $(u_1,u_2)^k$ 
		w.r.t.\,the least squares problem 
		\eqref{eq:weightedLSsplit}
		and 
		obtain $(u_1,u_2)^{k+1/2}$.
		
		Perturb the iterate 
		proportionally to $\beta_k$ with the Potts prior, i.e., 
		$(u_1,u_2)^{k+1}\leftarrow\mathcal{P}_{\beta_k}(u_1,u_2)^{k+1/2}$
		
		Update 
		$\beta_{k+1} \leftarrow a \beta_k$, $k\leftarrow k+1$. Go to 2.
\end{algorithm}
Typically, in the context of superiorization, the perturbation $\mathcal{P}_{\beta}$
in step 3
corresponds to
adding a \emph{non-ascending direction} w.r.t.\,the target function 
to the current iterate with appropriate step-sizes as in \Cref{def:perturbationResilience}.
Thereby, the target function values of the iterates are decreased 
so that the final result should become more desirable in terms
of the target function.
For a smooth target function, a non-ascending direction may be obtained
from its negative gradient.
Decreasing the parameter $\beta$ in step 4 ensures that the perturbations 
--denoted by $\mathcal{P}_{\beta_k}$ in step 3-- become smaller in the course of the iterations and the sequence is summable, i.e.,
$\sum_{k=0}^{\infty}\beta_k < \infty$.

As the Potts prior is non-smooth, we resort to
the evaluation of proximal mappings -- a common practice in non-smooth optimization.
To this end, we define for $(u_1,u_2)$ the \emph{block-wise Potts prior}
\begin{equation}
\label{eq:BlockwisePottsPriorAniso}
F(u_1,u_2) = \| \nabla_1 u_1 \|_0 + \| \nabla_2 u_2 \|_0.
\end{equation}
The block-wise Potts prior \eqref{eq:BlockwisePottsPriorAniso} counts the number of 
row-wise jumps in $u_1$ and the number of column-wise jumps in $u_2$.
The proximal mapping of $F$ for $\beta > 0$ 
is given by
\begin{equation}
\label{eq:BlockwisePottsProxAniso}
\prox_{\beta F}(u_1,u_2) =\argmin_{w_1,w_2}
F(w_1,w_2) + \tfrac{1}{2\beta} \big\| (u_1,u_2) - (w_1,w_2) \big\|^2.
\end{equation}
We recall that the right-hand side of \eqref{eq:BlockwisePottsProxAniso}
is unique for almost all $u_1,u_2$, \cite{wittich2008complexity}, so that
using the equality in \eqref{eq:BlockwisePottsProxAniso} is reasonable.
The proximal mapping \eqref{eq:BlockwisePottsProxAniso}
can be evaluated efficiently.
First, we note that the minimization in \eqref{eq:BlockwisePottsProxAniso} 
can be performed block-wise for $w_1$ and $w_2$, respectively.
These smaller problems are instances of \eqref{eq:genericPottsSubproblem} 
for $\gamma = 2\beta$,
which can be solved by dynamic programming as described 
in Section \ref{sec:PottsModel}. Hence, \eqref{eq:BlockwisePottsProxAniso} can be efficiently
evaluated.
In this context, we note
that omitting the introduction of $u_1,u_2$ in \eqref{eq:weightedLSsplit}
would lead to a proximal mapping in a single variable only whose evaluation 
corresponds to solving a two-dimensional Potts problem which is NP-hard.

It is well-known that
the proximal mapping corresponding to a convex lower semicontinuous function 
can be written in additive form \eqref{eq:basicSuperiorization} using the generalized gradient 
\cite{censor2020superiorization}. This however, is not feasible for
the block-wise Potts prior as it is not convex.
\Cref{prop:nonascending} will ensure that 
we still can use \eqref{eq:BlockwisePottsProxAniso} to obtain non-ascending 
directions for the block-wise Potts prior. To this end, we first note the following lemma
whose proof can be found in the appendix.
\begin{lemma}\label{lemma:neighboringEqualityPotts}
	Let $J_1(w)$ denote the positions of the horizontal jumps of $w$ and 
	$J_2(w)$ the positions of vertical jumps of $w$, i.e., 
	$\{x,x+d_s\}\in J_s(w)$ if and only if $w(x)\neq w(x+d_s)$, $s=1,2$.
	Then we have the inclusions $J_1(\bar{u}_1)\subset J_1(u_1)$ and $J_2(\bar{u}_2)\subset J_2(u_2)$
	for $(\bar{u}_1,\bar{u}_2) = \prox_{\beta F}(u_1,u_2)$.
	In other words, the proximal mapping of the block-wise Potts prior does not introduce
	jumps which were not already present in its arguments.
\end{lemma}

\begin{lemma}\label{prop:nonascending}
	For any $u=(u_1,u_2)$ and $\beta \geq 0$, the vector given by 
	\begin{equation}\label{eq:nonascendingDirectionBlockPotts}
	v=(v_1,v_2):=\begin{cases}
	\frac{\prox_{\beta F}(u_1,u_2)-(u_1,u_2)}{\delta} \quad &\text{ if }\prox_{\beta F}(u_1,u_2)\neq (u_1,u_2), \\
	0 &\text{ otherwise, }
	\end{cases}
	\end{equation}
	where $\delta=\|\prox_{\beta F}(u_1,u_2)-(u_1,u_2)\|$,
	satisfies  $\| v\| \leq 1$ and $F\big(u + t\cdot v\big) \leq F(u)$ for all $ t\geq 0$
	w.r.t.\,to the
	block-wise Potts prior $F$.
	Thus, \eqref{eq:nonascendingDirectionBlockPotts} yields a non-ascending direction for
	the block-wise Potts prior at $u=(u_1,u_2)$.
\end{lemma}
\noindent The proof of \Cref{prop:nonascending} can be found in the appendix.

A perturbation strategy for line 3 of Algorithm \ref{proc:Superiorization}
in terms of adding non-ascending directions is now given by 
\begin{equation}\label{eq:perturbNonAscending}
\begin{split}
\mathcal{P}_{\beta_k}\Big((u_1,u_2)^{k+1}\Big)&=
\begin{cases}
(u_1,u_2)^{k+1/2} + \beta_k
  \tfrac{\prox_{{\beta_k} F}\big((u_1,u_2)^{k+1/2}\big) -
	(u_1,u_2)^{k+1/2}}{\delta} &\text{if }\delta > 0,\\
(u_1,u_2)^{k+1/2}&\text{otherwise }
\end{cases}\\
&=:(u_1,u_2)^{k+1/2}+\beta_k\cdot\big(v_1,v_2\big)^{k},
\end{split}
\end{equation}
where $\delta=\|\prox_{\beta_k F}(u_1,u_2)^{k+1/2}-(u_1,u_2)^{k+1/2}\|$.
The summability of the parameters
$\beta_k$ by line 4 of Algorithm \ref{proc:Superiorization}
ensures that \eqref{eq:perturbNonAscending} produces
a sequence of bounded perturbations
as defined in \Cref{def:perturbationResilience}.
In particular, the sequence $\beta_k$ of annealing parameters 
satisfies $\sum_k \beta_k < \infty$ and the 
sequence of additive perturbations $v^k=(v_1,v_2)^{k}$ is bounded as 
$\|v^k\|\leq 1$.

We summarize the above considerations in the following theorem.
\begin{theorem}\label{prop:convergenceBasicScheme}
	Algorithm \ref{proc:Superiorization} with the 
	perturbations given by  \eqref{eq:perturbNonAscending}, i.e.,
	adding non-ascending directions w.r.t.\,the block-wise Potts prior,
	and the stopping criterion
	\begin{equation}\label{eq:SuperiorizationStopCrit}
	\| W^\frac{1}{2}Au_1 -W^\frac{1}{2}f\|^2
	+\| W^\frac{1}{2}Au_2 -W^\frac{1}{2}f\|^2
	+\| \mu( u_1 -u_2)\|^2 < \varepsilon
	\end{equation} 
	terminates for every $\varepsilon>\varepsilon_0$ and initializations
	$u_1,u_2$, where
	$\varepsilon_0 $ is the minimal value of the 
	underlying (weighted) least squares problem
	\eqref{eq:weightedLS} or, equivalently, 
	\eqref{eq:weightedLSsplit}. 
\end{theorem}
\begin{proof}
The proof follows from \Cref{thm:perturbationResilienceCG},
\Cref{prop:nonascending} 
and property (ii) in \Cref{def:perturbationResilience}.
\end{proof}
Concerning the parameters $\beta_k$, we used the scaling
${\| A^Tf\|_2}/{\|A\|_2^2}$
to adjust $\beta_k$ to the scale of the data
which corresponds to employing the sequence
$\beta_k = \frac{\| A^Tf\|_2}{\|A\|_2^2}\tilde{\beta}_k$, where
$\tilde{\beta}_k$ is summable, i.e., $\sum_{k}\tilde{\beta}_k < \infty$, e.g.,
$\tilde{\beta}_k = a^k\beta_0$ for $0<a<1$ and $\beta_0>0$ as in 
Algorithm \ref{proc:Superiorization}.

\paragraph{Perturbations by the proximal mapping of the block-wise Potts prior.}
We derived perturbations in terms of the block-wise
Potts prior by taking a step towards its proximal mapping.
By \Cref{prop:nonascending} this approach yields a sequence of non-ascending directions 
w.r.t.\,the block-wise Potts prior and \Cref{prop:convergenceBasicScheme} ensures the
termination of the corresponding instance of 
Algorithm \ref{proc:Superiorization}.
As \eqref{eq:perturbNonAscending} takes 
a step towards the 
proximal mapping of the block-wise Potts prior, we may consider 
perturbations which take the proximal mapping itself, that is,
\begin{equation}\label{eq:perturbProx}
\mathcal{P}_{\beta_k}\Big((u_1,u_2)^{k+1}\Big)= \prox_{\beta_k F}\Big((u_1,u_2)^{k+1/2}\Big).
\end{equation}
It follows immediately from optimality in the proximal mapping
that the block-wise Potts value  of 
$\mathcal{P}_{\beta_k}\Big((u_1,u_2)^{k+1}\Big)$ 
is lower or equal than the one of $(u_1,u_2)^{k+1/2}$.
(This also follows from \Cref{prop:nonascending}.)

In our experiments (see \Cref{fig:SuperiorizationApproaches}), we observed
that this perturbation strategy improves upon the perturbations
given by \eqref{eq:perturbNonAscending}.
Furthermore, the perturbation strategy \eqref{eq:perturbProx} 
needs not to be scaled by some constant factor to bring
$(u_1,u_2)^{k+1/2}$ 
and its perturbation $\mathcal{P}_{\beta_k}\Big((u_1,u_2)^{k+1}\Big)$
to the same scale.
We remark that \eqref{eq:perturbProx} may also be seen as an
additive perturbation by defining
$(v_1,v_2)^k = \frac{1}{\beta_k}\Big( \prox_{\beta_k F}\big((u_1,u_2)^{k+1/2}\big) -(u_1,u_2)^{k+1/2}\Big)$,
so that \eqref{eq:perturbProx} could be written as
$(u_1,u_2)^{k+1/2} + \beta_k (v_1,v_2)^k$.

\begin{remark}\label{rem:proximalBoundedPerturbation}
	It is an open question whether \eqref{eq:perturbProx} yields
	bounded perturbations in the sense of \Cref{def:perturbationResilience}. 
	In \cite{censor2020superiorization}, this was proven for
	the special case of (smoothed) total variation.
	The proof uses the Lipschitz continuity of the target function. 
	Unfortunately, neither the Potts prior nor the block-wise Potts prior
	are even continuous.
	In general, the results for perturbations obtained from
	evaluating proximal mappings are rather rare in the literature.
	However, the practical results encourage future work on this topic.
\end{remark}

\subsection{Potts S-CG}
In view of \eqref{eq:weightedLSsplit}, a large choice of $\mu$ 
enforces the CG steps to put more emphasis on the deviations between $u_1$ and $u_2$. 
Thus, large values of $\mu$ lead to results $u_1,u_2$ which are closer to each other.
Furthermore, the block-wise Potts prior \eqref{eq:BlockwisePottsPriorAniso}
perturbs $u_1,u_2$ towards row-wise and column-wise 
piecewise constancy, respectively.
In order to obtain solutions 
which satisfy equality, i.e., $u_1 = u_2$,
and which are in addition genuinely piecewise constant, we 
propose a modified approach of Algorithm \ref{proc:Superiorization} 
inspired by penalty methods for energy minimization:
after the lines 2-3 in Algorithm \ref{proc:Superiorization} have
been conducted,
we increase the coupling parameter $\mu$. 
By starting with a low value $\mu_0$, the CG steps give more weight
to the data fidelity in the first iterations, while in the later iterations they give more weight to the discrepancies between $u_1$ and $u_2$.
Thus, the variables become closer in the later iterations
until they become (approximately) equal.
This modified superiorization approach 
which employs perturbations by evaluating proximal mappings
is summarized in Algorithm \ref{proc:PottsCGscheme},
which we call \emph{Potts S-CG} (Potts superiorized conjugate gradient).

\begin{figure}[t]
	\centering
	\captionsetup{justification=centering}
	\def\figWidth{0.23\textwidth}
	\def\hs{\hspace{0.7em}}
	\def\vs{\\[0.5em]}	
	\begin{subfigure}{\figWidth}
		\includegraphics[width=\textwidth]{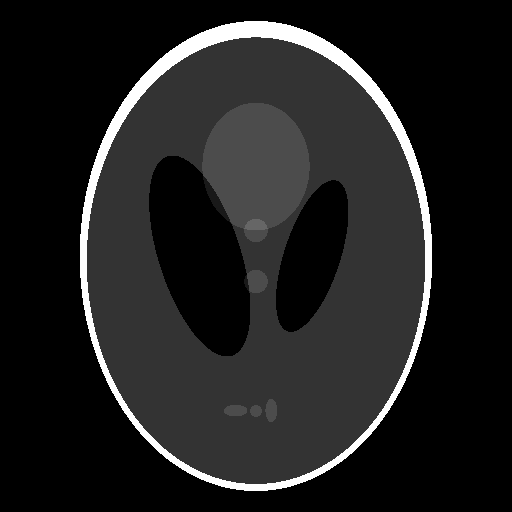}
		\subcaption{Ground truth\\$512\times 512$\label{subfig:LoganPhantom}\\~}
	\end{subfigure}\hs	
	\begin{subfigure}{\figWidth}
		\includegraphics[width=\textwidth]{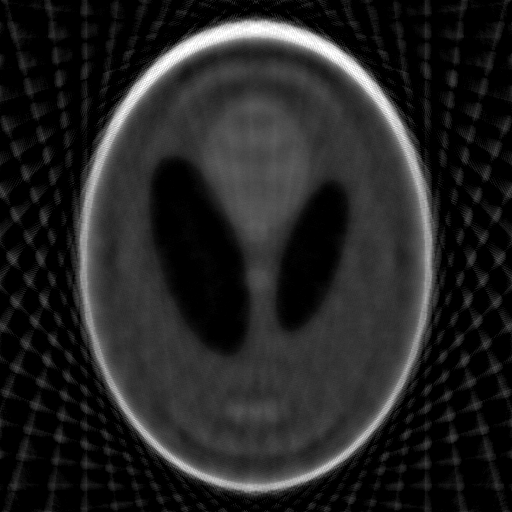}
		\subcaption{Unperturbed CG,\label{subfig:conjGrad}
			MSSIM=\protect\input{Experiments/SuperiorizationComparison/score_cg.txt}\unskip\\~}
	\end{subfigure}\hs
	\begin{subfigure}{\figWidth}
		\includegraphics[width=\textwidth]{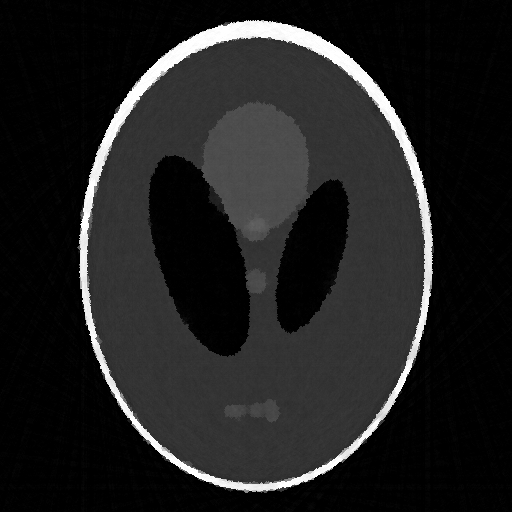}
		\subcaption{Perturbed by non-ascending directions,
			\label{subfig:classicSuper}
			$\mu$=\protect\input{Experiments/SuperiorizationComparison/rho.txt}\unskip,
			MSSIM=\protect\input{Experiments/SuperiorizationComparison/score_additive.txt}\unskip}
	\end{subfigure}\vs	
	\begin{subfigure}{\figWidth}
		~
		\subcaption*{~\\~}
	\end{subfigure}\hs		
	\begin{subfigure}{\figWidth}
		\includegraphics[width=\textwidth]{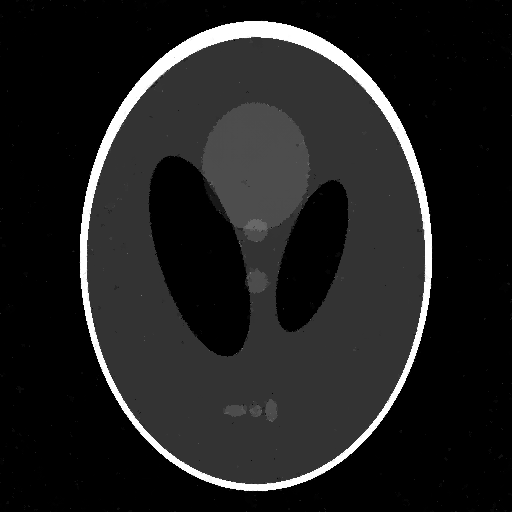}
		\subcaption{Perturbed by proximal mappings,\label{subfig:proxSuper}
			$\mu$=\protect\input{Experiments/SuperiorizationComparison/rho.txt}\unskip,
			MSSIM=\protect\input{Experiments/SuperiorizationComparison/score_prox.txt}\unskip}
	\end{subfigure}\hs	
	\begin{subfigure}{\figWidth}
		\includegraphics[width=\textwidth]{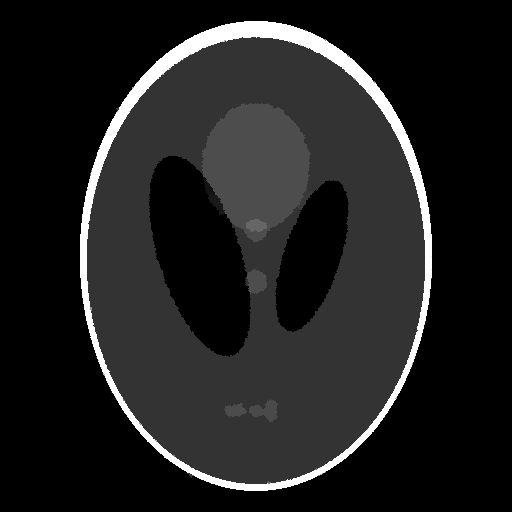}
		\subcaption{Potts S-CG,\label{subfig:modifiedSuper}\\
			$\mu_0$=\protect\input{Experiments/SuperiorizationComparison/rho0.txt}\unskip,
			MSSIM=\protect\input{Experiments/SuperiorizationComparison/score_pottsSCG.txt}\unskip\\~}
	\end{subfigure}	
	\captionsetup{justification=justified}	
	\caption[Comparison of Potts superiorization approaches for the 
	reconstruction from undersampled and noisy Radon data]{\label{fig:SuperiorizationApproaches}
		Reconstructions from undersampled Radon data 
		(\protect\input{Experiments/SuperiorizationComparison/nAngles.txt}\unskip~projection angles corrupted by Gaussian noise with 
		$\sigma=\protect\input{Experiments/SuperiorizationComparison/sigma.txt}$\unskip). 
		(\subref{subfig:LoganPhantom}) Original.
		(\subref{subfig:conjGrad}) Unperturbed CG method.
		(\subref{subfig:classicSuper}) Perturbing the CG iterations with non-ascending
		directions for the block-wise Potts prior improves upon the unperturbed
		CG method.
		(\subref{subfig:proxSuper}) Perturbing with the proximal mapping of the block-wise Potts prior yields a further regularized result.
		(\subref{subfig:modifiedSuper}) The proposed Potts S-CG method 
		produces a less grainy and piecewise constant result and achieves the highest
		MSSIM value.
	}
\end{figure}

\begin{nobreak}
\begin{algorithm}
	\caption{Basic anisotropic Potts S-CG scheme\label{proc:PottsCGscheme}}
	\SetCommentSty{footnotesize}
	\PrintSemicolon
		Choose an initial perturbation parameter $\beta_0>0$ 
		and coupling parameter $\mu_0>0$. Choose an annealing
		parameter $0<a<1$, set $k\leftarrow 0$.\newline
		Iterate until 
		$u_1,u_2$ become equal:
		
		Perform a CG step on $(u_1,u_2)^k$ 
		w.r.t.\,the least squares problem \eqref{eq:weightedLSsplit}
		for $\mu = \mu_k$ and 
		obtain $(u_1,u_2)^{k+1/2}$.
		
		Perturb $(u_1,u_2)^{k+1/2}$ by the proximal mapping of the 
		block-wise Potts prior for the current perturbation parameter $\beta_k$, i.e.,
		$(u_1,u_2)^{k+1}\leftarrow\prox_{\beta_k F} (u_1,u_2)^{k+1/2}$.
		
		Update 
		$\beta_{k+1} \leftarrow a \beta_k$,
		increase the coupling parameter $\mu_{k+1} \leftarrow (\mu_0\beta_0) / \beta_k$,
		$k\leftarrow k+1$.
		
		Go to 2.
\end{algorithm}
\end{nobreak}
In Figure \ref{fig:SuperiorizationApproaches}, we compare the 
results of the perturbation
strategies \eqref{eq:perturbNonAscending}, \eqref{eq:perturbProx} 
and the Potts S-CG approach for 
image reconstruction from undersampled noisy Radon data 
(measurements from $\protect 25$ projection angles corrupted by Gaussian noise of variance 
$\sigma = 0.25$). 
Please note that this corresponds to a simplified model compared to the more 
accurate model \eqref{eq:Poisson}. However, we use this simplified
model to obtain an initial comparison of the presented Potts superiorization methods.
For the perturbation strategies \eqref{eq:perturbNonAscending}, \eqref{eq:perturbProx},
we used the coupling parameter $\mu=10$
and the stopping rule \eqref{eq:SuperiorizationStopCrit} 
with $\varepsilon = m\sigma^2 + \mu^2$, where
$m$ denotes the total number of measurements.
Potts S-CG used the initial coupling parameter $\mu_0=0.02$
and the iteration was stopped when $u_1,u_2$ became (approximately)
equal.
All three approaches used the annealing parameters $a=0.99$ 
and $\beta_0=1$.
We observe that all perturbation strategies improve upon 
unperturbed CG. 
Furthermore, we see that 
taking the proximal mapping \eqref{eq:perturbProx} produces a more regularized solution than adding non-ascending directions 
\eqref{eq:perturbNonAscending}.
Finally, the result of Potts S-CG
is genuinely piecewise constant and
has the highest mean structural similarity index (MSSIM).
(We refer to \Cref{sec:Experiments} for more details on the
MSSIM.)

In the following, we give further details on the Potts S-CG approach.
In particular, we
describe the extensions to more isotropic solutions
and multi-channel data.

\paragraph{Extension to more isotropic discretizations and multi-channel measurements.}
We extend the Potts S-CG approach to more 
isotropic discretizations, i.e., directions $d_1,\ldots,d_S$,
and multi-channel measurements $f$. First, we note that 
the basic PWLS problem for multi-channel measurements 
is given by the channel-wise sum
\begin{equation}\label{eq:weightedLSmultichannel}
\min_u \sum_{c=1}^{C} \frac{1}{2}\| W^{\frac{1}{2}}Au_c - W^{\frac{1}{2}}f_c \|^2.
\end{equation}
As a preparation, we define the block-matrices $\overline{A}_{\mu,c}$ given by
\begin{equation}
\label{eq:superiorizationSystemMatrix}
\overline{A}_{\mu,c }= 
\begin{pmatrix}
W_c^{\frac{1}{2}}A & &    	 &        &       &    \\
& W_c^{\frac{1}{2}}A       	    &    \\
&    	   &   	& \ddots  & & \\
&    	   &   	&  &    &W_c^{\frac{1}{2}}A\\[0.4em]
\mu I   	& -\mu I 	& 0		&	\ldots	&	0	      & 0 \\
\mu I   	& 0  	    & -\mu I& \ldots    &	0	      & 0 \\
& \vdots    & 		& 			&\vdots	&        \\
\mu I	    & 0  	    & 0     & \ldots    &	          &-\mu I\\

0			&\mu I   	& -\mu I 	& \ldots	&	0	      & 0 \\
&\vdots   	& 	     	&  			&	\vdots	&        \\
0			&\mu I   	& 0 	& \ldots	&	0	      & -\mu I \\
& 		   	&     	&\vdots  			&	&        \\
0     		&0 		   	&0     	&\ldots  				&    \mu I  &  -\mu I
\end{pmatrix}\quad \text{and} \quad
\overline{f}_c = 
\begin{pmatrix}
W_c^{\frac{1}{2}}f_c\\
W_c^{\frac{1}{2}}f_c\\
\vdots \\
W_c^{\frac{1}{2}}f_c\\[0.4em]
0 \\
0 \\
\vdots\\
\vdots\\
\vdots\\
\vdots\\
\vdots\\
0 
\end{pmatrix}.
\end{equation}
The upper block part of $\overline{A}_{\mu,c}$ realizes
the data fidelities for the $c$-th channel
and its lower part the mutual deviations
between the variables
$(u_1,\ldots,u_S)$ in the $c$-th channel.
Please note that we use the notation $u_{s,c}$ 
to refer to the $c$-th channel of $u_s$.
Using this notation we can now formulate
the corresponding PWLS problem,
that is, the counterpart of \eqref{eq:weightedLSsplit};
it is given by 
\begin{equation}
\label{eq:weightedLSsplitIsotropic}
(u_1^\ast,\ldots,u_S^\ast)\in \argmin_{u_1,\ldots,u_S}
\sum_{c=1}^C
\frac{1}{2}
\big\|
\overline{A}_{\mu,c}\,
(u_{1,c}, \ldots,u_{S,c})^T
- 
\overline{f}_c
\big\|^2
\end{equation}
and the corresponding normal equations for each channel are given by
\begin{equation}
\label{eq:weightedLSsplitNormalEquationsIso}
\overline{A}_{\mu,c}^T\overline{A}_{\mu,c}\, (u_{1,c}, \ldots,u_{S,c})^T = 
\overline{A}_{\mu,c}^T\, \overline{f}_c.
\end{equation}
After evaluating the multiplications in \eqref{eq:weightedLSsplitNormalEquationsIso},
we obtain the following more explicit equations
\begin{equation}
\label{eq:weightedLSsplitIsotropicNormalEquations}
\begin{pmatrix}
\sum_{t=2}^{S} \mu^2 (u_{1,c}-u_{t,c}) + A^TW_cAu_{1,c} \\
\vdots\\
\sum_{t\neq s}^{S} \mu^2 (u_{s,c}-u_{t,c}) + A^TW_cAu_{s,c} \\
\vdots \\
\sum_{t=1}^{S-1} \mu^2 (u_{S,c}-u_{t,c}) + A^TW_cAu_{S,c}
\end{pmatrix}
=
\begin{pmatrix}
A^T W_cf_c \\ \vdots \\ A^T W_cf_c \\ \vdots \\A^T W_cf_c
\end{pmatrix}
\end{equation}
for each channel $c=1,\ldots,C$. The next proposition
establishes the relation between the basic weighted least squares problem \eqref{eq:weightedLSmultichannel} 
and its augmented counterpart \eqref{eq:weightedLSsplitIsotropic}.
A proof can be found in the appendix.
\begin{proposition}\label{prop:WeightedLSBlockwiseEquality}
	(i) A minimizer $u^\ast$ of the least squares problem 
	\eqref{eq:weightedLSmultichannel}
	induces a minimizer 
	of the augmented least squares problem \eqref{eq:weightedLSsplitIsotropic} 
	via $(u^\ast,\ldots,u^\ast)$. 
	(ii) A minimizer $(u_1^\ast,\ldots,u_S^\ast)$ of the augmented least squares problem
	\eqref{eq:weightedLSsplitIsotropic} satisfies $u_1^\ast = \ldots = u_S^\ast$. 
	Further, $u_1^\ast$ is a minimizer of \eqref{eq:weightedLSmultichannel} as well.
	(iii) If $A$ has full column rank, then $u^\ast$ is the (unique) minimizer of \eqref{eq:weightedLSmultichannel} if and only if 
	$(u^\ast,\ldots,u^\ast)$ is the (unique) minimizer of \eqref{eq:weightedLSsplitIsotropic}.
\end{proposition}

By following the lines of the general CG algorithm
(cf. Algorithm \ref{alg:CGstep} in the appendix), the CG step specifically for \eqref{eq:weightedLSsplitIsotropicNormalEquations} is given by 
Algorithm \ref{alg:GCstepSuper} in the appendix.

Next, we derive the perturbation step by means of the proximal
mapping of the block-wise Potts prior with $S$ directions.
We defined the block-wise Potts prior for the 
special case $S=2$ in \eqref{eq:BlockwisePottsPriorAniso}. For the more general 
situation with $S$ directions as above,
it is given by 
\begin{equation}
\label{eq:BlockwisePottsPriorIso}
F(u_1,\ldots,u_S) = \sum_{s=1}^{S} \omega_s \| \nabla_{d_s} u_s \|_0.
\end{equation}
(Recall the definition of $\| \nabla_{d_s} u_s \|_0$ in \eqref{eq:discreteL0Multichannel} 
and the weights $\omega_s$ from \Cref{sec:PottsPrior}.)
To perturb the CG iterates of \eqref{eq:weightedLSsplitIsotropicNormalEquations}
by means of the block-wise Potts prior \eqref{eq:BlockwisePottsPriorIso},
we need to evaluate its proximal mapping as above.
The proximal mapping of the block-wise Potts prior for $S$ directions \eqref{eq:BlockwisePottsPriorIso}
is defined by
\begin{equation}
	\label{eq:BlockwisePottsProxIso}
	\prox_{\beta F}(u_1,\ldots,u_S)
	= \argmin_{w_1,\ldots,w_S} F(w_1,\ldots,w_S) + \tfrac{1}{2\beta} 
	\sum_{s=1}^{S} \| u_s- w_s \|^2.
\end{equation}
As seen in \Cref{sec:PottsModel},
the minimization in \eqref{eq:BlockwisePottsProxIso} decomposes into 
problems, which depend on one of the $u_s$ only. These smaller problems are of the form 
\eqref{eq:genericPottsSubproblem},
which can be efficiently solved by dynamic programming as described in Section \ref{sec:PottsModel}.
We provide the pseudocode of the Potts S-CG 
for more isotropic discretizations and multi-channel data
in Algorithm \ref{alg:superiorizedCG}.

\begin{algorithm}[t]
	\def\vs{\\[0.3em]}
	\caption{\label{alg:superiorizedCG}
		Potts S-CG
	}
	\SetCommentSty{footnotesize}
	\PrintSemicolon
	\KwIn{Forward operator $A\in\R^{m\times n^2}$, multispectral sinogram $f\in\R^{m\times c}$, PWLS weights $W\in\R^{m\times m}$,
		annealing parameter $0<a<1$, initial parameters $\beta_0,\mu_0>0$,
		stopping parameter $\mathtt{tol}>0$}
	\KwOut{$u^\ast\in\R^{n\times n\times C}$
	}
	\BlankLine
	Initialize for all $s=1,\ldots,S$:
	
	$u_s^{0} = A^TWf$,~~$p_s^0=A^TWf$,~~
	$h_s^0=A^TWAp_s^0+\sum_{t\neq s} \mu^2_0(p_s^0-p_t^0)$
	
	$k=0$
	\BlankLine
	\Repeat{$\| u_s^k-u^k_{s+1} \|_\infty < \mathtt{tol}$ ~for all $s=1,\ldots,S-1$}
	{
		\tcc{Increase the splitting penalty parameter proportionally to $\beta_k$}
		$\mu_k \leftarrow (\mu_0\beta_0) / \beta_k$ 
		
			\tcc{Perturb the iterates with the block-wise Potts prior by solving linewise Potts problems for jump penalty
				$2\beta_k$: }
			$(u_1,\ldots,u_S)^{k+1/2}
			= \argmin\limits_{w_1,\ldots,w_S} 
			\sum_{s}\| u_s^{k}-w_s\|^2 +2\beta_k \omega_s \| \nabla_{d_s}w_s\|_0$ 			
			
			\tcc{Peform a CG step on $(u_1,\ldots,u_S)^{k+1/2}$
				w.r.t.\,the weighted normal equations \eqref{eq:weightedLSsplitIsotropicNormalEquations} : }
			\For{$c=1,\ldots,C$}{
				Compute $u_{s,c}^{k+1},h_{s,c}^{k+1},p_{s,c}^{k+1}$ 
				from $u_{s,c}^{k+1/2},h_{s,c}^{k},p_{s,c}^{k}$ 
				for all $s=1,\ldots,S$ by Algorithm \ref{alg:GCstepSuper}
				for $\mu=\mu_k$
			}
		
		$\beta_{k+1} \leftarrow a\beta_k$
		
		$k\leftarrow k+1$
	}
	\BlankLine
	\Return{$u^\ast = \frac{1}{S}\sum_{s=1}^{S}u_s^k$}
\end{algorithm}

\section{Comparison of Potts ADMM and Potts S-CG}
\label{sec:ComparisonADMMandPottsCG}
\paragraph{Handling the forward operator.}
Potts ADMM introduces an additional splitting variable 
$v$ in \eqref{eq:PottsSplitted}
to cope with the forward operator $A$. Consequently,
one has to solve a $\ell_2$-regularized problem \eqref{eq:Vstep} w.r.t.\,$A$ and $v$ in each iteration of Potts ADMM.
In contrast, Potts S-CG does not introduce such a splitting variable
and only requires a CG step w.r.t.\,$A$ in each iteration, i.e., it requires only applying $A$ and $A^T$ in each iteration instead of solving 
a $\ell_2$-regularized problem.
As a result, 
Potts S-CG needs more iterations,
while the iterations are cheaper than in Potts ADMM
(\Cref{fig:PottsADMMvsPottsSCG}: Potts S-CG 0.95\,sec/iteration, Potts ADMM 2.35\,sec/iteration).

\begin{figure}[t]
	\centering
	\captionsetup{justification=centering}
	\def\figWidth{0.26\textwidth}
	\def\plotWidth{0.47\textwidth}
	\def\hs{\hspace{0.5em}}
	\def\hss{\hspace{1.5em}}
	\def\vs{\\[1em]}	
	\begin{subfigure}{\figWidth}
		\includegraphics[width=\textwidth]{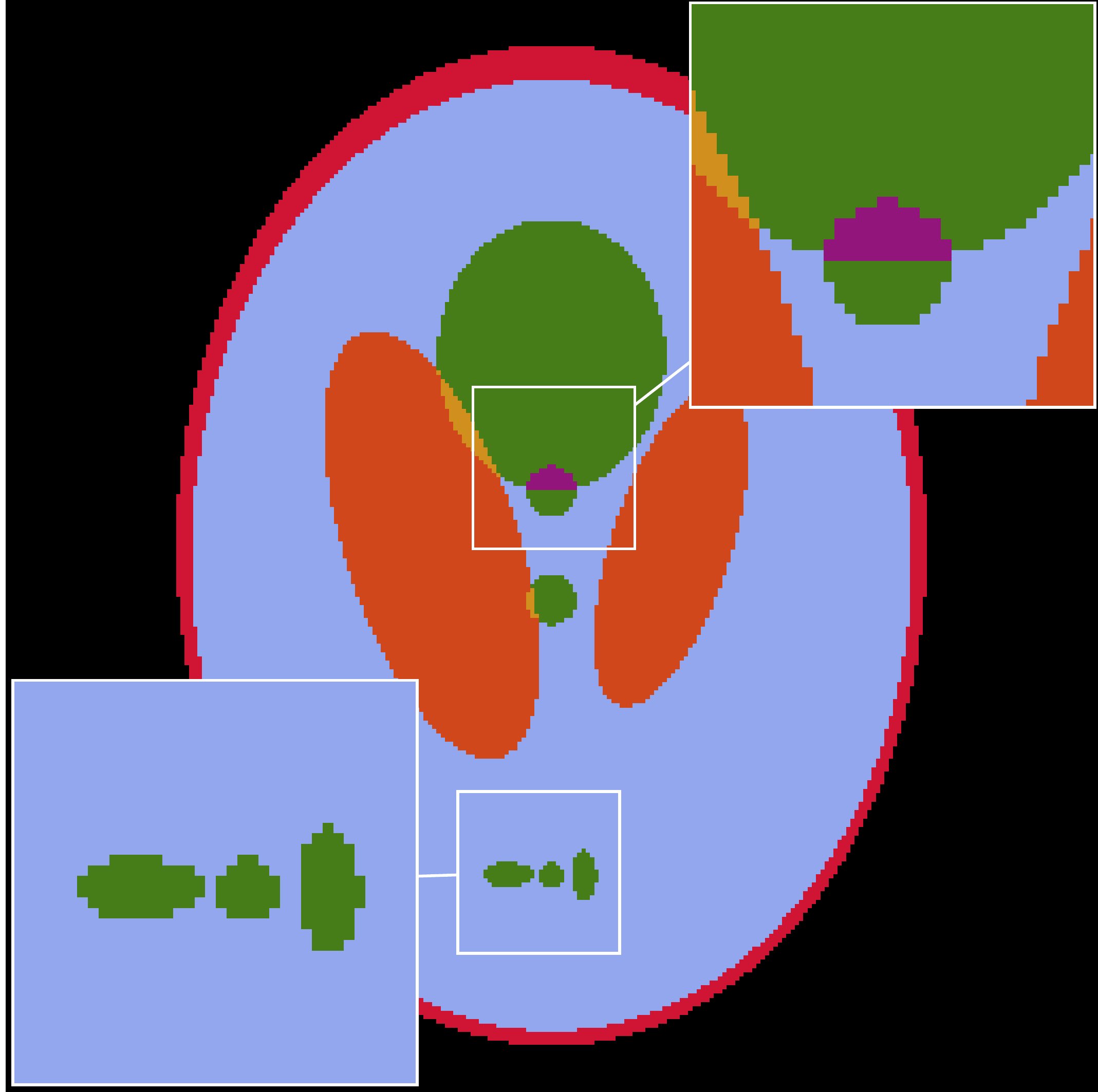}
		\subcaption{Original $256\times 256$\label{subfig:ColoredLoganPhantom}}
	\end{subfigure}\hs	
	\begin{subfigure}{\figWidth}
		\includegraphics[width=\textwidth]{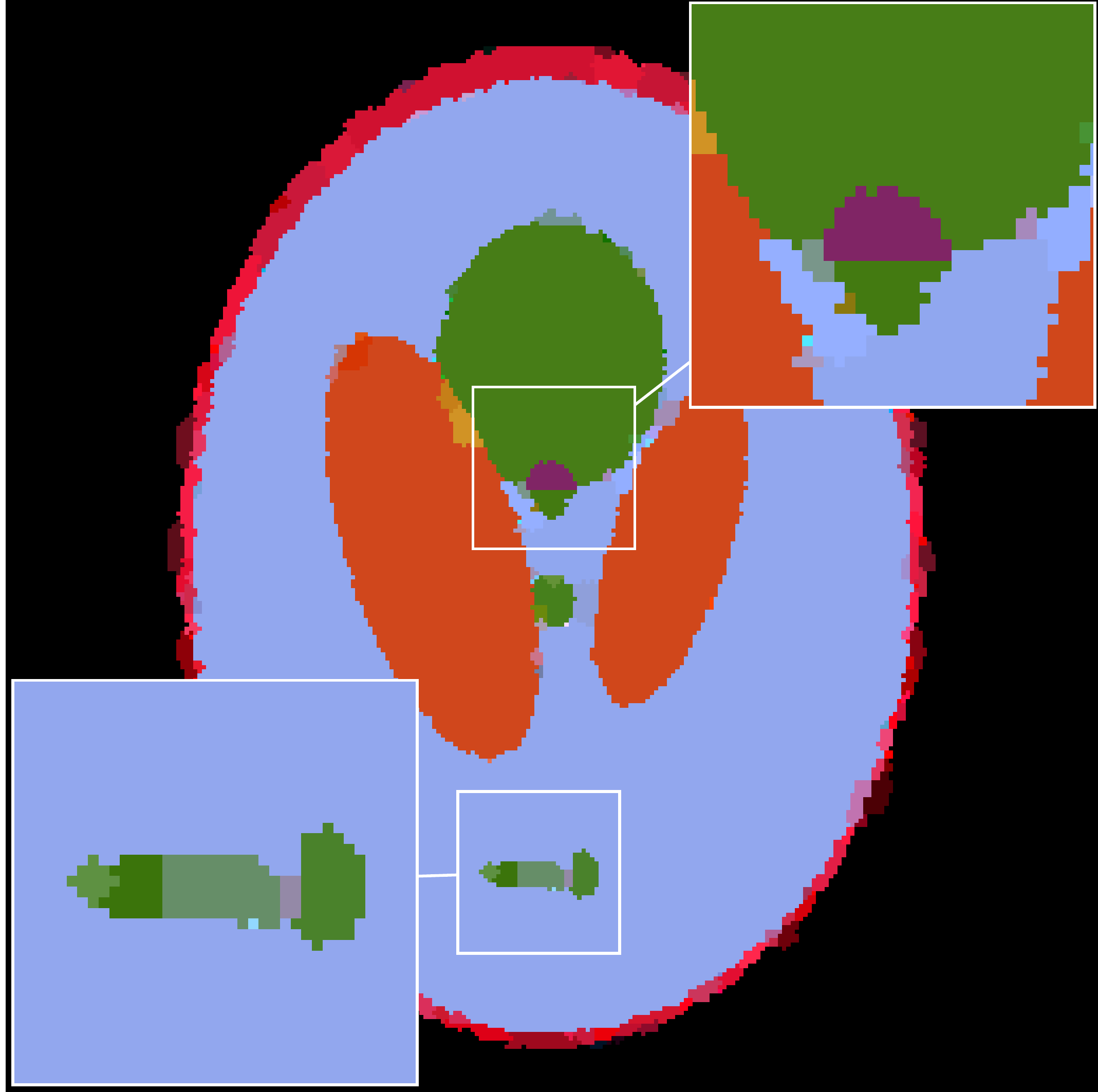}
		\subcaption{Penalty method\label{subfig:PottsRelaxation}}
	\end{subfigure}\hs
	\begin{subfigure}{\figWidth}
		\includegraphics[width=\textwidth]{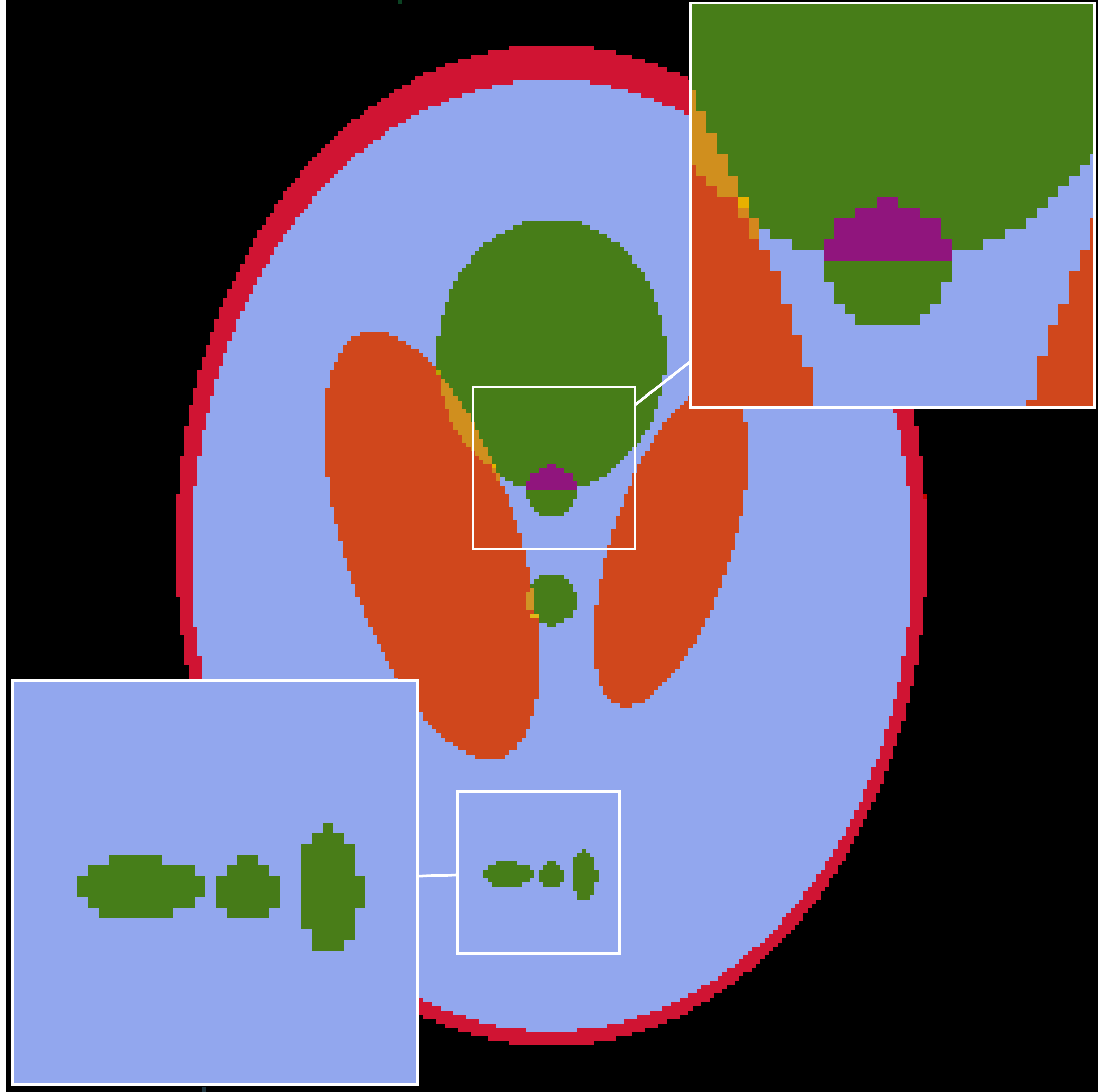}
		\subcaption{Potts ADMM\label{subfig:PottsADMM}}
	\end{subfigure}\vs
	\begin{subfigure}{\figWidth}
		~
		\subcaption*{~}
	\end{subfigure}\hs
	\begin{subfigure}{\figWidth}
		\includegraphics[width=\textwidth]{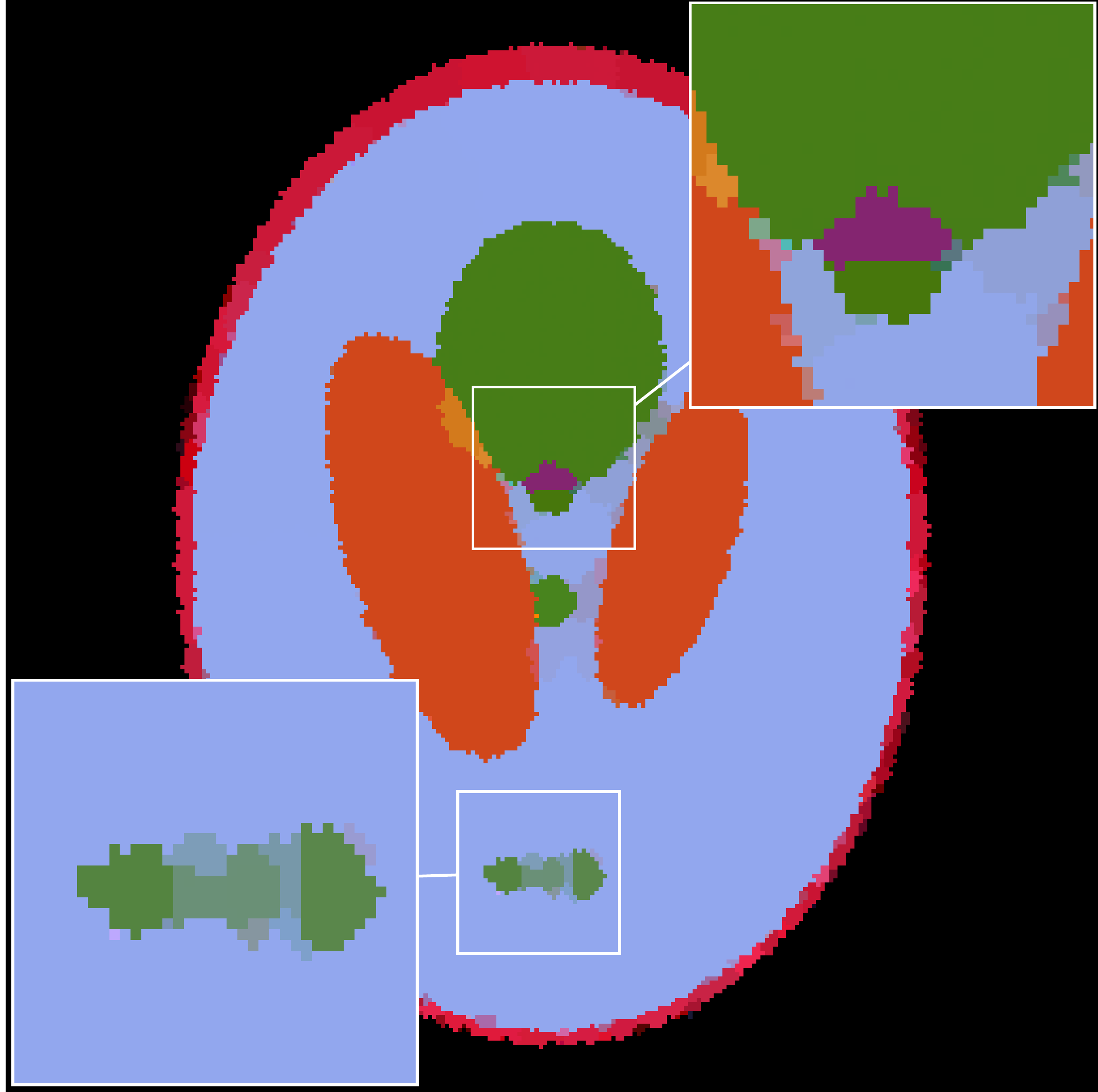}
		\subcaption{Potts S-Landweber\label{subfig:PottsLand}}
	\end{subfigure}\hs
	\begin{subfigure}{\figWidth}
		\includegraphics[width=\textwidth]{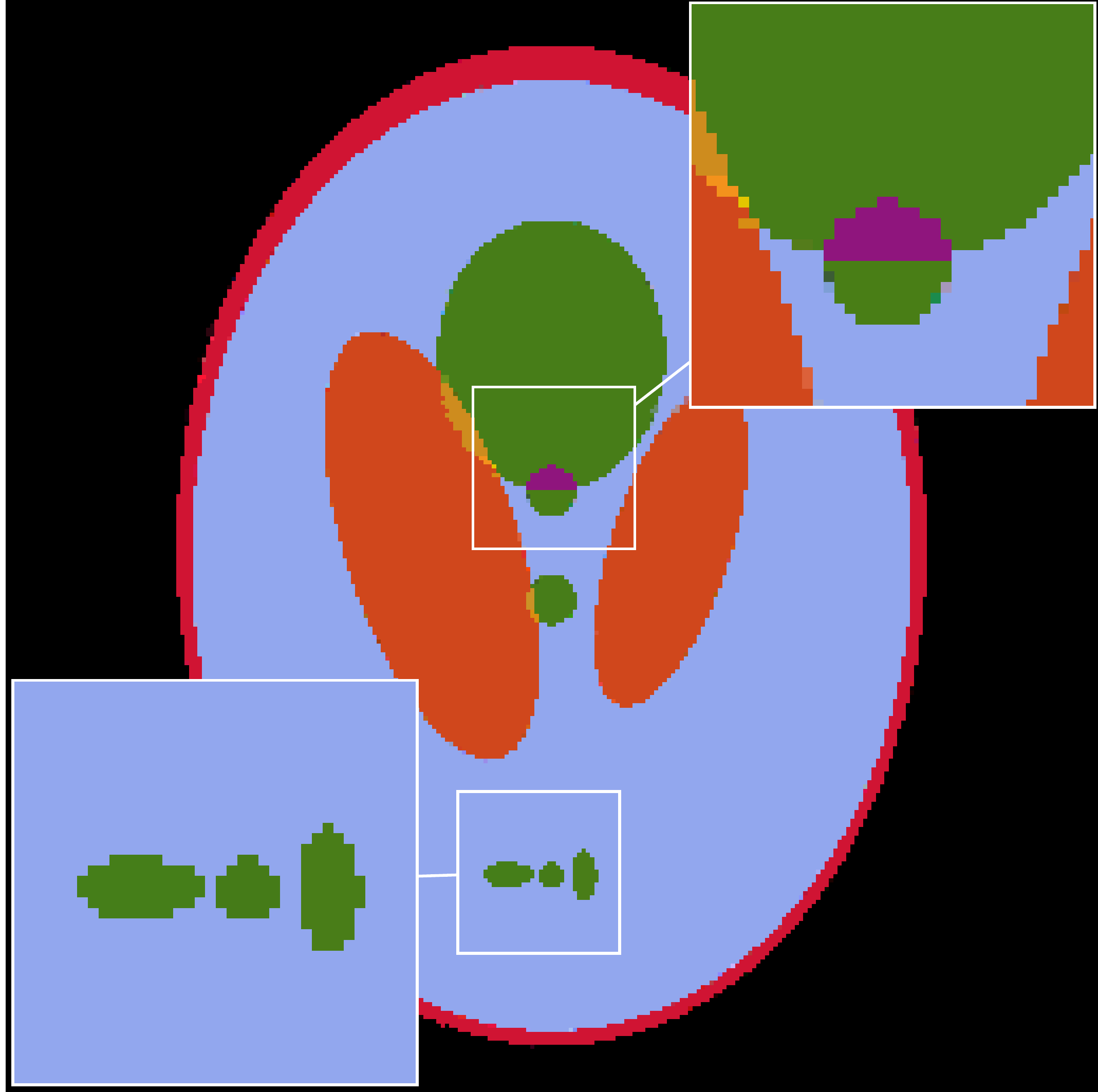}
		\subcaption{Potts S-CG\label{subfig:PottsCG}}
	\end{subfigure}
			\captionsetup{justification=justified}	
			\caption[Potts ADMM and Potts S-CG for undersampled multi-channel Radon data]{
			Potts ADMM and Potts S-CG for undersampled noisy multi-channel Radon data
			(20 angles, $\sigma=0.35$).
			The jumps of all results are spatially aligned across the channels.
			However, the penalty method and Potts S-Landweber produce some spurious artifacts near the segment boundaries which can be seen particularly in the zoomed regions.
			Potts ADMM and Potts S-CG produce improved reconstruction results which are
			close to the ground truth.
			On close inspection, we observe that some of the segment boundaries are briefly smoother in 
			the Potts ADMM result than in the Potts S-CG result (see, e.g., the zoomed regions).\label{fig:PottsADMMvsPottsSCG}}
	~\vs
	\begin{subfigure}{\plotWidth}
		\includegraphics[width=\textwidth]{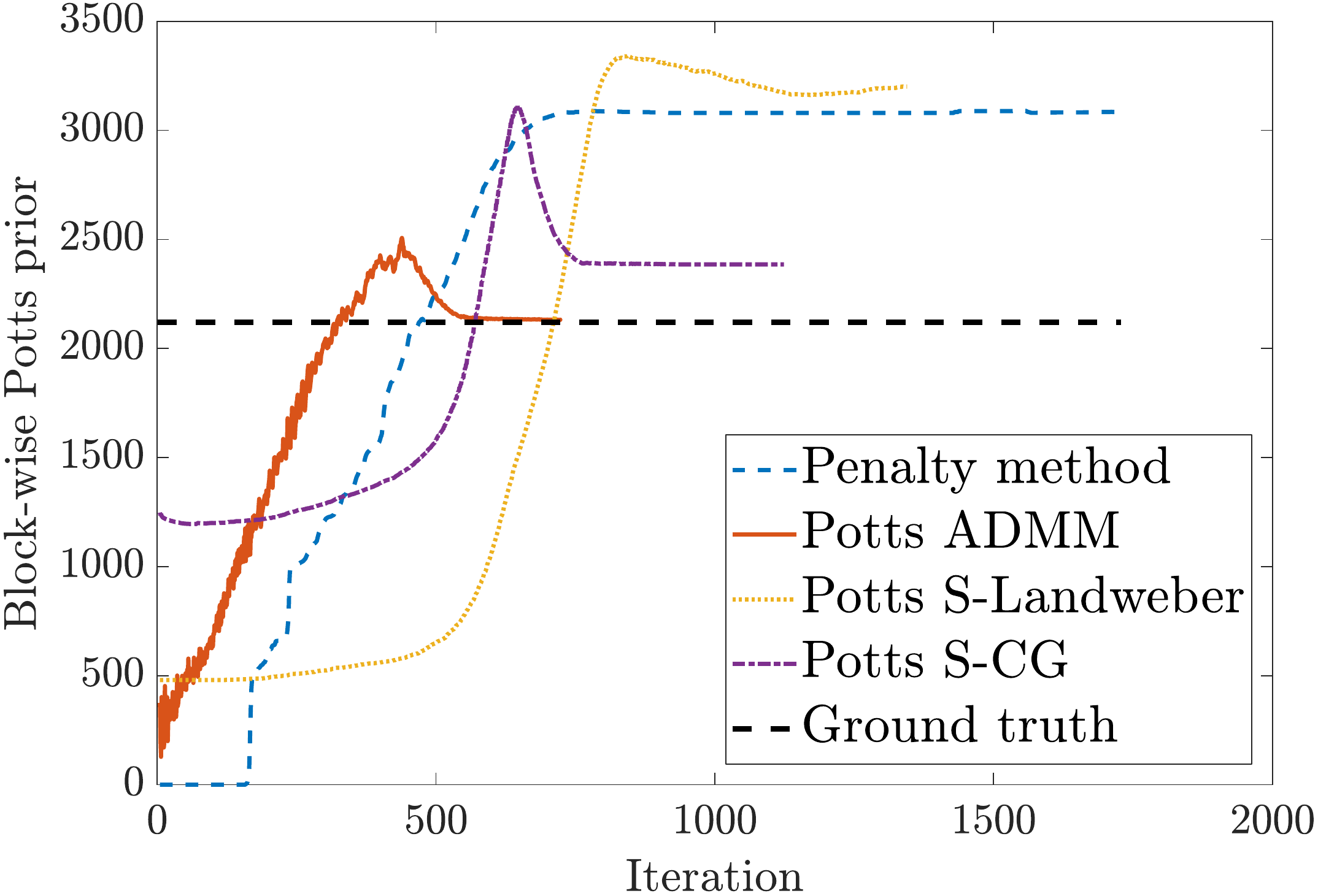}
		\subcaption{Block-wise Potts prior\label{subfig:BlockPottsVsIterations}}
	\end{subfigure}\hss
	\begin{subfigure}{\plotWidth}
		\includegraphics[width=\textwidth]{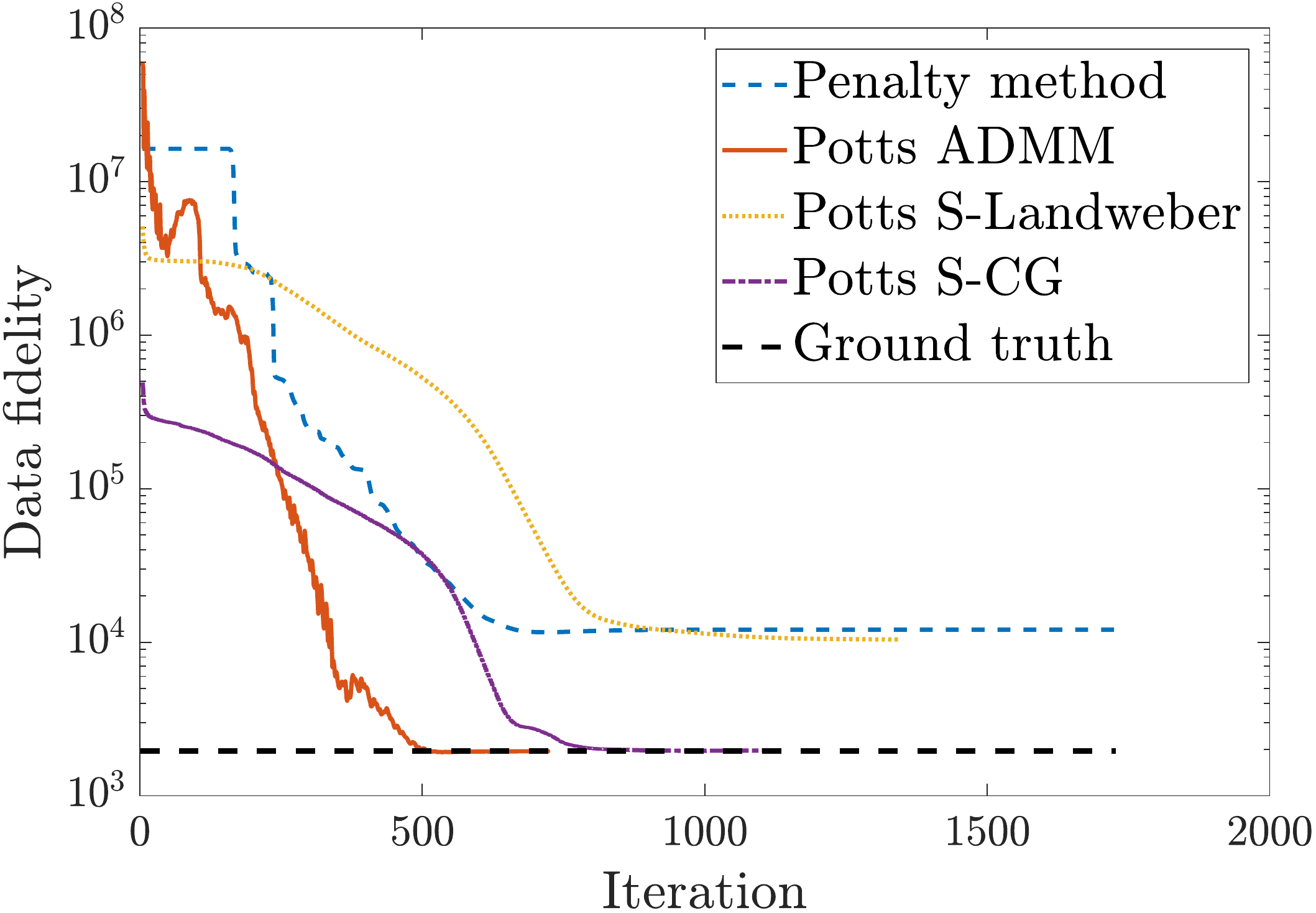}
		\subcaption{Data fidelity\label{subfig:DataVsIterations}}
	\end{subfigure}
	\captionsetup{justification=justified}	
	\caption{
		Values of the block-wise Potts prior \eqref{eq:BlockwisePottsPriorIso} and the data deviations over the iterations
		for the results in \Cref{fig:PottsADMMvsPottsSCG}.
		Potts ADMM and Potts S-CG achieve block-wise Potts prior values 
		and data deviations which are near the respective values of the ground truth.
		This reflects the qualitative findings in \Cref{fig:PottsADMMvsPottsSCG}
		as it shows that the solutions of Potts ADMM and Potts S-CG
		simultaneously agree with the data and minimize the (block-wise) Potts prior.
		\label{fig:PottsADMMvsPottsSCG_Graphs}
	}
\end{figure}

\paragraph{Comparison for multi-channel Radon data.}
Next, we consider the reconstruction of a colored, i.e., three-channel, Shepp-Logan Phantom
from undersampled multi-channel Radon data corrupted
by additive Gaussian noise (20 angles, $\sigma=0.35$).
Further, we study the data deviation,
$\tfrac{1}{S} \sum_{s,c} \| Au^k_{s,c}-f_c\|^2$,
and the block-wise Potts prior,
$\sum_{s} \omega_s \| \nabla_{d_s}u^k_s\|_0$,
as a function of the iteration index.
To put the results into context, we also include the reconstruction result and 
iteration data of a penalty method for the Potts model which 
relaxes the constrained problem \eqref{eq:PottsSplitted} by replacing the constraints 
by a sequence of softened constraints.
To fix ideas, this method corresponds to Potts ADMM with the Lagrange multipliers set to zero.
Furthermore, we include the Potts superiorized Landweber method (Potts S-Landweber), which corresponds
to Potts S-CG with Landweber steps instead of CG steps. 
Recall that a Landweber step corresponds to a gradient descent step w.r.t.\,the (weighted) least squares problem
\eqref{eq:weightedLSmultichannel}.

In \Cref{fig:PottsADMMvsPottsSCG}, we show the results. 
First, we note that the jumps of all results are spatially aligned across
the channels which illustrates the considerations on the
multi-channel Potts prior in \Cref{sec:MultichannelPottsPrior}.
Concerning the reconstruction quality,
we observe that the penalty method as well as Potts S-Landweber produce spurious artifacts near segment boundaries.
Potts ADMM and Potts S-CG yield improved reconstruction results which are
close to the ground truth.
This confirms the theoretical consideration of \Cref{sec:BasicPottsPerturbing}
that the Landweber method is not a favorable basic algorithm in the sense of superiorization,
when large operator norms are involved.
Moreover, the results confirm that the Lagrange multipliers 
in Potts ADMM (the ``memory'' of the iterations) play a crucial role
for the reconstruction quality.
These qualitative findings are reflected by the graphs in \Cref{fig:PottsADMMvsPottsSCG_Graphs}
which show the values of the block-wise Potts prior \eqref{eq:BlockwisePottsPriorIso}
 and the data deviations over the iterations: 
Potts ADMM and Potts S-CG arrive at values which are close to the values of the ground truth.
This illustrates that the results of Potts ADMM and Potts S-CG
agree with the data and at the same time minimize the (block-wise) Potts prior.

\section{Experimental Results}
\label{sec:Experiments}
In this section, we illustrate the potential of the 
Potts prior for multi-spectral computed tomography. 
More precisely, we consider the Potts S-CG algorithm (Algorithm \ref{alg:superiorizedCG})
and the Potts ADMM algorithm (Algorithm \ref{alg:ADMM})
and compare them with existing TV-type approaches.
First, we provide the necessary implementation details of Potts ADMM and
Potts S-CG. Next, we briefly recall the TV, TNV and dTVp prior.
To obtain a qualitative comparison, we consider the geocore phantom used
in \cite{kazantsev2018joint}.
Finally, we give a quantitative comparison of the methods for a phantom
which consists of few organic materials.

\paragraph{Implementation details.}
We provide the necessary implementation details. 
We begin with Potts ADMM (Algorithm \ref{alg:ADMM}).
As coupling sequences, we employ $\rho_k = 10^{-7}\cdot k^{2.01}$, $\mu_k = \rho_k/S$.
For the stopping criterion, we set $\mathtt{tol}=10^{-5}$.
We solve the subproblems \eqref{eq:Vstep} using Matlab's built-in function
$\mathtt{pcg}$ which is applied to the weighted normal equations 
\eqref{eq:weightedNormalEquations} with standard tolerance $10^{-6}$ and a maximum
number of 2000 iterations. We use the previous iteration $v^{k-1}$ as initial guess so that
after the first ADMM iterations only a few $\mathtt{pcg}$ iterations are needed.
Regarding Potts S-CG (Algorithm \ref{alg:superiorizedCG}), we used the annealing parameter $a=0.999$ and the initial
coupling $\mu_0 = 10^{-4}$.
For the stopping criterion, we set $\mathtt{tol}=10^{-5}$.
The ADMM subproblems \eqref{eqUstep} in Potts ADMM and 
the evaluation of the proximal mappings \eqref{eq:BlockwisePottsProxIso}
in Potts S-CG are solved by using the Potts solver provided by the
Pottslab software toolbox
which makes use of parallelization for multicore CPUs 
(\url{https://github.com/mstorath/Pottslab}).
For 3D-data and large problem sizes an implementation using the GPU becomes reasonable. 
However, such an implementation is beyond the scope of this paper.

\paragraph{Methods for comparison.}
\begin{table}[t]
	\centering
	\def\TableRowSpace{\\[0.5em]}
	\begin{tabular}{lcl}
		\toprule
		Method & Initialization & Stopping criterion \\
		\midrule
		FISTA TV &0 & number of iterations: $k=\textrm{maxIter}$\TableRowSpace
		FISTA TNV &0 & number of iterations: $k=\textrm{maxIter}$\TableRowSpace
		FISTA dTVp &0 & number of iterations: $k=\textrm{maxIter}$\TableRowSpace
		Potts ADMM &0 & distance of variables: $\| u^k_s-u^k_{s+1}\|_\infty < \mathtt{tol}$ \TableRowSpace
		Potts S-CG & $A^TWf$ & distance of variables: $\| u^k_s-u^k_{s+1}\|_\infty < \mathtt{tol}$ \\
		\bottomrule
	\end{tabular}
	\caption{List of initializations and stopping criteria of the considered methods.
		In our experiments, we used the parameters $\textrm{maxIter}=250$, see \cite{kazantsev2018joint},
		and $\mathtt{tol}=10^{-5}$. The methods are quantitatively compared in terms of the achieved
		mean structural similarity index (MSSIM) \eqref{eq:MSSIM}.
		\label{tab:algorithmicParameters}}
\end{table}
We compare the Potts ADMM and the Potts S-CG method with classical channel-wise
total variation (TV), channel-coupling total nuclear variation (TNV)
\cite{rigie2015joint}
and the probabilistic directional TV method (dTVp), which was proposed in 
\cite{kazantsev2018joint}. 
Channel-wise TV corresponds to the $\ell_1$-norm of the discrete gradient 
$\|\nabla u\|_1$ for each channel separately.
Its application amounts to solving the convex problem given by
\begin{equation}
\label{eq:TVmodel}
\argmin_u\sum_{c=1}^{C} \| Au_c - f_c\|_W + \alpha \|\nabla u_c \|_1
\end{equation}
for a parameter $\alpha > 0$.
The TNV regularizer uses the nuclear norm of the discrete Jacobian in each pixel
of a multi-channel image to correlate the channels. Applying TNV
corresponds to the convex problem
\begin{equation}
\label{eq:TNVmodel}
\argmin_u \sum_{c=1}^{C} \|Au_c - f_c\|_W + \sum_{x\in\Omega}\alpha
\|Du(x) \|_\ast,
\end{equation}
where $\alpha >0$ is a parameter, $Du(x)$ is the $2\times C$ matrix of channel-wise finite differences
of $u$ at the point $x$ and $\| \cdot \|_\ast$ denotes its nuclear norm, i.e.,
the sum of its singular vales.
The dTVp regularizer is a modification of the 
directional TV (dTV) regularizer. 
The dTV regularizer enforces correlation between a channel image $w$ and a reference channel $v$
by means of the directional diffusion of the channel given the known reference channel.
To fix ideas, we denote by 
$\mathrm{vec}(w)\in\R^{n^2}$ the column-wise vectorization of $w$ and by
$D_x$ the $2\times n^2$-matrix such that $D_x \mathrm{vec}(w)$ is the finite differences 
vector of $w$ at the point $x$.
Then, the dTV regularizer is given by 
\begin{equation}
\label{eq:dTV}
\mathrm{dTV}(w,v)=\sum_{x\in \Omega'} \|P_v D_x\mathrm{vec}(w)\|_2, \quad
P_v = \begin{cases}
I-\frac{D_x\mathrm{vec}(v)\mathrm{vec}(v)^T D_x^T}{\|D_x \mathrm{vec}(z)\|^2} &\text{ if } D_x\mathrm{vec}(v) \neq 0, \\
I &\text{ if } D_x\mathrm{vec}(v) = 0
\end{cases}
\end{equation}
for a single-channel image $w$ and reference $v$.
The multi-channel version of \eqref{eq:dTV} is given by 
\begin{equation}\label{eq:dTVmultichannel}
\sum_{c=1}^{C} \mathrm{dTV}(u_c,v_c),
\end{equation}
where $v_c$ denote the reference channel for $u_c$.
In contrast to dTV, 
for dTVp the reference channels are not fixed, 
but chosen probabilistically in each iteration of the reconstruction process.
The reference channels are chosen among the channels of the former iterate of $u$.
This process uses a probability mass function, which is based on the channel-wise geometric 
mean of the estimated signal-to-noise ratios (SNR). Hence, 
channels with high SNR are more likely to become a reference channel.
We note that the dTVp regularizer is non-convex \cite{kazantsev2018joint}.

We used the implementation of \cite{kazantsev2018joint},
which employs the FISTA method 
(fast iterative shrinkage-thresholding) \cite{beck2009fast}
to approach the respective minimization problems
(\url{https://github.com/dkazanc/ multi-channel-X-ray-CT}).
The FISTA method converges to the global minimizer for the convex regularizers TV and TNV.
However, this does not necessarily hold for dTVp which is non-convex.
In \Cref{tab:algorithmicParameters}, we list the considered methods together 
with their respective initializations and stopping criteria.

\begin{figure}[t]
	\centering
	\captionsetup[subfigure]{justification=centering}
	\def\figwidth{0.2\textwidth}
	\def\hs{\hspace{1.5em}}
	\begin{subfigure}{\figwidth}
		\includegraphics[width=\textwidth]{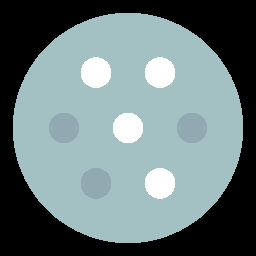}
		\caption*{Ground truth\\~\\~\\~}
	\end{subfigure}\hs
	\begin{subfigure}{\figwidth}
		\includegraphics[width=\textwidth]{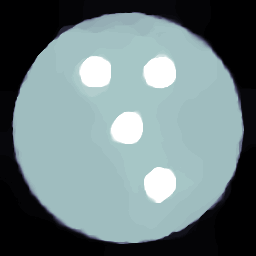}
		\caption*{TV,\\
			RMSE=\protect\input{Experiments/OrganicSpheres/rmse_TV_3.txt}\unskip,\\
			MAE=\protect\input{Experiments/OrganicSpheres/mae_TV_3.txt}\unskip,\\
			MSSIM=\protect\input{Experiments/OrganicSpheres/ssim_TV_3.txt}
			}
	\end{subfigure}\hs
	\begin{subfigure}{\figwidth}
		\includegraphics[width=\textwidth]{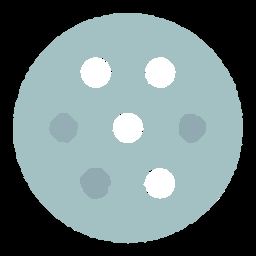}
		\caption*{Potts ADMM,
			RMSE=\protect\input{Experiments/OrganicSpheres/rmse_admm_3.txt}\unskip,\\
			MAE=\protect\input{Experiments/OrganicSpheres/mae_admm_3.txt}\unskip,\\
			MSSIM=\protect\input{Experiments/OrganicSpheres/ssim_admm_3.txt}
		}
	\end{subfigure}
	\caption{\label{fig:RMSElimitations}Comparison of error measures for the reconstruction of a compound body (excerpt from \Cref{fig:OrganicPhantomReconstructions}). 
	Even though TV smooths out some of the spheres, it achieves lower RMSE than Potts ADMM 
	which preserves these spheres. 
	The larger RMSE of Potts ADMM 
	may be attributed to few outliers at the boundaries of the white spheres. 
	In comparison, Potts ADMM achieves lower MAE and higher, i.e., more preferable, MSSIM than TV. This illustrates the sensitivity of RMSE to outliers in the reconstruction.
}
\end{figure}

\paragraph{Assessment of the reconstruction quality.}
We discuss quality measures for reconstructions of compound bodies.
A standard choice corresponds to the root mean square error (RMSE).
For $u,v\in\R^{n\times n}$ it is given by 
\begin{equation}
\textrm{RMSE}(u,v) = \sqrt{\frac{1}{n^2}\sum_{j=1}^{n^2}|u(j)-v(j)|^2}\cdot 100,
\end{equation}
where we use the factor 100 for scaling.
However, it is well-known that error measures based on the $\ell_2$-distance 
such as the RMSE are sensitive to outliers which may only mildly
affect the reconstruction quality.
Outliers consisting of only few pixels can often be removed by a basic post-processing, e.g. 
median filtering.
An alternative distance that is less sensitive to outliers
corresponds to the mean absolute error (MAE) which
is given by 
\begin{equation}
\textrm{MAE}(u,v) = \frac{1}{n^2}\sum_{j=1}^{n^2} |u(j)-v(j)|\cdot100, 
\end{equation}
where the factor 100 is again used for scaling. 
A quality measure which was developed to reflect the visual similarity between images
better than classical pointwise measures
is the mean structural similarity index (MSSIM) \cite{wang2004image}.
It is given by 
\begin{equation}\label{eq:MSSIM}
\textrm{MSSIM}(u,v) = \frac{1}{n^2}\sum_{x\in\Omega'}
\frac{(2\mu_u\mu_v +C_1)(2\sigma_{uv}+C_2)}
{(\mu_u^2 + \mu_v^2+C_1)(\sigma_u^2 + \sigma_v^2+C_2)}
\end{equation}
for the mean intensity $\mu$ and standard deviation $\sigma$ 
of a $11\times 11$-window centered at pixel $x$ weighted by a circular-symmetric Gaussian weighting function with standard deviation 1.5 normalized to unit sum.
The constants are set to $C_1 = 10^{-4}, C_2 = 9\cdot 10^{-4}$.
The MSSIM is bounded from above by 1 and higher values are preferable. 
We use the Matlab implementation of the MSSIM.
Please note that for comparing piecewise constant images, the Rand index 
would be a suitable quality measure as well
since it quantifies directly the similarity between clusters/segmentations
\cite{rand1971objective,arbelaez2010contour}.
However, as TV-based priors typically yield results which are not piecewise constant, 
we refrain from using the Rand index.

In \Cref{fig:RMSElimitations}, we compare RMSE, MAE and MSSIM for the reconstruction of
a piecewise constant image with TV and Potts ADMM. 
We note that TV smoothed out whole regions of the ground truth, while Potts ADMM reconstructs these 
regions and distinguishes them sharply from the surrounding region. 
Nevertheless TV achieves a lower RMSE than Potts ADMM.
On the other hand, Potts ADMM achieves lower MAE and higher MSSIM, respectively, 
than TV as MAE and MSSIM are less sensitive to outliers. 
We attribute the higher RMSE of Potts ADMM 
to only few outliers at the boundaries of some of the white segments. 
Thus, the MSSIM seems to be more suitable than RMSE for comparing reconstruction qualities
in the present setup.
However, we report the MSSIM, MAE and RMSE in our experiments.

\paragraph{Geocore phantom.}
We start out with a qualitative comparison w.r.t.\,the geocore 
phantom of \cite{kazantsev2018joint}.
The phantom has size $512\times 512$ and
consists of four materials: quartz, pyrite, galena, and gold; see \Cref{fig:geocorePhantomMaterials}.
The width of the domain is set 1\,cm, the distance from the X-ray source 
to the rotation center to 3\,cm, the distance from the source to the detector is 5\,cm
and the width of the detector array is 2\,cm.
The X-ray spectrum was simulated with the spektr software package 
\cite{siewerdsen2004spektr}.
The energy spectrum of the emitted photons is shown in \Cref{fig:multiCT}
for the used tube potential (120 kVp) and
photon flux $I_0 = 4\cdot 10^4$.
The photon attenuation coefficients for the materials were obtained 
from the PhotonAttenuation software package \cite{tusz2020photonAtt}.
The energy range is chosen as 45-114 keV and is discretized into 70 energy bins.
The multi-spectral measurements were simulated with the Astra-Toolbox
\cite{van2015astra}. In particular, we set the fan-beam scanning geometry 
and obtained the measurements according to 120 projection angles
from a larger version of the phantom ($1024\times 1024$).
For TV, TNV and dTVp, we used the model parameters as given in \cite{kazantsev2018joint}, i.e.,
$\alpha_{TV}=4.2\cdot10^{-4}$,\,
$\alpha_{TNV}=1.2\cdot10^{-3}$ and 
$\alpha_{dTVp}=3.3\cdot 10^{-3}$.
Potts ADMM used the jump penalty
$\gamma=0.075
$ and Potts S-CG 
the initial perturbation parameter
$\beta_0=500$.

In \Cref{fig:geocorePhantom},
we show the reconstructions for three depicted channels.
We observe that Potts ADMM and Potts S-CG produce sharp boundaries across the channels, while also
showing considerably fewer artifacts than the methods of comparison.
These qualitative findings are confirmed by the corresponding MSSIM values 
(see also \Cref{fig:multiCT})
which are higher for Potts ADMM and Potts S-CG for the majority 
of the channels.
\begin{figure}[t]
	\centering
	\captionsetup[subfigure]{justification=centering}
	\def\figwidth{0.24\textwidth}
	\def\vs{\\[0.5em]}
	\def\hs{\hfill}
	\def\hh{\textwidth}
	\begin{subfigure}{\figwidth}
		\includegraphics[width=\textwidth,height=\hh]{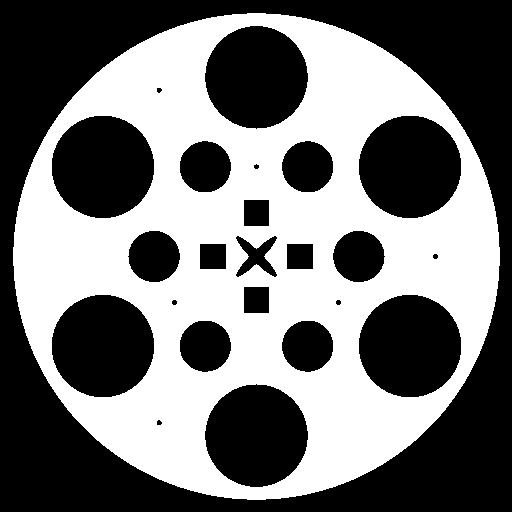}
		\subcaption{Quartz}
	\end{subfigure}\hs
	\begin{subfigure}{\figwidth}
		\includegraphics[width=\textwidth,height=\hh]{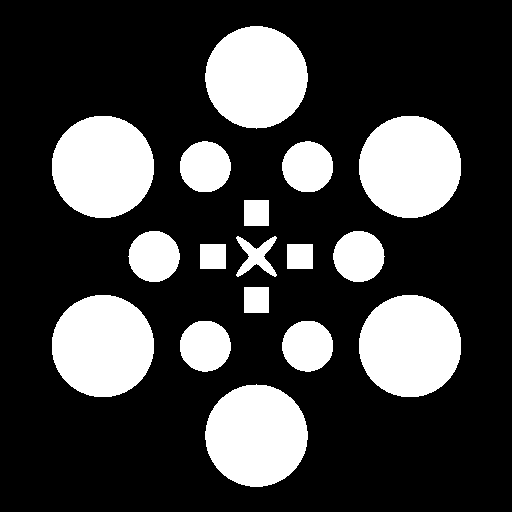}
		\subcaption{Pyrite}
	\end{subfigure}\hs
	\begin{subfigure}{\figwidth}
		\includegraphics[width=\textwidth,height=\hh]{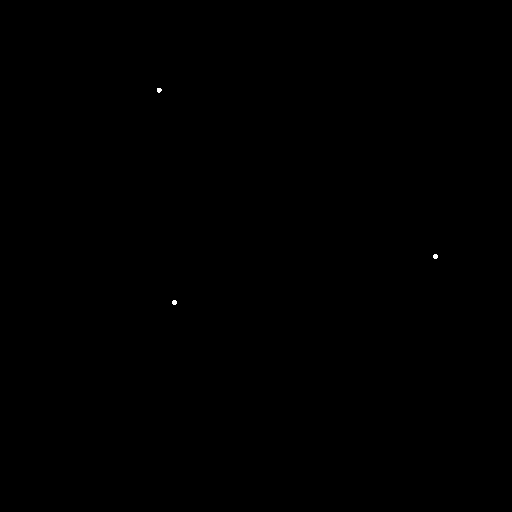}
		\subcaption{Galena}
	\end{subfigure}\hs
	\begin{subfigure}{\figwidth}
		\includegraphics[width=\textwidth,height=\hh]{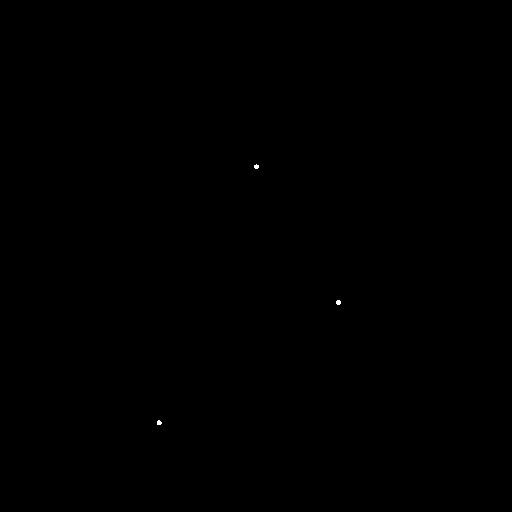}
		\subcaption{Gold}
	\end{subfigure}
	\caption{\label{fig:geocorePhantomMaterials}The geocore phantom consists of four inorganic materials:
		(a) quartz, (b) pyrite, (c) galena and (d) gold. The background is air.}
	~\vs
	\def\figwidth{0.48\textwidth}\centering
	\begin{subfigure}[b]{\figwidth}
		\includegraphics[width=\textwidth]{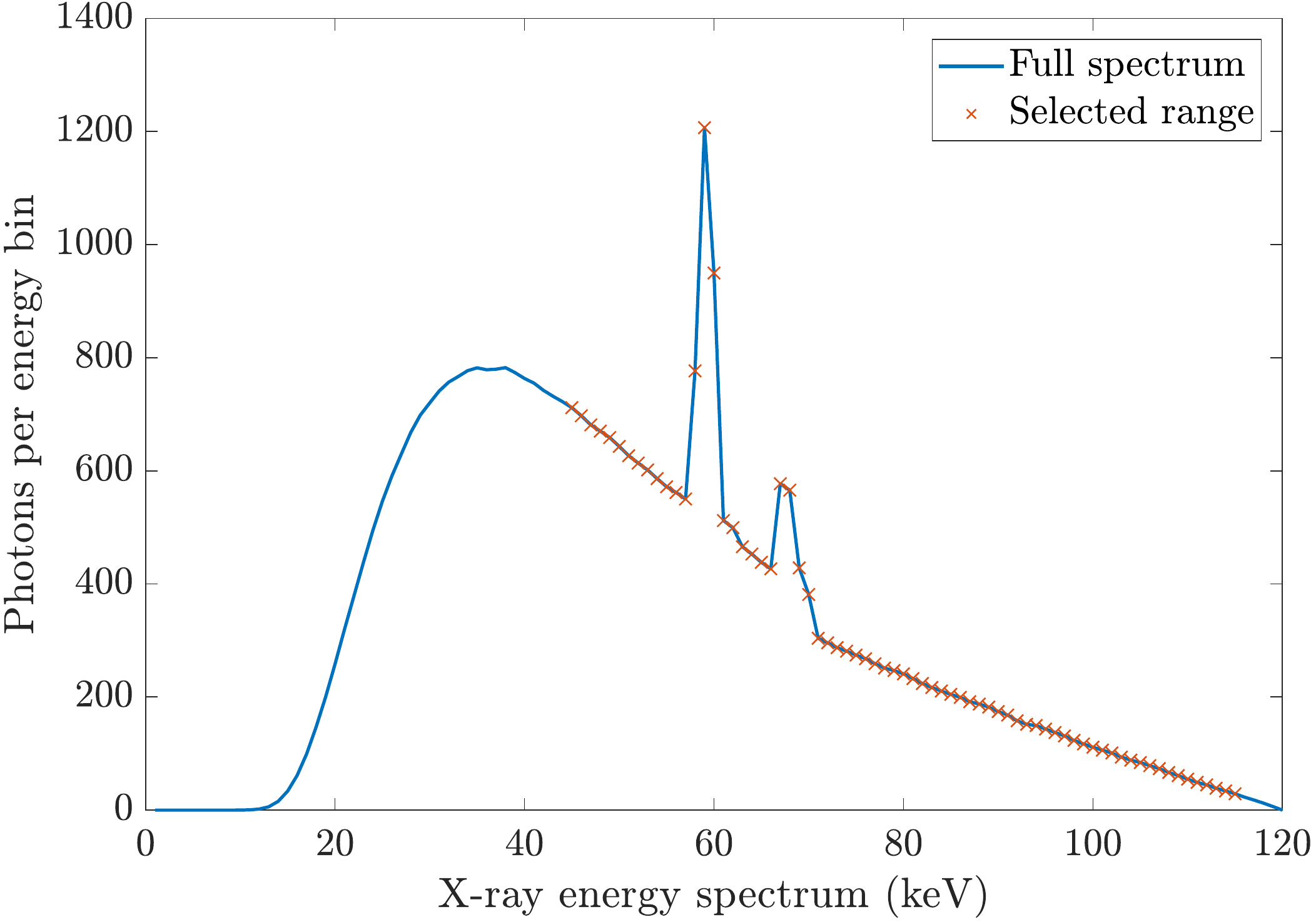}
	\end{subfigure}\hfill
	\begin{subfigure}[b]{\figwidth}	
		\includegraphics[width=\textwidth]{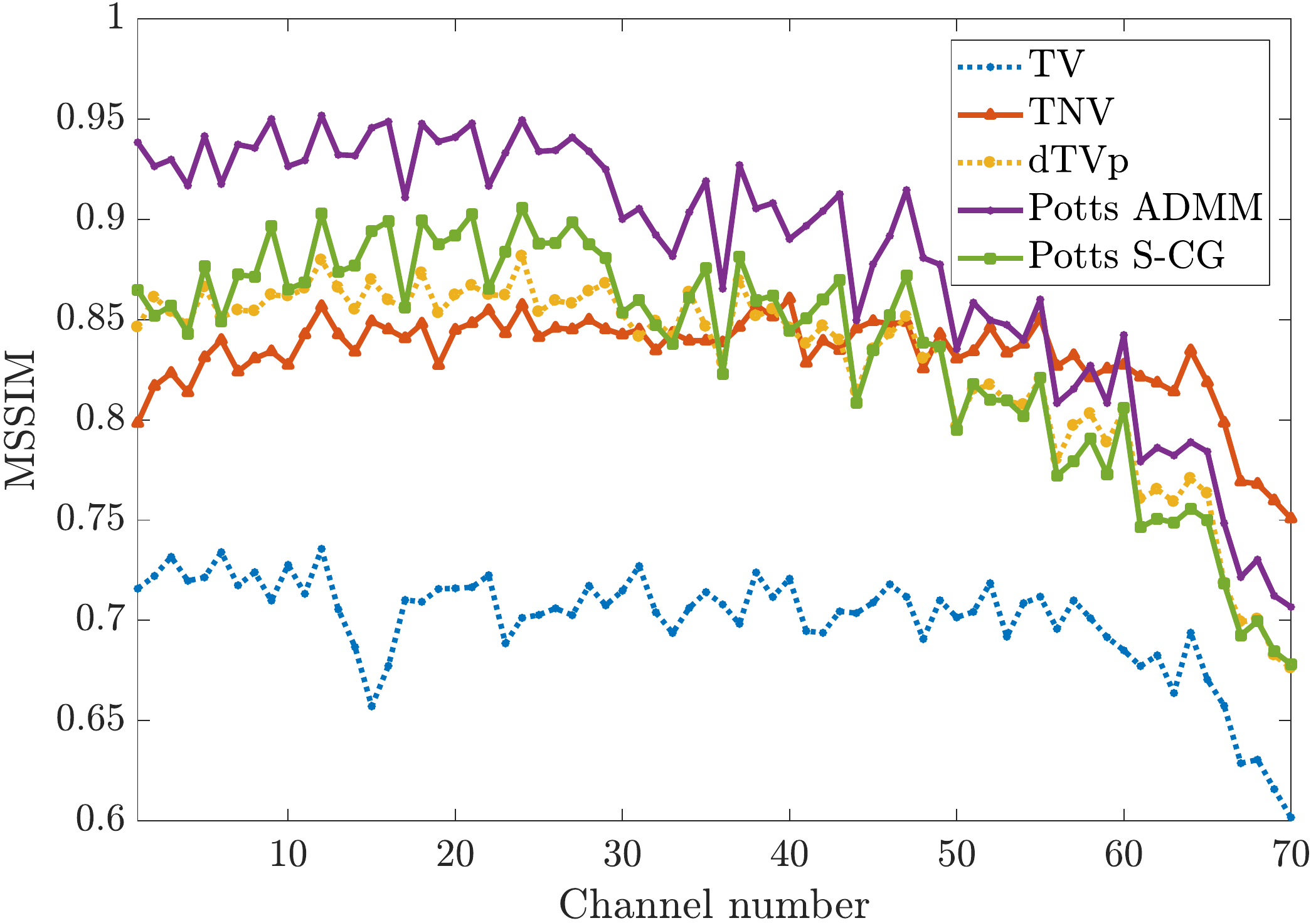}
	\end{subfigure}\\~
	\begin{subfigure}[b]{0.4675\textwidth}	
		\includegraphics[width=\textwidth]{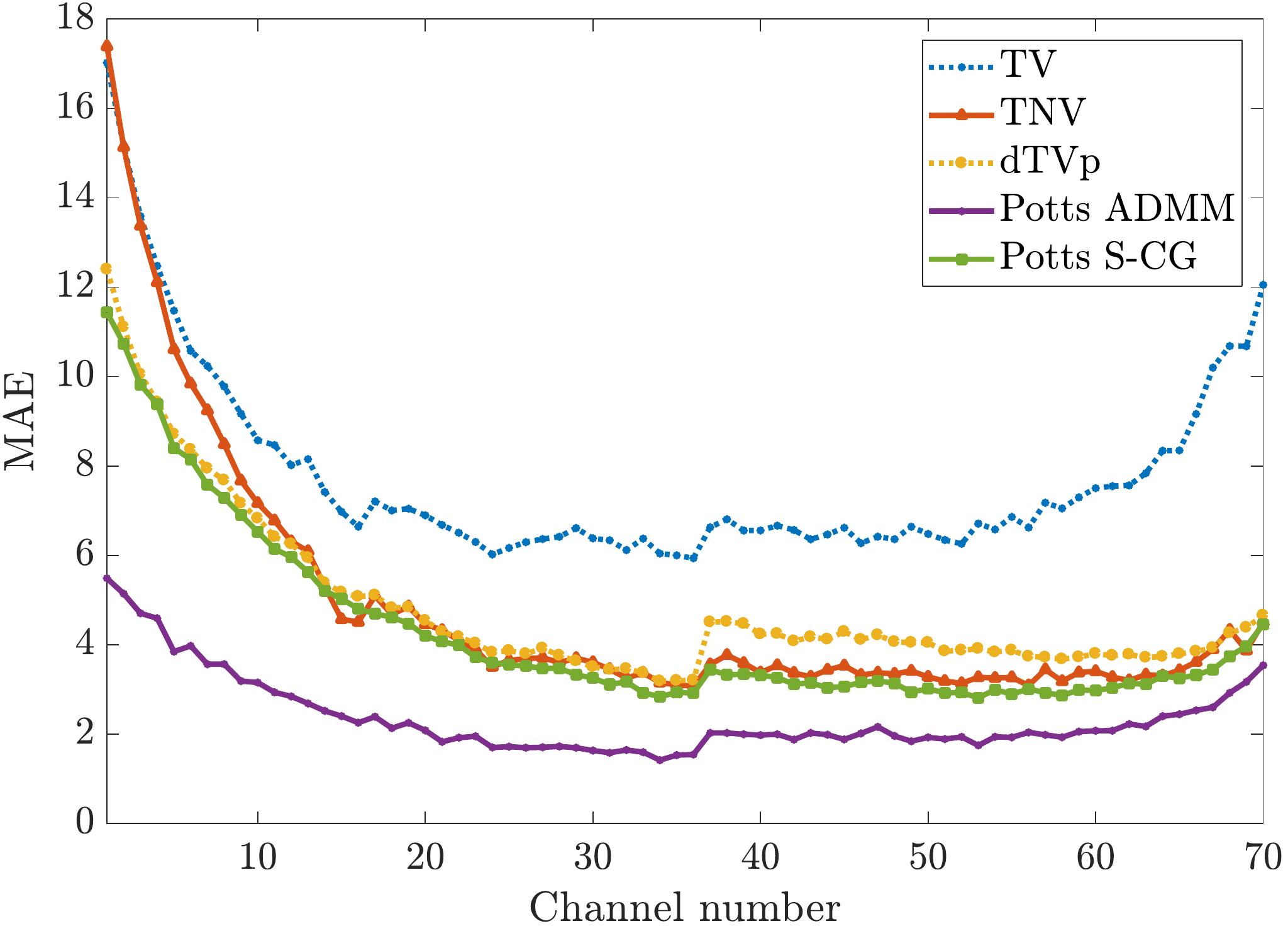}
	\end{subfigure}\hfill
	\begin{subfigure}[b]{\figwidth}	
		\includegraphics[width=\textwidth]{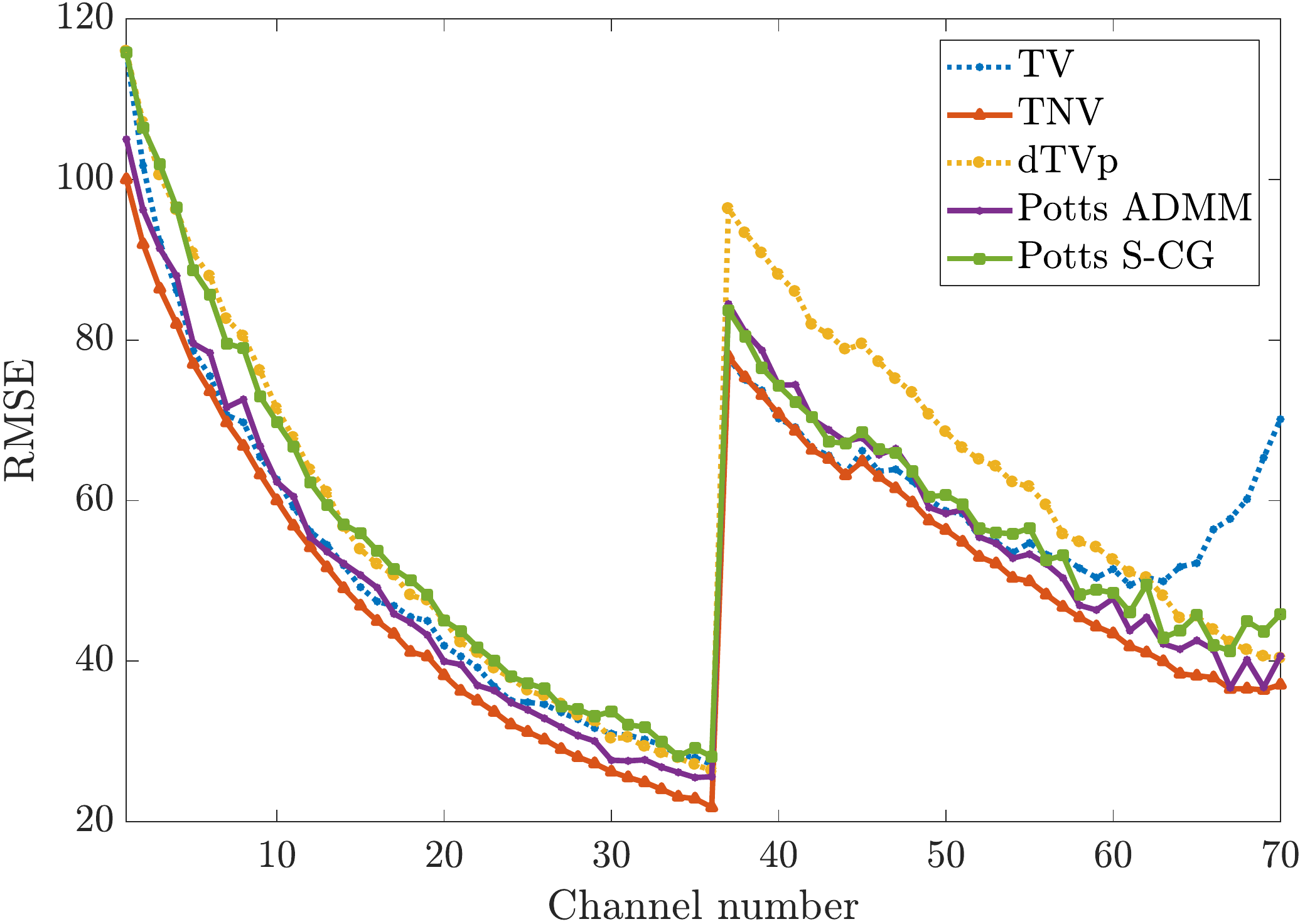}
	\end{subfigure}
	\caption{\label{fig:multiCT}
		\textit{Top left}: The full X-ray spectrum and the range of energies 
		(45-114\,keV) used
		for modeling the measurements of the geocore phantom. The spectrum is discretized into 70 energy bins.
		\textit{Top right}: MSSIM of the reconstructions of 
		the geocore phantom for all energy bins.  
		The channel-coupling methods have larger MSSIM values than channel-wise TV.
		For the majority of the energies
		the result of Potts ADMM and Potts S-CG have higher MSSIM values than
		the TV-type regularizers dTVp and TNV
		(Potts ADMM: 57 channels, Potts S-CG: 38 channels).
		\textit{Bottom: } MAE and RMSE of the reconstructions of the geocore phantom.
			For the highest energies,
			which are particulary prone to noise, the channel-coupling methods have 
			considerably lower MAE and RMSE than channel-wise TV.
			Potts ADMM and Potts S-CG achieve the lowest MAE for most of the energies, while 
			the RMSE of the methods are close for most energies.
			(The jump in the RMSE at the 36th channel may be attributed to a sudden increase of the LAC of gold
			at roughly 80\,keV.)
	}
\end{figure}
\begin{figure}[hp]
	\centering
	\captionsetup[subfigure]{justification=centering}
	\def\figwidth{0.17\textwidth}
	\def\vs{\\[0.2em]}
	\def\hs{\hspace{0.75em}}
	\def\hh{\textwidth}
	\def\rb{28pt}
	\begin{subfigure}{\figwidth}
		\caption*{\textbf{45 keV}}
		\makebox[0pt][r]{\makebox[1.25em]{\raisebox{\rb}{\rotatebox[origin=c]{90}
					{\scriptsize \textbf{Ground truth}}}}}%
		{\includegraphics[width=\textwidth,height=\hh]{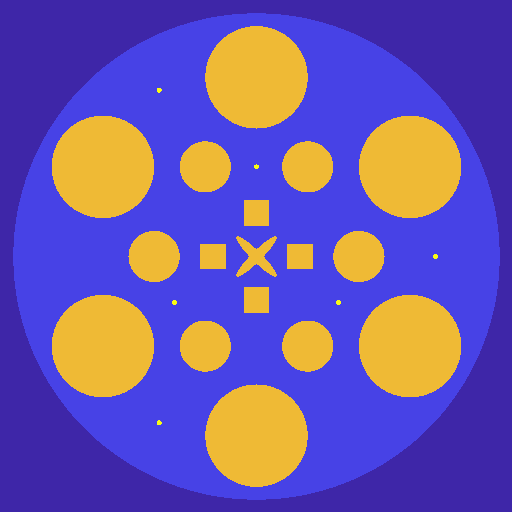}}
	\end{subfigure}\hs
	\begin{subfigure}{\figwidth}
		\caption*{\textbf{80 keV}}
		\includegraphics[width=\textwidth,height=\hh]{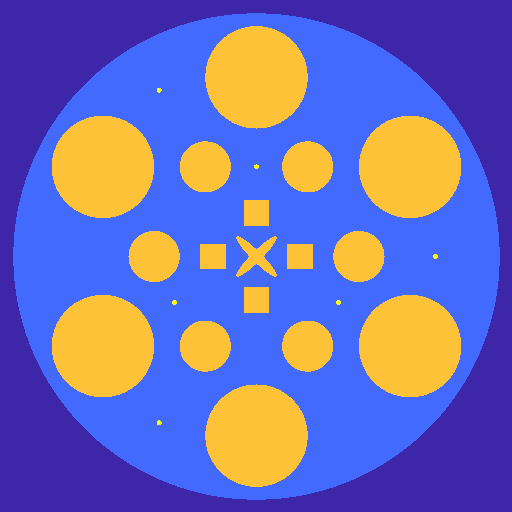}
	\end{subfigure}\hs
	\begin{subfigure}{\figwidth}
		\caption*{\textbf{114 keV}}
		\includegraphics[width=\textwidth,height=\hh]{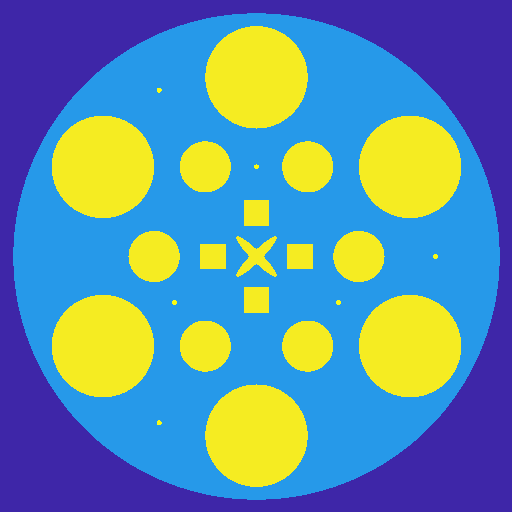}
	\end{subfigure}\\[1.25em]
	\begin{subfigure}{\figwidth}
		\makebox[0pt][r]{\makebox[1.25em]{\raisebox{\rb}{\rotatebox[origin=c]{90}
					{\scriptsize \textbf{TV}}}}}%
		\includegraphics[width=\textwidth,height=\hh]{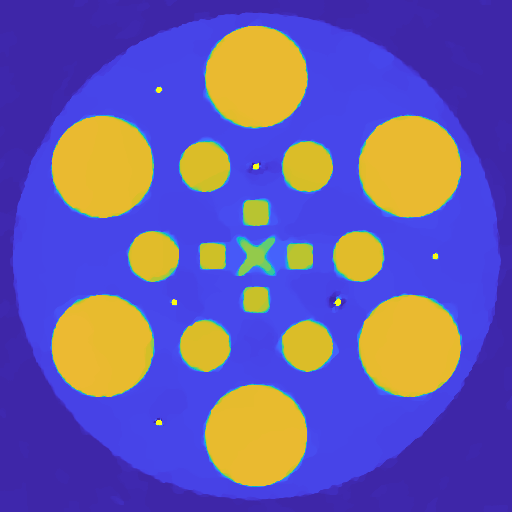}
		\subcaption*{MSSIM=\protect\input{Experiments/GeocorePhantom/ssim_tv_1.txt}
			\unskip,
			RMSE=\protect\input{Experiments/GeocorePhantom/rmse_tv_1.txt}\unskip,
			MAE=\protect\input{Experiments/GeocorePhantom/mae_tv_1.txt}
		}
	\end{subfigure}\hs
	\begin{subfigure}{\figwidth}
		\includegraphics[width=\textwidth,height=\hh]{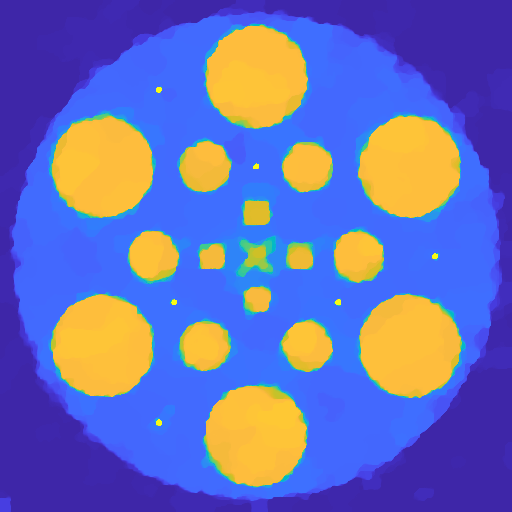}
		\subcaption*{MSSIM=\protect\input{Experiments/GeocorePhantom/ssim_tv_2.txt}
			\unskip,
			RMSE=\protect\input{Experiments/GeocorePhantom/rmse_tv_2.txt}\unskip,
			MAE=\protect\input{Experiments/GeocorePhantom/mae_tv_2.txt}
		}
	\end{subfigure}\hs
	\begin{subfigure}{\figwidth}
		\includegraphics[width=\textwidth,height=\hh]{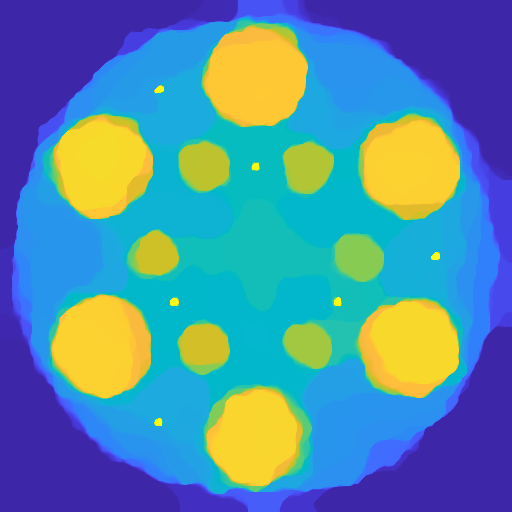}
		\subcaption*{MSSIM=\protect\input{Experiments/GeocorePhantom/ssim_tv_3.txt}
			\unskip,
			RMSE=\protect\input{Experiments/GeocorePhantom/rmse_tv_3.txt}\unskip,
			MAE=\protect\input{Experiments/GeocorePhantom/mae_tv_3.txt}
		}
	\end{subfigure}\vs
	\begin{subfigure}{\figwidth}
		\makebox[0pt][r]{\makebox[1.25em]{\raisebox{\rb}{\rotatebox[origin=c]{90}
					{\scriptsize \textbf{TNV}}}}}%
		\includegraphics[width=\textwidth,height=\hh]{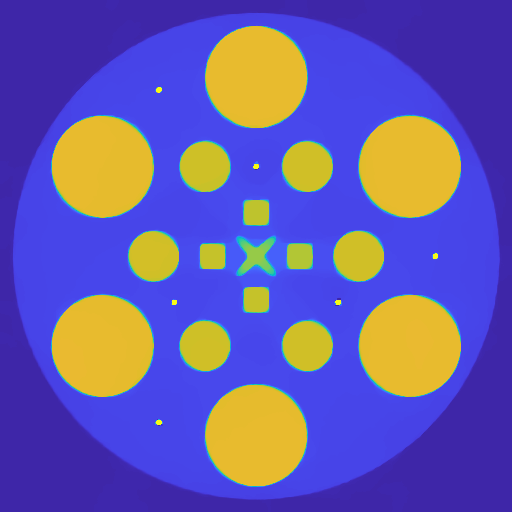}
		\subcaption*{MSSIM=\protect\input{Experiments/GeocorePhantom/ssim_tnv_1.txt}
			\unskip,
			RMSE=\protect\input{Experiments/GeocorePhantom/rmse_tnv_1.txt}\unskip,
			MAE=\protect\input{Experiments/GeocorePhantom/mae_tnv_1.txt}
		}
	\end{subfigure}\hs
	\begin{subfigure}{\figwidth}
		\includegraphics[width=\textwidth,height=\hh]{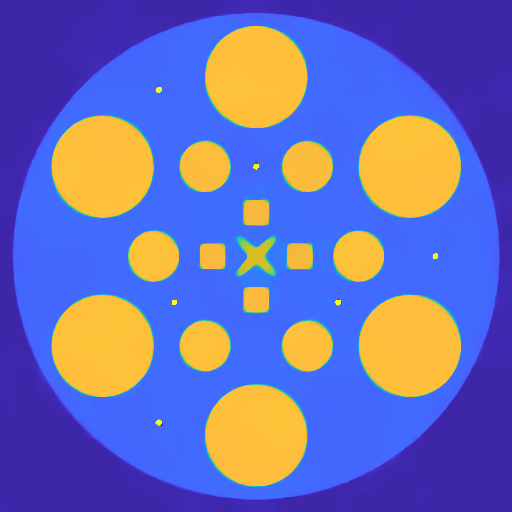}
		\subcaption*{MSSIM=\protect\input{Experiments/GeocorePhantom/ssim_tnv_2.txt}
			\unskip,
			RMSE=\protect\input{Experiments/GeocorePhantom/rmse_tnv_2.txt}\unskip,
			MAE=\protect\input{Experiments/GeocorePhantom/mae_tnv_2.txt}
		}
	\end{subfigure}\hs
	\begin{subfigure}{\figwidth}
		\includegraphics[width=\textwidth,height=\hh]{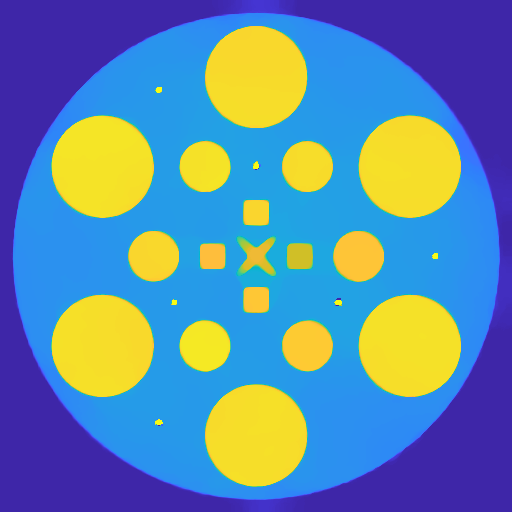}
		\subcaption*{MSSIM=\protect\input{Experiments/GeocorePhantom/ssim_tnv_3.txt}
			\unskip,
			RMSE=\protect\input{Experiments/GeocorePhantom/rmse_tnv_3.txt}\unskip,
			MAE=\protect\input{Experiments/GeocorePhantom/mae_tnv_3.txt}
		}
	\end{subfigure}\vs
	\begin{subfigure}{\figwidth}
		\makebox[0pt][r]{\makebox[1.25em]{\raisebox{\rb}{\rotatebox[origin=c]{90}
					{\scriptsize \textbf{dTVp}  }}}}%
		\includegraphics[width=\textwidth,height=\hh]{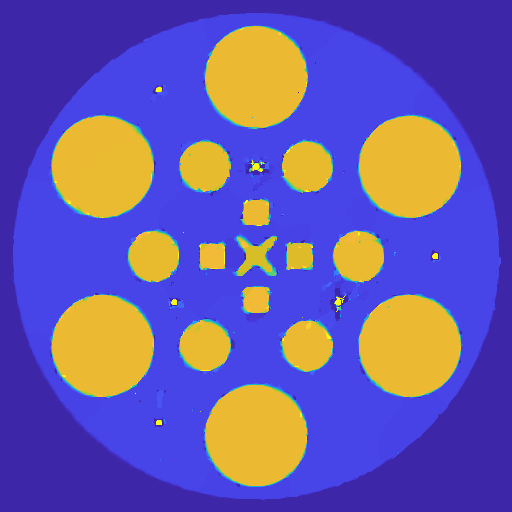}
		\subcaption*{MSSIM=\protect\input{Experiments/GeocorePhantom/ssim_dtvp_1.txt}
			\unskip,
			RMSE=\protect\input{Experiments/GeocorePhantom/rmse_dtvp_1.txt}\unskip,
			MAE=\protect\input{Experiments/GeocorePhantom/mae_dtvp_1.txt}
		}
	\end{subfigure}\hs
	\begin{subfigure}{\figwidth}
		\includegraphics[width=\textwidth,height=\hh]{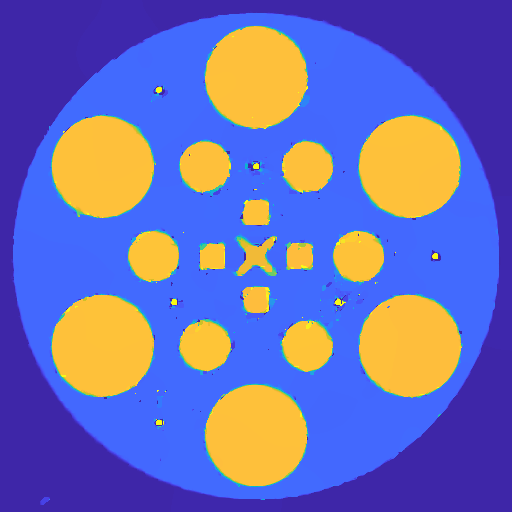}
		\subcaption*{MSSIM=\protect\input{Experiments/GeocorePhantom/ssim_dtvp_2.txt}
			\unskip,
			RMSE=\protect\input{Experiments/GeocorePhantom/rmse_dtvp_2.txt}\unskip,
			MAE=\protect\input{Experiments/GeocorePhantom/mae_dtvp_2.txt}
		}
	\end{subfigure}\hs
	\begin{subfigure}{\figwidth}
		\includegraphics[width=\textwidth,height=\hh]{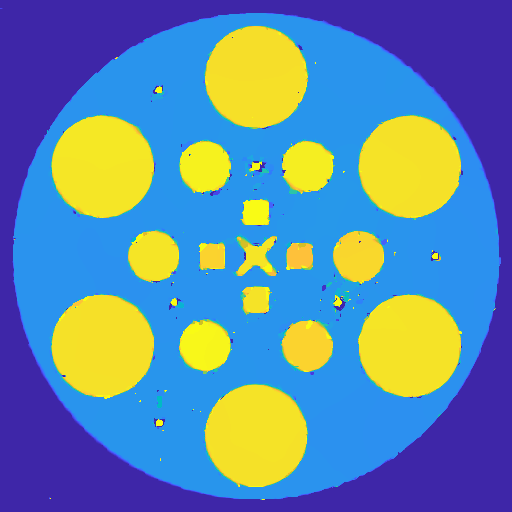}
		\subcaption*{MSSIM=\protect\input{Experiments/GeocorePhantom/ssim_dtvp_3.txt}
			\unskip,
			RMSE=\protect\input{Experiments/GeocorePhantom/rmse_dtvp_3.txt}\unskip,
			MAE=\protect\input{Experiments/GeocorePhantom/mae_dtvp_3.txt}
		}
	\end{subfigure}\vs
	\begin{subfigure}{\figwidth}
		\makebox[0pt][r]{\makebox[1.25em]{\raisebox{\rb}{\rotatebox[origin=c]{90}
					{\scriptsize\textbf{Potts ADMM}}}}}%
		\includegraphics[width=\textwidth,height=\hh]{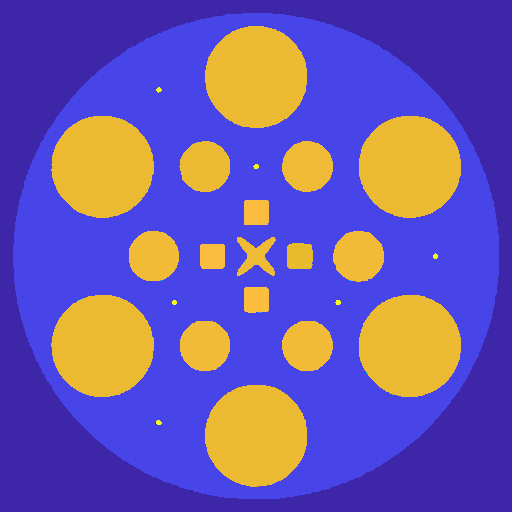}
		\subcaption*{MSSIM=\protect\input{Experiments/GeocorePhantom/ssim_ours_1.txt}
			\unskip,
			RMSE=\protect\input{Experiments/GeocorePhantom/rmse_ours_1.txt}\unskip,
			MAE=\protect\input{Experiments/GeocorePhantom/mae_ours_1.txt}
		}
	\end{subfigure}\hs
	\begin{subfigure}{\figwidth}
		\includegraphics[width=\textwidth,height=\hh]{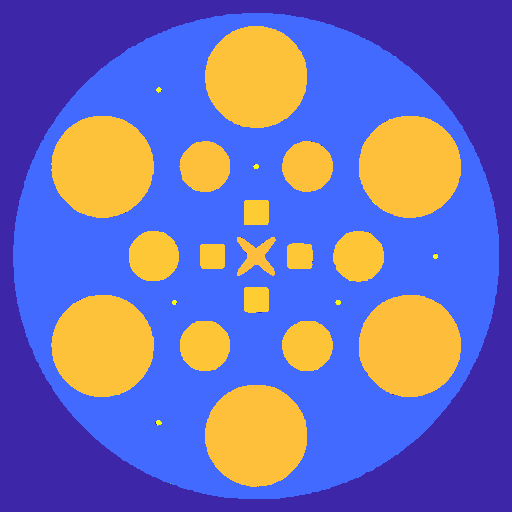}
		\subcaption*{MSSIM=\protect\input{Experiments/GeocorePhantom/ssim_ours_2.txt}
			\unskip,
			RMSE=\protect\input{Experiments/GeocorePhantom/rmse_ours_2.txt}\unskip,
			MAE=\protect\input{Experiments/GeocorePhantom/mae_ours_2.txt}
		}
	\end{subfigure}\hs
	\begin{subfigure}{\figwidth}
		\includegraphics[width=\textwidth,height=\hh]{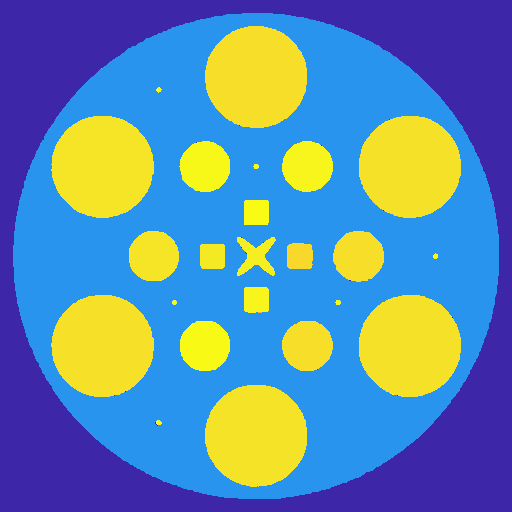}
		\subcaption*{MSSIM=\protect\input{Experiments/GeocorePhantom/ssim_ours_3.txt}
			\unskip,
			RMSE=\protect\input{Experiments/GeocorePhantom/rmse_ours_3.txt}\unskip,
			MAE=\protect\input{Experiments/GeocorePhantom/mae_ours_3.txt}
		}
	\end{subfigure}\vs
	\begin{subfigure}{\figwidth}
		\makebox[0pt][r]{\makebox[1.25em]{\raisebox{\rb}{\rotatebox[origin=c]{90}
					{\scriptsize \textbf{Potts S-CG}}}}}%
		\includegraphics[width=\textwidth,height=\hh]{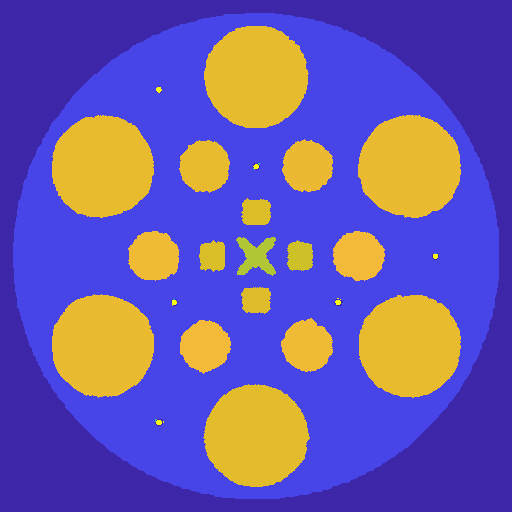}
		\subcaption*{MSSIM=\protect\input{Experiments/GeocorePhantom/ssim_sup_1.txt}
			\unskip,
			RMSE=\protect\input{Experiments/GeocorePhantom/rmse_sup_1.txt}\unskip,
			MAE=\protect\input{Experiments/GeocorePhantom/mae_sup_1.txt}
		}
	\end{subfigure}\hs
	\begin{subfigure}{\figwidth}
		\includegraphics[width=\textwidth,height=\hh]{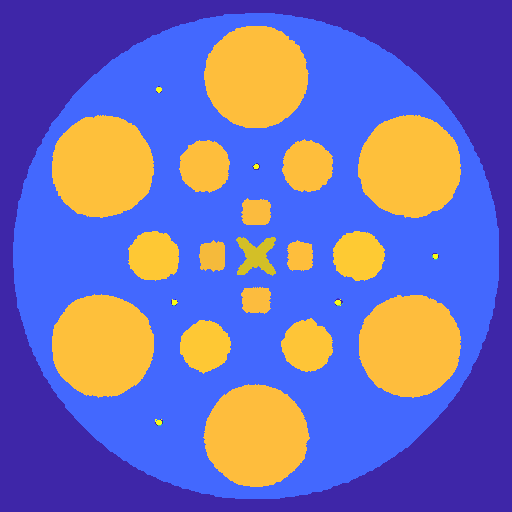}
		\subcaption*{MSSIM=\protect\input{Experiments/GeocorePhantom/ssim_sup_2.txt}
			\unskip,
			RMSE=\protect\input{Experiments/GeocorePhantom/rmse_sup_2.txt}\unskip,
			MAE=\protect\input{Experiments/GeocorePhantom/mae_sup_2.txt}
		}
	\end{subfigure}\hs
	\begin{subfigure}{\figwidth}
		\includegraphics[width=\textwidth,height=\hh]{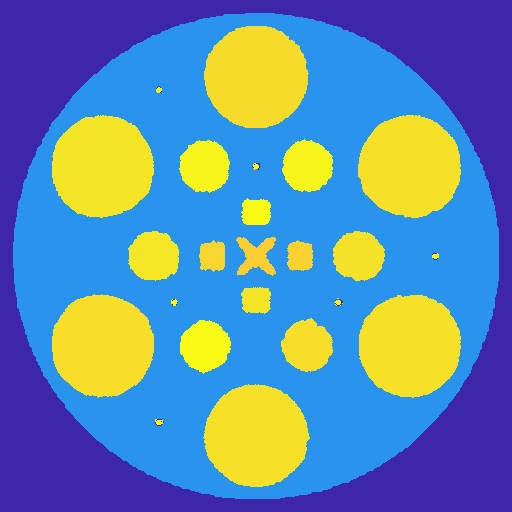}
		\subcaption*{MSSIM=\protect\input{Experiments/GeocorePhantom/ssim_sup_3.txt}
			\unskip,
			RMSE=\protect\input{Experiments/GeocorePhantom/rmse_sup_3.txt}\unskip,
			MAE=\protect\input{Experiments/GeocorePhantom/mae_sup_3.txt}
		}
	\end{subfigure}\vs
	\caption[Reconstruction results of the geocore phantom for three selected energy bins.]
	{Reconstruction results of the geocore phantom for three selected energy bins. 
		Channel-wise TV exhibits artifacts at the segment boundaries in the second channel
		and does not recover the inner segments in the third channel.
		TNV yields an improved reconstruction, but still shows (blueish) 
		artifacts at the segment boundaries. 
		The dTVp result has sharper boundaries than TV and TNV, 
		but the artifacts at the segment boundaries are still present.
		The Potts ADMM and the Potts S-CG results show less artifacts and the boundaries are sharp.
		Further, they achieve the highest MSSIM values
		for two out of the three depicted channels, which is representative for the full
		energy spectrum; see \Cref{fig:multiCT}.
		\label{fig:geocorePhantom}}
\end{figure}

\paragraph{Organic spheres phantom.}
\begin{figure}[ht]
	\centering
	\captionsetup[subfigure]{justification=centering}
	\def\figwidth{0.25\textwidth}
	\def\vs{\\[0.2em]}
	\def\hs{\hspace{1.25em}}
	\def\hh{\textwidth}
	\begin{subfigure}{\figwidth}
		\includegraphics[width=\textwidth,height=\hh]{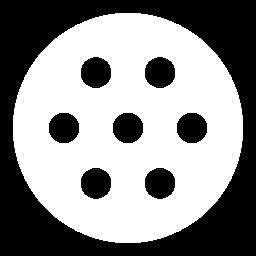}
		\subcaption{Muscle tissue}
	\end{subfigure}\hs
	\begin{subfigure}{\figwidth}
		\includegraphics[width=\textwidth,height=\hh]{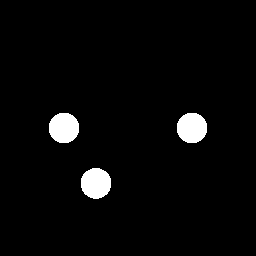}
		\subcaption{Fat tissue}
	\end{subfigure}\hs
	\begin{subfigure}{\figwidth}
		\includegraphics[width=\textwidth,height=\hh]{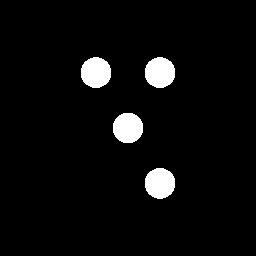}
		\subcaption{Bone tissue}
	\end{subfigure}
	\caption{\label{fig:OrganicSpheresPhantom}The organic spheres phantom consists of three materials.
		(a) Muscle tissue, (b) fat tissue and (c) bone tissue.
		The background is air. 
		The main challenge is to separate the regions with muscle tissue
		and those with fat tissue as fat and muscle have very similar LAC's throughout 
		the considered energy spectrum; see Figure \ref{fig:OrganicSpheresSpectrum}.}
	~\vs
	\def\figwidth{0.51\textwidth}
	\includegraphics[width=\figwidth]{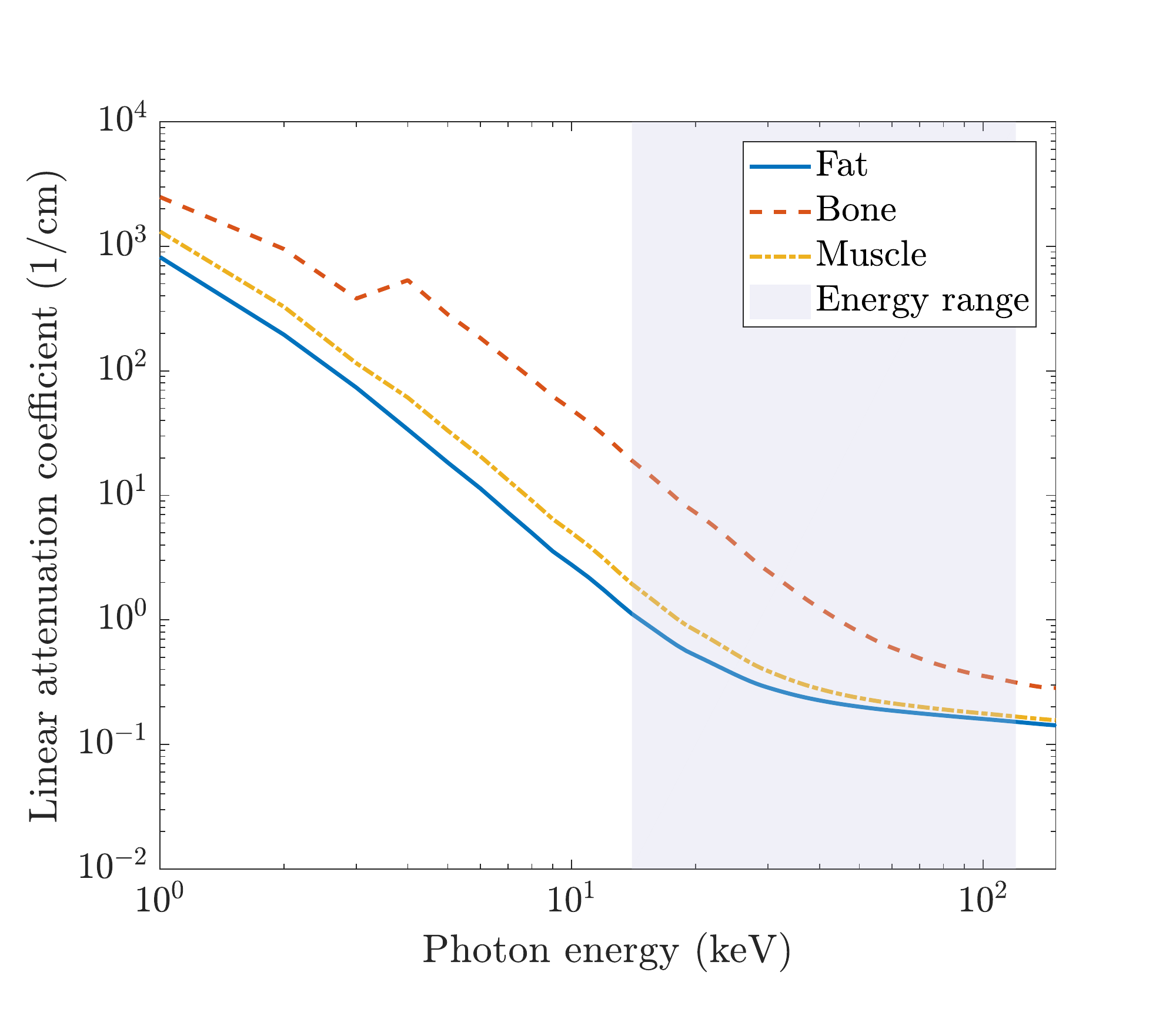}	
	\hspace{-1.25em}
	\includegraphics[width=\figwidth]{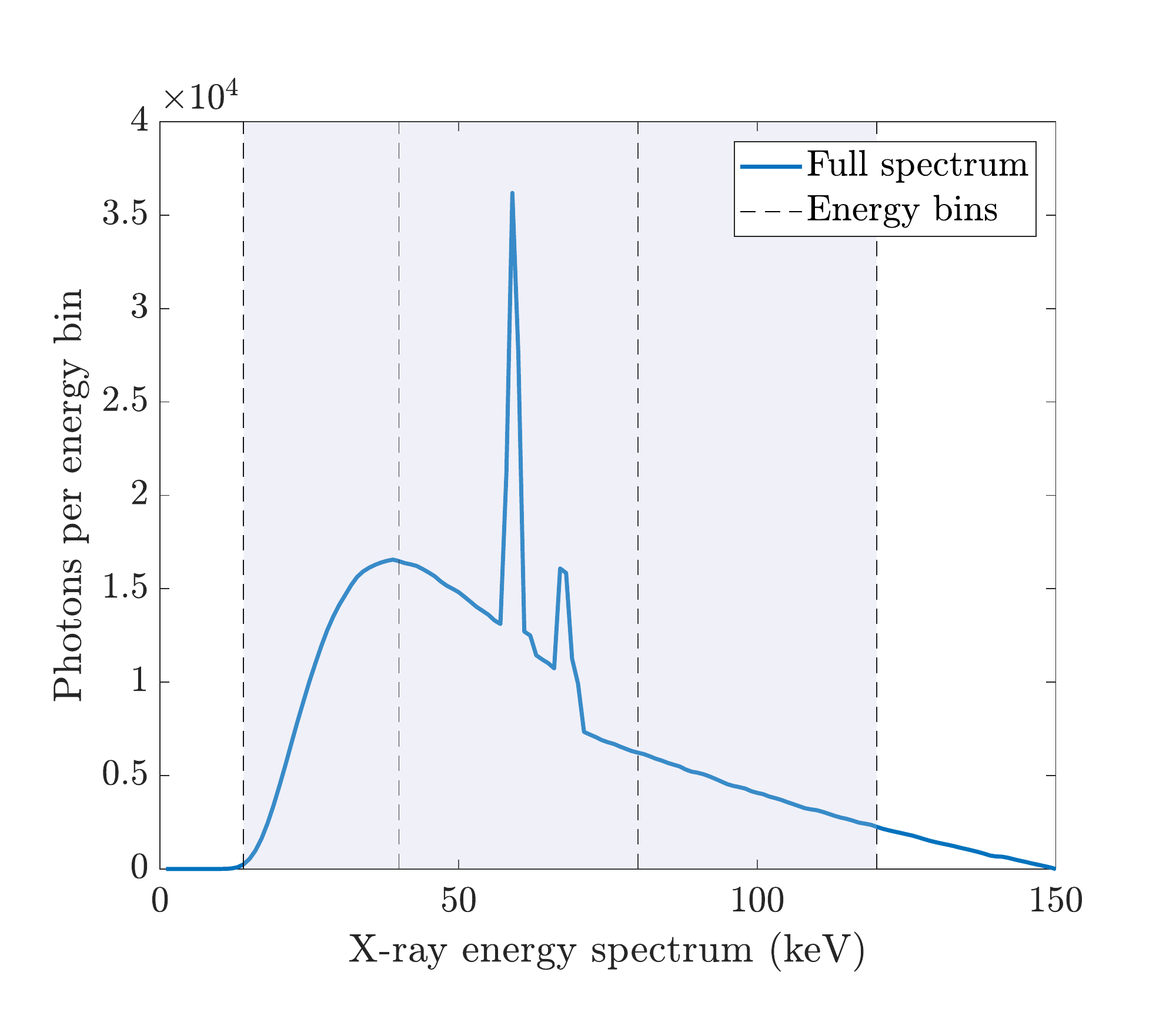}
	\caption{\label{fig:OrganicSpheresSpectrum}Modeling the measurements of the organic spheres
		phantom. \textit{Left}: the LAC's of fat, bone and muscle
		tissue. The highlighted area depicts the used energy range.
		The LAC curves of fat and muscle are close.
		\textit{Right}: the full X-ray spectrum and the used energy range  
		(15-120\,keV) for modeling the data of the organic spheres phantom.
		The spectrum is binned into three energy bins by the detectors and their boundaries 
		are depicted by the dashed vertical lines.
	}
\end{figure}
In the last paragraph, we
considered the geocore phantom, which consisted of inorganic materials, many energy channels and a
moderate number of projections were available.
Next, we conduct a quantitative comparison in terms of a phantom,
which consists of organic materials and only few energy bins and projection angles 
are available.
To this end, we consider the $256\times 256$ phantom of \Cref{fig:OrganicSpheresPhantom}, which consists of homogeneous regions
filled with muscle, fat and bone tissue.
The domain and detector width as well as the distances involved in the measurement
process are the same as for the geocore phantom.
The X-ray spectrum as well as the LAC's of muscle, fat and bone 
were simulated with the spektr software package
for the used tube potential (150 kVp) and
photon flux $I_0 = 10^5$.
In this context, we note that the measurements were simulated according to
\eqref{eq:DiscreteMeasurements} to avoid the inverse crime. 
Then, for the reconstruction, the linear model \eqref{eq:DiscreteMeasurementsSimplified} was used.
We show the LAC's of the involved materials and the X-ray spectrum in \Cref{fig:OrganicSpheresSpectrum}. 
The detectors bin the X-ray spectrum into three energy bins.
We note that the LAC's of fat and muscle tissue are very close, especially for
higher energies. 
Thus, the reconstruction approaches should aid 
the recovery of the third channel
(the one which holds highest energies among the three channels)
by employing the other two (less problematic) channels.
Again, we simulate the measurements with the Astra-toolbox and use a fan-beam scanning geometry
with 25\unskip~projection angles 
(obtained from a larger $512\times 512$ version of the phantom).
Hence, the reconstruction approaches have to deal with the additional challenge of 
highly undersampled measurements.

The individual model parameters were chosen empirically.
More precisely, we determined the model parameters
by choosing the optimal ones w.r.t.\,to the mean value of the
channel-wise MSSIM's.
In particular, for TV and TNV, the optimal parameter $\alpha$ was chosen among $\{10^{-4},1.5\cdot 10^{-4},\ldots,10^{-3}\}$.
As dTVp is of probabilistic nature,
we repeated the computation for each parameter $\alpha$ five times and chose the best result.
Here, we let $\alpha \in \{10^{-4},2\cdot 10^{-4},\ldots,2\cdot 10^{-3}\}$.
The jump penalty of Potts ADMM was chosen from $\gamma\in \{10^{-5},2\cdot10^{-5},\ldots, 2\cdot 10^{-4}\}$
and for Potts S-CG, the initial perturbation parameter was obtained from $\beta_0 \in \{1,\ldots,20\}$.

In \Cref{fig:OrganicPhantomReconstructions}, we show the reconstructions
together with the corresponding (optimal) MSSIM.
We observe that  Potts ADMM and Potts S-CG 
achieve the highest MSSIM values
for all three channels and the MSSIM values 
of Potts ADMM and Potts S-CG are rather close.
Furthermore, the fact that Potts ADMM and Potts S-CG achieve 
higher MSSIM values than the other methods is 
reflected by the reconstructions:
the result of channel-wise TV exhibits blurry boundaries and
the segments corresponding to the fat tissue are not recovered in the third channel.
TNV provides an improved reconstruction of the third channel. However,
the segment boundaries remain diffuse. 
The dTVp method produces sharp boundaries in all three channels, but
introduces some spurious artifacts near the boundaries.
Potts ADMM and Potts S-CG 
recover the segments and provide sharp boundaries throughout the channels.
Further, they show fewer artifacts.

\begin{figure}[hp]
	\centering
	\captionsetup[subfigure]{justification=centering}
	\def\figwidth{0.175\textwidth}
	\def\vs{\\[0.2em]}
	\def\hs{\hspace{0.75em}}
	\def\hh{\textwidth}
	\def\rb{30pt}
	\begin{subfigure}{\figwidth}
		\caption*{\textbf{15-40\,keV}\vspace{-0.5em}}
		\makebox[0pt][r]{\makebox[1.25em]{\raisebox{\rb}{\rotatebox[origin=c]{90}
					{\scriptsize \textbf{Ground truth}}}}}%
		{\includegraphics[width=\textwidth,height=\hh]{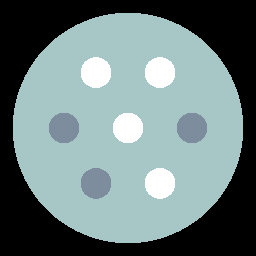}}
	\end{subfigure}\hs
	\begin{subfigure}{\figwidth}
		\caption*{\textbf{41-80\,keV}\vspace{-0.5em}}
		\includegraphics[width=\textwidth,height=\hh]{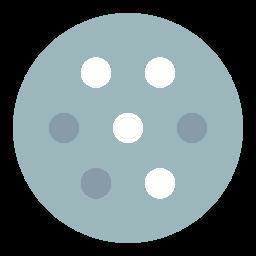}
	\end{subfigure}\hs
	\begin{subfigure}{\figwidth}
		\caption*{\textbf{81-120\,keV}\vspace{-0.5em}}
		\includegraphics[width=\textwidth,height=\hh]{Experiments/OrganicSpheres/GT_3.png}
	\end{subfigure}\\[1.25em]
	\begin{subfigure}{\figwidth}
		\makebox[0pt][r]{\makebox[1.25em]{\raisebox{\rb}{\rotatebox[origin=c]{90}
			{\scriptsize \textbf{TV}}}}}%
		\includegraphics[width=\textwidth,height=\hh]{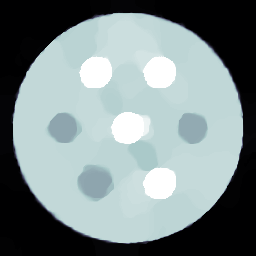}
		\subcaption*{MSSIM=\protect\input{Experiments/OrganicSpheres/ssim_TV_1.txt}
			\unskip,
			RMSE=\protect\input{Experiments/OrganicSpheres/rmse_TV_1.txt}\unskip,
			MAE=\protect\input{Experiments/OrganicSpheres/mae_TV_1.txt}
		}
	\end{subfigure}\hs
	\begin{subfigure}{\figwidth}
		\includegraphics[width=\textwidth,height=\hh]{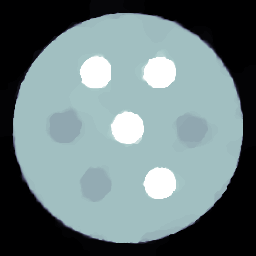}
		\subcaption*{MSSIM=\protect\input{Experiments/OrganicSpheres/ssim_TV_2.txt}
			\unskip,
			RMSE=\protect\input{Experiments/OrganicSpheres/rmse_TV_2.txt}\unskip,
			MAE=\protect\input{Experiments/OrganicSpheres/mae_TV_2.txt}
		}
	\end{subfigure}\hs
	\begin{subfigure}{\figwidth}
		\includegraphics[width=\textwidth,height=\hh]{Experiments/OrganicSpheres/result_TV_3.png}
		\subcaption*{MSSIM=\protect\input{Experiments/OrganicSpheres/ssim_TV_3.txt}
			\unskip,
			RMSE=\protect\input{Experiments/OrganicSpheres/rmse_TV_3.txt}\unskip,
			MAE=\protect\input{Experiments/OrganicSpheres/mae_TV_3.txt}
		}
	\end{subfigure}\vs
	\begin{subfigure}{\figwidth}
		\makebox[0pt][r]{\makebox[1.25em]{\raisebox{\rb}{\rotatebox[origin=c]{90}
			{\scriptsize \textbf{TNV}}}}}%
		\includegraphics[width=\textwidth,height=\hh]{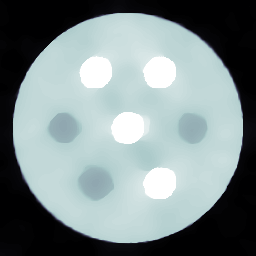}
		\subcaption*{MSSIM=\protect\input{Experiments/OrganicSpheres/ssim_TNV_1.txt}
			\unskip,
			RMSE=\protect\input{Experiments/OrganicSpheres/rmse_TNV_1.txt}\unskip,
			MAE=\protect\input{Experiments/OrganicSpheres/mae_TNV_1.txt}
		}
	\end{subfigure}\hs
	\begin{subfigure}{\figwidth}
		\includegraphics[width=\textwidth,height=\hh]{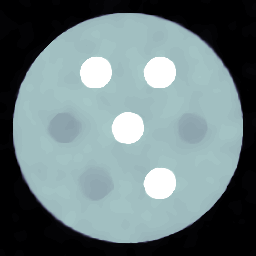}
		\subcaption*{MSSIM=\protect\input{Experiments/OrganicSpheres/ssim_TNV_2.txt}
			\unskip,
			RMSE=\protect\input{Experiments/OrganicSpheres/rmse_TNV_2.txt}\unskip,
			MAE=\protect\input{Experiments/OrganicSpheres/mae_TNV_2.txt}
		}
	\end{subfigure}\hs
	\begin{subfigure}{\figwidth}
		\includegraphics[width=\textwidth,height=\hh]{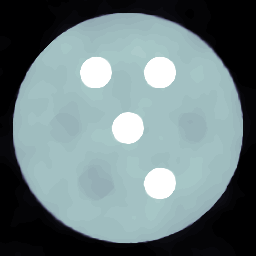}
		\subcaption*{MSSIM=\protect\input{Experiments/OrganicSpheres/ssim_TNV_3.txt}
			\unskip,
			RMSE=\protect\input{Experiments/OrganicSpheres/rmse_TNV_3.txt}\unskip,
			MAE=\protect\input{Experiments/OrganicSpheres/mae_TNV_3.txt}
		}
	\end{subfigure}\vs
	\begin{subfigure}{\figwidth}
		\makebox[0pt][r]{\makebox[1.25em]{\raisebox{\rb}{\rotatebox[origin=c]{90}
			{\scriptsize \textbf{dTVp}  }}}}%
		\includegraphics[width=\textwidth,height=\hh]{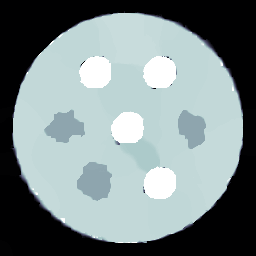}
		\subcaption*{MSSIM=\protect\input{Experiments/OrganicSpheres/ssim_dTVp_1.txt}
			\unskip,
			RMSE=\protect\input{Experiments/OrganicSpheres/rmse_dTVp_1.txt}\unskip,
			MAE=\protect\input{Experiments/OrganicSpheres/mae_dTVp_1.txt}
		}
	\end{subfigure}\hs
	\begin{subfigure}{\figwidth}
		\includegraphics[width=\textwidth,height=\hh]{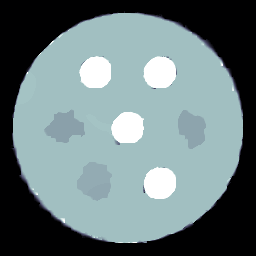}
		\subcaption*{MSSIM=\protect\input{Experiments/OrganicSpheres/ssim_dTVp_2.txt}
			\unskip,
			RMSE=\protect\input{Experiments/OrganicSpheres/rmse_dTVp_2.txt}\unskip,
			MAE=\protect\input{Experiments/OrganicSpheres/mae_dTVp_2.txt}
		}
	\end{subfigure}\hs
	\begin{subfigure}{\figwidth}
		\includegraphics[width=\textwidth,height=\hh]{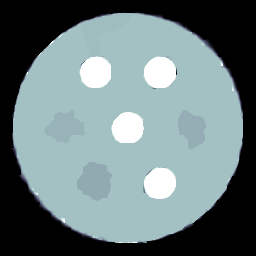}
		\subcaption*{MSSIM=\protect\input{Experiments/OrganicSpheres/ssim_dTVp_3.txt}
			\unskip,
			RMSE=\protect\input{Experiments/OrganicSpheres/rmse_dTVp_3.txt}\unskip,
			MAE=\protect\input{Experiments/OrganicSpheres/mae_dTVp_3.txt}
		}
	\end{subfigure}\vs
	\begin{subfigure}{\figwidth}
		\makebox[0pt][r]{\makebox[1.25em]{\raisebox{\rb}{\rotatebox[origin=c]{90}
			{\scriptsize\textbf{Potts ADMM}}}}}%
		\includegraphics[width=\textwidth,height=\hh]{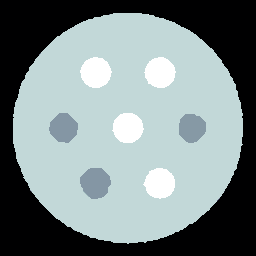}
		\subcaption*{MSSIM=\protect\input{Experiments/OrganicSpheres/ssim_ADMM_1.txt}
			\unskip,
			RMSE=\protect\input{Experiments/OrganicSpheres/rmse_ADMM_1.txt}\unskip,
			MAE=\protect\input{Experiments/OrganicSpheres/mae_ADMM_1.txt}
		}
	\end{subfigure}\hs
	\begin{subfigure}{\figwidth}
		\includegraphics[width=\textwidth,height=\hh]{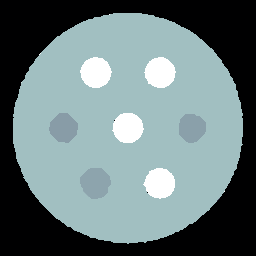}
		\subcaption*{MSSIM=\protect\input{Experiments/OrganicSpheres/ssim_ADMM_2.txt}
			\unskip,
			RMSE=\protect\input{Experiments/OrganicSpheres/rmse_ADMM_2.txt}\unskip,
			MAE=\protect\input{Experiments/OrganicSpheres/mae_ADMM_2.txt}
		}
	\end{subfigure}\hs
	\begin{subfigure}{\figwidth}
		\includegraphics[width=\textwidth,height=\hh]{Experiments/OrganicSpheres/result_ADMM_3.png}
		\subcaption*{MSSIM=\protect\input{Experiments/OrganicSpheres/ssim_ADMM_3.txt}
			\unskip,
			RMSE=\protect\input{Experiments/OrganicSpheres/rmse_ADMM_3.txt}\unskip,
			MAE=\protect\input{Experiments/OrganicSpheres/mae_ADMM_3.txt}
		}
	\end{subfigure}\vs
	\begin{subfigure}{\figwidth}
		\makebox[0pt][r]{\makebox[1.25em]{\raisebox{\rb}{\rotatebox[origin=c]{90}
			{\scriptsize \textbf{Potts S-CG}}}}}%
		\includegraphics[width=\textwidth,height=\hh]{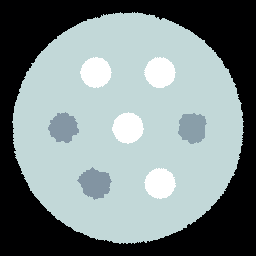}
		\subcaption*{MSSIM=\protect\input{Experiments/OrganicSpheres/ssim_super_1.txt}
			\unskip,
			RMSE=\protect\input{Experiments/OrganicSpheres/rmse_super_1.txt}\unskip,
			MAE=\protect\input{Experiments/OrganicSpheres/mae_super_1.txt}
		}
	\end{subfigure}\hs
	\begin{subfigure}{\figwidth}
		\includegraphics[width=\textwidth,height=\hh]{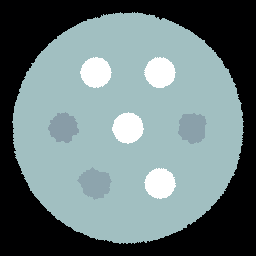}
		\subcaption*{MSSIM=\protect\input{Experiments/OrganicSpheres/ssim_super_2.txt}
			\unskip,
		RMSE=\protect\input{Experiments/OrganicSpheres/rmse_super_2.txt}\unskip,
		MAE=\protect\input{Experiments/OrganicSpheres/mae_super_2.txt}
	}
	\end{subfigure}\hs
	\begin{subfigure}{\figwidth}
		\includegraphics[width=\textwidth,height=\hh]{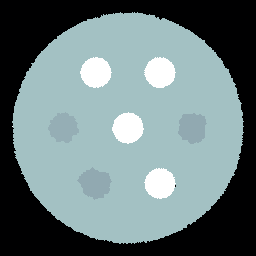}
		\subcaption*{MSSIM=\protect\input{Experiments/OrganicSpheres/ssim_super_3.txt}
			\unskip,
			RMSE=\protect\input{Experiments/OrganicSpheres/rmse_super_3.txt}\unskip,
			MAE=\protect\input{Experiments/OrganicSpheres/mae_super_3.txt}
		}
	\end{subfigure}\vs
	\caption{Reconstruction of the organic spheres phantom.
	The channel-wise TV result exhibit blurry boundaries and smooths out the fat segments in the third channel.
	TNV yields an improved reconstruction for the third channel. However,
	the boundaries are still blurry. The dTVp method produces sharp boundaries in all three channels, but
	introduces some spurious artifacts near the boundaries.
	Potts ADMM and Potts S-CG 
	recover the segments and provide sharp bound\-aries through\-out the channels.
	Further, they show fewer artifacts and
	achieve the highest MSSIM values
	for all three channels.\label{fig:OrganicPhantomReconstructions}}
\end{figure}

\section{Conclusion}
\label{sec:Conclusion}

We considered the reconstruction problem in multi-spectral CT, focusing 
on indirect measurements from energy-discriminating photon-counting detectors.
We proposed the multi-channel Potts prior
which promotes piecewise constant reconstructions.
We found that the multi-channel Potts prior provides a strong 
channel coupling in  that the jumps in the solution are enforced to be at the same spatial positions
across the channels.
This property is especially beneficial for  multi-channel images 
in multi-spectral CT as their channels have a strong structural correlation.
We employ the multi-channel Potts prior by minimizing the multi-channel Potts model.
To this end, we adapted the ADMM strategy proposed in \cite{storath2015joint} to the multi-channel reconstruction problem in multi-spectral CT.
Furthermore, we proposed new reconstruction approaches based on Potts superiorization of the
conjugate gradients method (CG). More precisely, the CG iterates 
were perturbed w.r.t.\,the (block-wise) Potts prior
towards more desirable solutions. 
We have shown that one obtains non-ascending directions w.r.t.\,the block-wise
Potts prior when taking steps towards its proximal mapping.
The corresponding superiorization approach, which perturbs the iterates
by adding these non-ascending directions, yielded
better results than the non-perturbed CG.
Furthermore, we provided theory  that ensures termination
	of the resulting algorithm.
In practice, we observed further improved results 
when we perturbed the iterates with the proximal mapping itself.
Based on this observation, we developed a new Potts-based superiorization approach, which we called Potts S-CG. Potts S-CG uses the proximal mapping as before and 
additionally lets the underlying PWLS problem evolve
in the course of the iterations, so that the final result
becomes genuinely piecewise constant. 

The energy minimization approach (Potts ADMM) and 
the superiorization approach (Potts S-CG) both produce solutions which are 
(multi-channel) Potts-regularized.
We identified Potts ADMM and Potts S-CG as suitable choices 
within their respective class of methods
by comparing them to a penalty method
and a method which Potts-superiorizes the Landweber iteration, respectively.
Despite the different abstract
interpretations, the iterations in Potts ADMM as well as
Potts S-CG involve a data step and a regularizing step. 
The latter decomposes into
univariate Potts problems (solved efficiently by dynamic programming).
A significant difference between Potts ADMM and Potts S-CG is that 
the data step of Potts ADMM corresponds to solving an $\ell_2$-regularized problem,
while the data step of Potts S-CG corresponds to a single 
CG step.

In our numerical experiments,
we applied Potts ADMM and Potts S-CG to simulated multi-spectral CT data and compared them to the
existing TV-based methods. Our Potts prior based methods produced sharper edges and mostly higher MSSIM values
than  TV-type methods. We attribute a large extend of these improvements to the channel-coupling promoted by the multi-channel Potts prior.

\section*{Acknowledgements}
We thank J. S. J\o rgensen (Technical University of Denmark) for simulation code and D. Kazantsev (Diamond Light Source, UK) for his advice on the provided Github code.
L. Kiefer and A. Weinmann were supported by the German Research Foundation (DFG) Grant WE5886/4-1. M. Storath was supported by DFG Grant STO1126/2-1.
\appendix
\setcounter{section}{1}
\section*{Appendix}
\subsection{Derivation of the ADMM subproblems \eqref{eq:Vstep}-\eqref{eqUstep}.}
\label{sec:derivationADMM}
	In each iteration, we minimize the Lagrangian $\mathcal{L}$ as given in \eqref{eq:AugmentedLagrangian}
w.r.t. $v$ and $u_1,\ldots,u_S$. 
To simplify the expressions for $\argmin_{v} \mathcal{L}$ and $\argmin_{u_s} \mathcal{L}$,
we will use repeatedly the fact that 
\begin{equation}\label{eq:transformSum}
	\sum_{i=1}^{n} x_i(p-t)^2 ) = 
	\bigg( \sum_{i=1}^n x_i \bigg)
	\bigg(  p-\frac{\sum_{i=1}^{N}t_ix_i}{\sum_{i=1}^{n}x_i} \bigg)^2 + K
\end{equation}
holds for $p,t_1,\ldots,t_n\in\R$ and $x_1,\ldots,x_n>0$ and a constant $K$ 
that does not depend on $p$; see, e.g., \cite{storath2015joint}.
After dropping the terms that do not depend on $v$, the subproblem w.r.t. $v$ reads
\begin{equation}\label{eq:AppSubproblemV1}
	\argmin_{v} \sum_{c=1}^{C}
	\mathcal{D}(Av_c,f_c) + \sum_{s=1}^{S} \tfrac{\rho}{2} \| v-u_s + \tfrac{\tau_s}{\rho}\|^2.
\end{equation}
Applying \eqref{eq:transformSum} to \eqref{eq:AppSubproblemV1} yields
exactly the subproblem in $v$ as formulated in \eqref{eq:Vstep}. 
The subproblem w.r.t. $u_s$ (again, after dropping terms that do not depend on $u_s$) 
is given by 
\begin{equation}\label{eq:AppSubproblemU1}
	\argmin_{u_s}\, \gamma\omega_s \| \nabla_{d_s} u_s\|_0+
	\tfrac{\rho}{2} \| v-u_s + \tfrac{\tau_s}{\rho}\|^2
	+\sum_{r=1}^{s-1} \| u_s-u_r - \tfrac{\lambda_{r,s}}{\mu}\|^2
	+\sum_{t=s+1}^{S} \| u_s-u_t + \tfrac{\lambda_{s,t}}{\mu}\|^2.
\end{equation}
After applying \eqref{eq:transformSum} to \eqref{eq:AppSubproblemU1}, we obtain
\begin{equation}\label{eq:AppSubproblemU2}
	\begin{split}
		\argmin_{u_s}\,&\gamma\omega_s \| \nabla_{d_s} u_s\|_0+
		\tfrac{\rho}{2} \| v-u_s + \tfrac{\tau_s}{\rho}\|^2\\
		&+ \tfrac{(s-1)\mu}{2} \Big\| u_s - \frac{\sum_{r=1}^{s-1}(u_r+\tfrac{\lambda_{r,s}}{\mu})}{s-1}\Big\|^2
		+ \tfrac{(S-s)\mu}{2} \Big\| u_s - \frac{\sum_{t=s+1}^{S}(u_t-\tfrac{\lambda_{s,t}}{\mu})}{S-s}\Big\|^2.
	\end{split}
\end{equation}
We use again \eqref{eq:transformSum}, which yields
\begin{equation}\label{eq:AppSubproblemU3}
	\argmin_{u_s}\gamma\omega_s \| \nabla_{d_s} u_s\|_0+
	\tfrac{\rho}{2} \| v-u_s + \tfrac{\tau_s}{\rho}\|^2 +
	\tfrac{(S-1)\mu}{2}
	\bigg\|
	u_s - \frac{\sum_{r=1}^{s-1}(u_r+\tfrac{\lambda_{r,s}}{\mu}) + 
		\sum_{t=s+1}^{S}(u_t-\tfrac{\lambda_{s,t}}{\mu})}{S-1}
	\bigg\|^2.
\end{equation}
A final application of \eqref{eq:transformSum} leads to 
\begin{equation}\label{eq:AppSubproblemU4}
	\argmin_{u^s}\gamma\omega_s \| \nabla_{d_s} u_s\|_0+
	\tfrac{\rho + (S-1)\mu}{2}
	\bigg\|
	u_s - \frac{\rho v+\tau_s +\sum_{r=1}^{s-1}(u_r+\tfrac{\lambda_{r,s}}{\mu}) + 
		\sum_{t=s+1}^{S}(u_t-\tfrac{\lambda_{s,t}}{\mu})}{S-1}
	\bigg\|^2
\end{equation}
and multiplying \eqref{eq:AppSubproblemU4} by $\tfrac{2}{\rho + (S-1)\mu}$
yields subproblem \eqref{eqUstep}.

\subsection{Conjugate gradient step for a generic least squares problem.}
In Algorithm \ref{alg:CGstep}, we provide the pseudocode 
of an iteration of the conjugate gradient method 
applied to the normal equations of a generic least squares problem 
$\| Bx -b\|^2$ (cf.\,\cite[Alg. 8]{zibetti2018super}).

\begin{algorithm}[H]
	\SetAlgoRefName{A.1}
	\def\vs{\\[0.3em]}
	\caption{\label{alg:CGstep}
		CG step for a least squares problem
	}
	\SetCommentSty{footnotesize}
	\PrintSemicolon
	\KwIn{System matrix $B$, current iterate $x$,
		current auxiliarly vectors $p,h$}
	\KwOut{Updated $x,p,h$
	}
	$r\leftarrow B^T(Bx-b)$
	
	$\alpha = \langle r,h\rangle / \langle p,h\rangle$
	
	$ p\leftarrow -r+\alpha p$
	
	$h \leftarrow B^TBp$
	
	$\kappa=-\langle r,p\rangle/\langle p,h\rangle$
	
	$x \leftarrow x + \kappa p $
\end{algorithm}

\subsection{Conjugate gradient step for the augmented PWLS problem \eqref{eq:weightedLSsplitIsotropicNormalEquations}}
In Algorithm \ref{alg:GCstepSuper}, we provide the pseudocode 
of a step of the CG method
applied to the normal equations corresponding to the augmented PWLS problem 
\eqref{eq:weightedLSsplitIsotropicNormalEquations}.

\begin{algorithm}[H]
	\SetAlgoRefName{A.2}
	\def\vs{\\[0.3em]}
	\caption{\label{alg:GCstepSuper}
		CG step for the augmented weighted least squares problem \eqref{eq:weightedLSsplitIsotropicNormalEquations}
	}
	\SetCommentSty{footnotesize}
	\PrintSemicolon
	\KwIn{Forward operator $A$, multispectral sinogram $f$, weights $W_c$, coupling parameter $\mu$, current iterates $u,p,h$}
	\KwOut{Updated iterates $u,p,h$
	}
	\BlankLine
	\For{$c=1,\ldots,C$}{
		$r_{s,c} \leftarrow A^TW_cAu_{s,c} - A^TW_cf_c + \sum_{t\neq s}\mu^2(u_{s,c}-u_{t,c})
		\quad\text{for all $s=1,\ldots,S$,}$
		
		$\alpha = \frac{\sum_s (r_{s,c})^T h_{s,c}}{\sum_s(p_{s,c})^T h_{s,c}},$
		
		$p_{s,c} \leftarrow -r_{s,c} + \alpha p_{s,c}
		\quad\text{for all $s=1,\ldots,S$,}$
		
		$h_{s,c} \leftarrow A^TW_cA p_{s,c} + \sum_{t\neq s} \mu^2\ (p_{s,c}-p_{t,c})
		\quad\text{for all $s=1,\ldots,S$,}$
		
		$\kappa = - \frac{\sum_{s} (r_{s,c})^Tp_{s,c}}{\sum_{s}(p_{s,c})^Th_{s,c}},$
		
		$u_{s,c} \leftarrow u_{s,c}+\kappa p_{s,c}
		\quad\text{for all $s=1,\ldots,S$}$
	}
\end{algorithm}

\subsection{Proofs} We here provide the proofs of \Cref{lemma:neighboringEqualityPotts}, \Cref{prop:nonascending} and \Cref{prop:WeightedLSBlockwiseEquality}.
\begin{proof}[Proof of \Cref{lemma:neighboringEqualityPotts}]
	We note that the statement follows from the analogous statement for the univariate Potts problem
	as the proximal mapping of the block-wise Potts prior 
	corresponds to row- and column-wise univariate Potts problems
	(as discussed below \eqref{eq:BlockwisePottsPriorAniso}).
	For the univariate Potts problem, the statement was proven in \cite[Lemma 4.2]{winkler2002smoothers}.
\end{proof}

\begin{proof}[Proof of \Cref{prop:nonascending}]
	The proof essentially follows  from \Cref{lemma:neighboringEqualityPotts}:
	starting from $u=(u_1,u_2)$, no additional jumps will be opened in $u + t\cdot v$ for any
	$t\geq 0$ as $v_1$ is constant on the (discrete) row intervals of constant value of 
	$u_1$ and $v_2$ on the (discrete) column intervals of constant value of $u_2$.
	Consequently, $F\big(u + t\cdot v\big) \leq F(u)$ for all $t\geq 0$ 
	which completes the proof.
\end{proof}

\begin{proof}[Proof of \Cref{prop:WeightedLSBlockwiseEquality}]
	(i) Let $u^\ast$ be a minimizer of \eqref{eq:weightedLSmultichannel} and define
	$(v_1,\ldots,v_S) = (u^\ast,\ldots,u^\ast)$. We show that $(v_1,\ldots,v_S)$ satisfies the normal equations \eqref{eq:weightedLSsplitIsotropicNormalEquations} which is sufficient for a minimizer of 
	\eqref{eq:weightedLSsplitIsotropic}. It holds by definition that $v_s = v_t$ for all
	$s,t$. As $u^\ast$ satisfies the normal equations of \eqref{eq:weightedLSmultichannel},
	we further obtain $A^TWAv_{s,c} = A^TWf_c$ for all channels $c=1,\ldots,C$. Together, 
	$(v_1,\ldots,v_S)$ satisfies \eqref{eq:weightedLSsplitIsotropicNormalEquations} for
	all channels $c=1,\ldots,C$.
	(ii) We show that the equality of the block variables $u_s^\ast$ 
	of a minimizer follows from optimality. To this end,
	let $(u_1^\ast,\ldots,u_S^\ast)$ be a minimizer of \eqref{eq:weightedLSsplitIsotropic}.
	Towards a contradiction, we assume that there are 
	$s\neq t$ such that $u_s^\ast \neq u_t^\ast$ which in particular means
	$\sum_{s=1}^S \sum_{t=s+1}^{S} \frac{1}{2} \| u_{s,c}^\ast - u^\ast_{t,c}\| > 0$ for at least one channel $c$.
	For $(v_1,\ldots,v_S):=(u^\ast,\ldots,u^\ast)$ as above we have that
	$\sum_{s=1}^{S}\frac{1}{2} \| W^{\frac{1}{2}}Av_c^\ast - W^{\frac{1}{2}}f_c\|^2$ is minimal
	for each $c$ in view of $u^\ast$ being a minimizer of \eqref{eq:weightedLSmultichannel}.
	Further, by definition the quadratic deviations between the block variables vanish, i.e.,
	$\sum_{s=1}^S \sum_{t=s+1}^{S} \frac{1}{2} \| v_{s,c}- v_{t,c}\|^2 = 0$ for all
	$c$.
	As the sum of these two terms corresponds to the objective 
	in \eqref{eq:weightedLSsplitIsotropic},
	$(v_1,\ldots,v_S)$ yields a lower objective value than $(u_1^\ast,\ldots,u_S^\ast)$ which
	is a contradiction, so $u_1^\ast = \ldots = u^\ast_S$.
	From the normal equations \eqref{eq:weightedLSsplitNormalEquationsIso} follows
	immediately  $A^TWAu^\ast_{1,c} = A^TWf_c$ for all $c$ which corresponds to the normal
	equations of \eqref{eq:weightedLSmultichannel}.
	(iii) The third assertion follows from (i) and (ii) together with the uniqueness of 
	least squares problems for full-rank matrices.
\end{proof}
\FloatBarrier
\section*{References}
{
\small
\bibliographystyle{plain}
\bibliography{MultispectralCT,SPLiterature}
}
\end{document}